\documentclass{aic}

\aicAUTHORdetails{title = {Homomorphism Tensors and Linear Equations}, author = {Martin Grohe, Gaurav Rattan, and Tim Seppelt},
plaintextauthor = {Martin Grohe, Gaurav Rattan, and Tim Seppelt},
plaintexttitle = {Homomorphism Tensors and Linear Equations}, runningtitle = {Homomorphism Tensors and Linear Equations}, 
keywords = {graph homomorphisms, homomorphism indistinguishability, labelled graphs, treewidth, pathwidth, treedepth, linear equations, Sherali--Adams relaxation, Specht--Wiegmann Theorem, Weisfeiler--Leman algorithm},
}   

\aicEDITORdetails{year={2025},
	number={4},
	received={13 February 2024},
        published={2 April 2025},  doi={10.19086/aic.2025.4},      }

\usepackage{amsmath, amsfonts, amsthm, amssymb, mathtools, stmaryrd}
\usepackage[capitalise,noabbrev]{cleveref}
\usepackage{csquotes}
\usepackage{relsize}

\usepackage{thm-restate}
\usepackage{subcaption}

\usepackage{newtxmath}

\usepackage[inline]{enumitem}
\usepackage{booktabs}

\usepackage{tikz}

\theoremstyle{definition}
\newtheorem{definition}{Definition}[section]

\newtheorem{example}[definition]{Example}
\newtheorem{remark}[definition]{Remark}

\theoremstyle{plain}
\newtheorem{lemma}[definition]{Lemma}
\newtheorem{corollary}[definition]{Corollary}
\newtheorem{theorem}[definition]{Theorem}

\newtheorem{fact}[definition]{Fact}
\Crefname{fact}{Fact}{Facts}

\theoremstyle{remark}
\newtheorem{claim}{Claim}
\Crefname{claim}{Claim}{Claims}
\newenvironment{claimproof}[1][Proof of Claim]{\begin{proof}[#1] }{ \end{proof}}
\setlist[enumerate, 1]{font=\upshape, noitemsep, nolistsep}
\setlist[enumerate, 2]{font=\upshape, noitemsep, nolistsep}
\setlist[itemize, 1]{noitemsep, nolistsep,font=\upshape}
\setlist[itemize, 2]{noitemsep, nolistsep,font=\upshape}

\newcommand{\abs}[1]{\left| #1 \right|}

\renewcommand\phi\varphi
\renewcommand\epsilon\varepsilon

\DeclareMathOperator{\soe}{soe}
\DeclareMathOperator{\spn}{span}
\DeclareMathOperator{\tr}{tr}
\DeclareMathOperator{\id}{id}

\DeclareMathOperator{\gca}{ca}

\usepackage{orcidlink}

\renewcommand{\vec}{\boldsymbol}

\usetikzlibrary{arrows.meta}

\newcommand{\matvec}{\mathsf}

\newcommand{\NN}{\mathbb{N}}

\newcommand{\calB}{\mathcal{B}}
\newcommand{\allones}{\boldsymbol{1}}

\DeclareMathOperator{\End}{End}

\usepackage{xspace}

\newsavebox{\fminibox}
\newlength{\fminilength}
\newenvironment{fminipage}[1][\linewidth]{
  \setlength{\fminilength}{#1-2\fboxsep-2\fboxrule}\begin{lrbox}{\fminibox}\begin{minipage}{\fminilength}}{
    \end{minipage}\end{lrbox}\noindent\fbox{\usebox{\fminibox}}
}

\newcommand{\Fiso}[1][]{\mathsf{F}_{\textnormal{iso}}^{#1}}
\newcommand{\Liso}[1][]{\mathsf{L}^{#1}_{\textnormal{iso}}}
\newcommand{\PW}{\mathsf{PW}}

\newcommand{\norm}[1]{\left\lVert #1 \right\rVert}

\begin{document}

	\begin{frontmatter}[classification=text]

\title{Homomorphism Tensors and\\ Linear Equations\titlefootnote{This work was presented at the \textit{49th International Colloquium on Automata, Languages, and Programming} (\textsmaller{ICALP}~2022)~\cite{grohe_homomorphism_icalp}. Views and opinions expressed are however those of the author(s) only and do not necessarily reflect those of the European Union or the European Research Council Executive Agency. Neither the European Union nor the granting authority can be held responsible for them.}} 

\author[magr]{Martin Grohe\thanks{Supported by the European Union (\textsmaller{ERC}, SymSim, 101054974).}}
		\author[gara]{Gaurav Rattan\thanks{Supported by the European Research Council (\textsmaller{ERC}) under the European Union's Horizon 2020 research and innovation programme (EngageS: grant agreement No. 820148) and by the German Research Foundation \textsmaller{DFG} (Research Grants Program, \textsmaller{RA} 3242/1-1 via project number 411032549).}}
		\author[tise]{Tim Seppelt\thanks{Supported by the German Research Foundation (\textsmaller{DFG}) within Research Training Group 2236/2 (\textsmaller{UnRAVeL}) and by the European Union (\textsmaller{ERC}, SymSim, 101054974, CountHom, 101077083).}}
		
\begin{abstract}
			Lovász~(1967) showed that two graphs $G$ and $H$ are isomorphic if and only if
			they are \emph{homomorphism indistinguishable} over the class of all graphs, i.e.\@ for every graph $F$, the number of homomorphisms from~$F$ to~$G$ equals the number of homomorphisms from~$F$ to~$H$. Recently, homomorphism indistinguishability over restricted classes of graphs such as bounded treewidth, bounded treedepth and planar graphs, has emerged as a surprisingly powerful framework for capturing diverse equivalence relations on graphs arising from logical equivalence and algebraic equation systems. 
			
			In this paper, we provide a unified algebraic framework for such results by examining the linear-algebraic structure of tensors counting homomorphisms from labelled graphs. The existence of certain linear transformations between such homomorphism tensor subspaces can be interpreted both as homomorphism indistinguishability over a graph class and as feasibility of an equational system. Following this framework, we obtain characterisations of homomorphism indistinguishability over several natural graph classes, namely trees of bounded degree, graphs of bounded pathwidth (answering a question of Dell et~al.\@ (2018)), and graphs of bounded treedepth.
		\end{abstract}
	\end{frontmatter}

\section{Introduction}

Representations in terms of homomorphism counts provide a surprisingly rich view on graphs and their properties. Homomorphism counts have direct connections to logic~\cite{dvorak_recognizing_2010,Grohe20,lovasz_operations_1967}, category theory~\cite{dawar_lovasz_2021,MontacuteS21}, the graph isomorphism problem~\cite{dell_lovasz_2018,dvorak_recognizing_2010,lovasz_operations_1967}, algebraic characterisations of graphs~\cite{dell_lovasz_2018}, and quantum groups~\cite{mancinska_quantum_2020}. Counting subgraph patterns in graphs has a wide range of applications, for example in graph kernels (see \cite{KriegeJM20}) and motif counting (see~\cite{AlonDHHS08,milo2002network}). Homomorphism counts can be used as a flexible basis for counting all kinds of substructures \cite{CurticapeanDM17}, and their complexity has been studied in great detail (e.g.~\cite{Bulatov13,BulatovZ20,CurticapeanDM17,RothW20}). It has been argued in \cite{Grohe-x2vec} that homomorphism counts are well-suited as a theoretical foundation for analysing graph embeddings and machine learning techniques on graphs, both indirectly through their connection with graph neural networks via the Weisfeiler--Leman algorithm \cite{dvorak_recognizing_2010,MorrisRFHLRG19,XuHLJ19} and directly as features for machine learning on graphs. The latter has also been confirmed experimentally \cite{BeaujeanSY21,Kuehner21,NguyenM20}.

The starting point of the theory is an old result due to Lov\'asz~\cite{lovasz_operations_1967}: two
graphs $G$, $H$ are isomorphic if and only if for every graph $F$, the
number $\hom(F,G)$ of homomorphisms from $F$ to $G$ equals
$\hom(F,H)$. For a class $\mathcal F$ of graphs, we say that $G$ and
$H$ are \emph{homomorphism indistinguishable} over $\mathcal F$ if and
only if $\hom(F,G)=\hom(F,H)$ for all $F\in\mathcal F$. A beautiful picture that has only emerged in the last few years shows that homomorphism indistinguishability over natural graph classes, such as paths, trees, or planar graphs, characterises a variety of natural equivalence relations on graphs.

Broadly speaking, there are two types of such results, the first relating homomorphism indistinguishability to logical equivalence, and the second giving algebraic characterisations of homomorphism equivalence derived from systems of linear (in)equalities for graph isomorphism. Examples of logical characterisations of homomorphism equivalence are the characterisation of homomorphism indistinguishability over graphs of treewidth at most~$k$ in terms of the $(k+1)$-variable fragment of first-order logic with counting \cite{dvorak_recognizing_2010} and the characterisation of homomorphism indistinguishability over graphs of treedepth at most~$k$ in terms of the quantifier-rank-$k$ fragment of first-order logic with counting \cite{Grohe20}. Results of this type have also been described in a general category-theoretic framework \cite{dawar_lovasz_2021, MontacuteS21}. Examples of equational characterisations are the characterisation of homomorphism indistinguishability over trees in terms of fractional isomorphism \cite{dell_lovasz_2018,dvorak_recognizing_2010,tinhofer_note_1991}, which may be viewed as the \textsmaller{LP} relaxation of a natural \textsmaller{ILP} for graph isomorphism, and a generalisation to homomorphism indistinguishability over graph of bounded treewidth in terms of the Sherali--Adams hierarchy over that \textsmaller{ILP} \cite{atserias_sherali-adams_2013,dell_lovasz_2018,grohe_pebble_2015,dvorak_recognizing_2010,mal14}. Further examples include a characterisation of homomorphism indistinguishability over paths in terms of the same system of equalities by dropping the non-negativity constraints of fractional isomorphism \cite{dell_lovasz_2018}, and a characterisation of homomorphism indistinguishability over planar graphs in terms of quantum isomorphism~\cite{mancinska_quantum_2020}. Remarkably, quantum isomorphism is derived from interpreting the same system of linear equations over C*-algebras \cite{AtseriasMRS19}.

\subsection{Results}
Two questions that remained open in \cite{dell_lovasz_2018} are (1) whether the equational characterisation of homomorphism indistinguishability over paths can be generalised to graphs of bounded pathwidth in a similar way as the characterisation of homomorphism indistinguishability over trees can be generalised to graphs of bounded treewidth, and (2) whether homomorphism indistinguishability over graphs of bounded degree suffices to characterise graphs up to isomorphism. In this paper, we answer the first question affirmatively. 

\begin{theorem}\label{thm:pw-informal}
For every $k \geq 1$, the following are equivalent for simple graphs $G$ and $H$:
\begin{enumerate}
\item $G$ and $H$ are homomorphism indistinguishable over the graphs of pathwidth at most~$k$,
\item the level-$(k+1)$ relaxation $\Liso[k+1](G, H)$ of the standard \textsmaller{ILP} for graph isomorphism has a rational solution.
\end{enumerate}
\end{theorem}

The detailed description of the system $\Liso[k+1](G, H)$ is provided in \cref{sec:lkiso}.
In fact, we also devise an alternative system of linear equations $\PW^{k+1}(G,H)$ characterising homomorphism indistinguishability over graphs of pathwidth at most~$k$. The definition of this system turns out to be very natural from the perspective of homomorphism counting, and as we explain later, it forms a fruitful instantiation of a more general representation-theoretic framework for homomorphism indistinguishability. 

Moreover, we obtain an equational characterisation of homomorphism indistinguishability over graphs of bounded treedepth. The resulting system $\mathsf{TD}^k(G, H)$ is very similar to $\Liso[k](G, H)$ and $\PW^k(G,H)$, except that variables are indexed by (ordered) $k$-tuples of variables rather than sets of at most $k$ variables, which reflects the order induced by the recursive definition of treedepth. 

\begin{restatable}{theorem}{treedepth}\label{thm:treedepth}
For every $k \geq 1$, the following are equivalent for simple graphs $G$ and $H$:
\begin{enumerate}
\item $G$ and $H$ are homomorphism indistinguishable over the graphs of treedepth at most~$k$,\label{thm:treedepth1} 
\item the linear system of equations $\mathsf{TD}^k(G, H)$ has a non-negative rational solution,\label{thm:treedepth2}
\item the linear system of equations  $\mathsf{TD}^k(G, H)$ has a rational solution.\label{thm:treedepth3}
\end{enumerate}
\end{restatable}

Along with \cite{Grohe20}, the above theorem implies that the logical equivalence of two graphs~$G$ and~$H$ over the quantifier-rank-$k$ fragment of first-order logic with counting can be characterised by the feasibility of the system $\mathsf{TD}^k(G, H)$ of linear equations.

We cannot answer the second open question from \cite{dell_lovasz_2018}, but we prove a partial negative result: homomorphism indistinguishability over trees of bounded degree is strictly weaker than homomorphism indistinguishability over all trees.
\begin{restatable}{theorem}{degreeGH}
\label{thm:degreeGH}
For every integer $d \geq 1$, there exist simple graphs $G$ and $H$
such that $G$ and $H$ are homomorphism indistinguishable over trees of degree at most $d$, but $G$ and $H$ are not homomorphism indistinguishable over the class of all trees. 
\end{restatable}

In conjunction with \cite{dell_lovasz_2018}, the above theorem yields the following corollary: counting homomorphisms from trees of bounded degree is strictly less powerful than the classical Colour Refinement algorithm \cite{GroheKMS17}, in terms of their ability to distinguish non-isomorphic graphs.

To prove these results, we develop a general theory that enables us to derive some of the  existing results as well as the new results in a unified algebraic framework exploiting a duality between algebraic varieties of ``tensor maps'' derived from homomorphism counts over families of rooted graphs and equationally defined equivalence relations, which are based on transformations of graphs in terms of unitary/orthogonal or, more often, pseudo-stochastic or doubly stochastic matrices. (We call a matrix over the complex numbers \emph{pseudo-stochastic} if its row and column sums are all $1$, and we call it \emph{doubly stochastic} if it is pseudo-stochastic and all its entries are non-negative reals.) 
The foundations of this theory have been laid in \cite{dell_lovasz_2018,dvorak_recognizing_2010} and, mainly,~\cite{mancinska_quantum_2020}. 
Some ideas can also be traced back to the work on homomorphism functions and connection matrices \cite{frelovsch07,lovasz_large_2012,LovaszS08,Schrijver09}, and a similar duality, called Galois connection there, that is underlying the algebraic theory of constraint satisfaction problems \cite{buljeakro05,bul17,thaziv16,zhu17}.

\subsection{Techniques}
To explain our core new ideas, let us start from a simple and well-known result: 
two symmetric real matrices $A$, $B$ are co-spectral if and only if for every $k\ge 1$ the matrices $A^k$ and $B^k$ have the same trace. 
If $A$, $B$ are the adjacency matrices of two graphs $G$, $H$, the latter can be phrased graph theoretically as: 
for every $k$, $G$ and $H$ have the same number of closed walks of length~$k$, or equivalently, the numbers of homomorphisms from a cycle $C_k$ of length~$k$ to~$G$ and to~$H$ are the same. 
Thus, $G$ and $H$ are homomorphism indistinguishable over the class of all cycles if and only if they are co-spectral. 
Note next that the graphs, or their adjacency matrices $A$, $B$, are co-spectral if and only if there is an orthogonal matrix $U$ such that $UA=BU$. 
Now, in \cite{dell_lovasz_2018} it was proved that $G$, $H$ are homomorphism indistinguishable over the class of all paths if and only if there is a pseudo-stochastic matrix $X$ such that $XA=BX$, and they are homomorphism indistinguishable over the class of all trees if and only if there is a doubly stochastic matrix $X$ such that $XA=BX$. 
From an algebraic perspective, the transition from an orthogonal matrix in the cycle result to a pseudo-stochastic in the path result is puzzling: 
where orthogonal matrices are very natural, pseudo-stochastic matrices are much less so from an algebraic point of view. 
Moving on to the tree result, we suddenly add non-negativity constraints---where do they come from? 
Our theory presented in \cref{sec:graphclasses} provides a uniform and very transparent explanation for the three results. 
It also allows us to analyse homomorphism indistinguishability over $d$-ary trees, for every $d\ge 1$, and to prove that it yields a strict hierarchy of increasingly finer equivalence relations.

Now suppose we want to extend these results to edge coloured graphs. Each edge-coloured graph corresponds to a family of matrices, one for each colour. Theorems due to Specht~\cite{specht_zur_1940} and Wiegmann~\cite{wiegmann_necessary_1961} characterise families of matrices that are simultaneously equivalent with respect to a unitary transformation. Interpreted over coloured graphs, the criterion provided by these theorems can be interpreted as homomorphism indistinguishability over coloured cycles. 
One of our main technical contributions (\cref{thm:soe,thm:soe-pos}) are variants of these theorems that establishes a correspondence between simultaneous equivalence with respect to pseudo-stochastic (doubly stochastic) transformations and homomorphism indistinguishability over coloured paths (trees). 
The proof is based on basic representation theory, in particular the character theory of semisimple algebras.

Interpreting graphs of bounded pathwidth in a ``graph-grammar style'' over coloured paths using graphs of bounded size as building blocks, we give an equational characterisation of homomorphism indistinguishability over graphs of pathwidth at most $k$. After further manipulations, we even obtain a characterisation in terms of a system of equations that are derived by lifting the basic equations for paths in a Sherali--Adams style. (The basic idea of these lifted equations goes back to \cite{atserias_sherali-adams_2013}.) This answers the open question from \cite{dell_lovasz_2018} stated above. In the same way, we can lift the characterisations of homomorphism indistinguishability over trees to graphs of treewidth $k$, and we can also establish a characterisation of homomorphism indistinguishability over graphs of ``cyclewidth''~$k$, providing a uniform explanation for all these results. 
Finally, we combine these techniques to prove a characterisation of homomorphism indistinguishability over graphs of treedepth~$k$ in terms of a novel system of linear equations.

\section{Preliminaries}

We write $\mathbb{N} = \{0,1,2,\dots\}$ for the set of non-negative integers.
All graphs in this article are finite, undirected, and without multiple edges.
A graph is \emph{simple} if it does not contain loops.

\subsection{Linear Algebra}
\label{sec:linalg}

The vector spaces considered in this article are over the rationals $\mathbb{Q}$, the reals~$\mathbb{R}$, or complex numbers~$\mathbb{C}$ and finite-dimensional. 
Thus, they carry an inner-product denoted by $\left< \cdot, \cdot \right>$.
We write $\mathbb{K} \in \{\mathbb{Q}, \mathbb{R}, \mathbb{C}\}$ as place holders for any of the respective fields.

For a vector space $V$ of dimension $n$, write $\End(V)$ for the vector space of \emph{endomorphisms} of $V$, that is the set of linear maps $V \to V$.
Let $\id_V \in \End(V)$ denote the identity map.
By standard linear algebra, $\End(V)$ can be identified with $\mathbb{K}^{n \times n}$. Any $A \in \End(V)$ corresponds to a matrix $(a_{ij}) \in \mathbb{K}^{n \times n}$ and $\tr A = \sum_{i = 1}^n a_{ii}$ where $\tr A$ denotes the \emph{trace} of the endomorphism~$A$. Note that the trace is a property of an endomorphism as it does not depend on the concrete basis chosen.

Contrarily, the \emph{sum-of-entries} of a matrix $A = (a_{ij}) \in \mathbb{K}^{n \times n}$ denoted by $\soe (a_{ij}) = \sum_{ij} a_{ij}$ is not a property of the endomorphism encoded by $(a_{ij})$ but of the matrix itself. For example, $\soe(B^{-1}AB ) \neq \soe(A)$ for $A = \left( \begin{smallmatrix} 3&0\\0&5 \end{smallmatrix} \right)$ and $B = \left( \begin{smallmatrix} 1&1\\0&1 \end{smallmatrix} \right)$.
Nevertheless, the vector spaces considered in this article always have a \enquote{natural} basis since they arise as subspaces of $\mathbb{K}^I$ for some finite index set $I$.

Of particular interest are maps between vector spaces.
Let $V$ and $W$ be $\mathbb{K}$-vector spaces.
A map $U\colon V \to W$ is \emph{unitary} if $U^*U = \id_V$ and $UU^* = \id_W$ for $U^*\colon W \to V$ the \emph{adjoint} of $U$, cf.\@ \cite[p.\@~185]{lang_linear_1987}.
For a matrix $A = (a_{ij}) \in \mathbb{K}^{n \times n}$, $A^*$ is the conjugate transpose of $A$ when the field is $\mathbb{C}$ and the transpose of $A$ when it is $\mathbb{R}$ or $\mathbb{Q}$. In the latter case, we write $A^T$ for $A^*$ and call a unitary matrix \emph{orthogonal}.
Note that a map $U\colon V \to W$ is unitary if and only if  
$\left<Ux,Uy\right> = \left<x,y\right>$  for all $x,y \in V$, i.e.\@ $U$ preserves the inner-product.

Let $I$ and $J$ be finite index sets. Fix vector spaces $V \leq \mathbb{K}^I$ and $W \leq \mathbb{K}^J$ such that the all-ones vectors $\boldsymbol{1}_I \in V$ and $\boldsymbol{1}_J \in W$. Then a map $X\colon V \to W$ is \emph{pseudo-stochastic} if $X\boldsymbol{1}_I = \boldsymbol{1}_J$ and $X^* \boldsymbol{1}_J = \boldsymbol{1}_I$ for $X^*$ the adjoint of $X$.

The Gram--Schmidt orthogonalisation procedure is a well-known method in linear algebra~\cite[Theorem~2.1]{lang_linear_1987}.
For the purpose of this article, a variant of it is required to construct unitary linear transformations between vector spaces spanned by possibly infinite sequences of vectors.

\begin{lemma} \label{lemma:gs}
	Let $V$ and $W$ be finite-dimensional inner-product spaces over a field $\mathbb{K}$.
	Let $\left<-,-\right>_V$ and $\left<-,-\right>_W$ denote their respective inner-products.
	Let $I$ be a possibly infinite set.
	Let $(v_i)_{i \in I}$ and $(w_i)_{i \in I}$ be two sequences of vectors such that $v_i \in V$ and $w_i \in W$ for all $i \in I$.
	Suppose that 
	\begin{enumerate}
		\item the $v_i$ for $i \in I$ span $V$, the $w_i$ for $i \in I$ span $W$, and
		\item $\langle v_i, v_j \rangle_V = \langle w_i, w_j \rangle_W$ for all $i, j \in I$.
	\end{enumerate}
	Then there exists a unitary linear map $\Phi \colon V \to W$ such that $\Phi(v_i) = w_i$ for all $i \in I$.
\end{lemma}
\begin{proof}
	Since the $v_i$ for $i \in I$ span $V$, there exists a finite set $I' \subseteq I$ such that the $v_i$ for $i \in I'$ are a basis for $V$.
	We define $\Phi$ on this basis by $v_i \mapsto w_i$ for $i \in I'$.

	The map $\Phi$ is such that $\Phi(v_k) = w_k$ for all $k \in I$. Indeed, let $k \in I$ be arbitrary.
	Then there exist $\alpha_i \in \mathbb{K}$, $i \in I'$, such that $v_k = \sum_{i \in I'} \alpha_i v_i$. Then for arbitrary $j \in I$,
	\[
	\left< w_j, \Phi(v_k) \right>_W
	= \sum_{i \in I'} \alpha_i \left< w_j,\Phi(v_i) \right>_W
	= \sum_{i \in I'} \alpha_i \left< w_j,w_i \right>_W
= \left< v_j, v_k \right>_V
	= \left< w_j, w_k \right>_W.
	\] 
	Thus, $\Phi(v_k) = w_k$.
	It follows that $\Phi$ is unitary. 
	Indeed, 
	$\left< \Phi v_k, \Phi v_j \right>_W
	= \left<w_k,  w_j\right>_W
	= \left<v_k, v_j\right>_V$ for all $k, j \in I$. 
\end{proof}

A recurring theme in this article is the study of $\mathbb{K}$-vector spaces over a finite set $I$. On such a space, the \emph{Schur product} of two vectors $v, w \in \mathbb{K}^I$ is defined via $(v \odot w)(x) \coloneqq v(x) w(x)$ for all $x \in I$. Write furthermore $v^{\odot i}$, $i \in \NN$, for the $i$-fold Schur product of the vector $v \in \mathbb{K}^I$ with itself. Conventionally, $v^{\odot 0} \coloneqq \allones_I$ where $\allones_I$ denotes the all-ones vector on $I$.
The following fact is classical:

\begin{fact}[{Vandermonde Determinant~\cite[p.\@~155]{lang_linear_1987}}] \label{fact:vd}
Let $I$ be a finite set of size $n$.
If $v \in \mathbb{K}^I$ has distinct entries then $\allones, v, v^{\odot 2}, \dots, v^{\odot (n - 1)}$ is a basis of $\mathbb{K}^I$.
\end{fact}

\Cref{fact:vd} will be applied in the following form:

\begin{lemma}[Vandermonde Interpolation] \label{lemma:vd}
Fix a finite set $I$.
Let $V \leq \mathbb{K}^I$ be a vector space closed under Schur product that contains $\boldsymbol{1}$.
Define an equivalence relation $\sim$ on $I$ by $x \sim x'$ iff $v(x) = v(x')$ for all $v \in V$. Then the indicator vectors $\boldsymbol{1}_{C_1}, \dots, \boldsymbol{1}_{C_r}$ of the equivalence classes form a basis of $V$.
\end{lemma}

\begin{proof}
Write $W$ for the space spanned by the $\boldsymbol{1}_{C_1}, \dots, \boldsymbol{1}_{C_r}$. Clearly, $V \leq W$.
Conversely, fix an equivalence class indicator vector $\boldsymbol{1}_C$. Then there exists $v \in V$ such that $v(i) \neq v(j)$ for some $i \in C$ and $j \in I \setminus C$. Furthermore, $v$ is constant on $C$. Let $\ell$ denote the number of different values that $v$ exhibits. Reduce $v$ to a length-$\ell$ vector $v'$ keeping only one entry for every occurring value. The space spanned by $\boldsymbol{1}, v', v'^{\odot 2}, \dots, v'^{\odot(\ell-1)}$ is $\ell$-dimensional by \cref{fact:vd}. Hence, every length-$\ell$ vector is a linear combination of iterated Schur products of $v'$. Lifting the reduced vectors by replicating entries, shows that $\boldsymbol{1}_C$ is in the space spanned by iterated Schur products of $v$. Hence $W \leq V$.
\end{proof}

For the sake of completeness, we give a proof of the following lemma on doubly stochastic matrices.
Recall that a matrix $X \in \mathbb{R}^{I \times J}$ is doubly stochastic if it is pseudo-stochastic and its entries are non-negative.

\begin{lemma}[{\cite[Lemma~1]{tinhofer_graph_1986}}] \label{lem:ds}
	Let $I$ and $J$ be finite sets.
Let $v \in \mathbb{R}^I$ and $w \in \mathbb{R}^J$.
If $X \in \mathbb{R}^{J \times I}$ is doubly stochastic such that $Xv = w$ and $X^T w = v$ then $v(i) = w(j)$ 
for all $i \in I$ and $j \in J$ such that $X(j, i) > 0$.
\end{lemma}
\begin{proof}
In virtue of a classical theorem of Hardy, Littlewood, and Polya \cite[Theorem~1a]{mirsky_results_1963}, $Xv = w$ and $X^Tw = v$ imply that $v$ and $w$ have the same multisets of entries.

Let $A \subseteq I$, respectively $B \subseteq J$, denote the set of indices on which $v$, respectively $w$, assumes its least value $r$.
It is claimed that if $i \in I \setminus A$ and $j \in B$ or if $i \in A$ and $j \in J \setminus B$ then $X(j,i) = 0$.
By induction on the number of different values in $v$ and $w$, this yields the claim.
First observe that for $j \in B$,
\begin{align*}
r
= w(j) 
= 
(Xv)(j) 
= r \sum_{i \in A} X(j, i)
+ \sum_{i \in I \setminus A} X(j,i)v(i) 
\geq r  \sum_{i \in I} X(j,i)
 = r (X\boldsymbol{1})(j)
 = r.
\end{align*}
Hence, equality holds throughout.
This implies that $\sum_{i \in I \setminus A} X(j,i) = 0$ as $v(i) > r$ for all $i \in I \setminus A$.
It follows that $X(j,i) = 0$ for $j \in B$ and $i \in I \setminus A$.
The same holds when $i \in A$ and $j \in J \setminus B$. 
\end{proof}

\subsection{Representation Theory of Involution Monoids}

A \emph{monoid} $\Gamma$ is a possibly infinite set equipped with an associative binary operation and an identity element denoted by~$1_\Gamma$.
An example for a monoid is the \emph{endomorphism monoid} $\End V$ for a vector space~$V$ over~$\mathbb{K}$ with composition as binary operation and $\id_V$ as identity element. A \emph{monoid representation} of $\Gamma$ is a map $\phi \colon \Gamma \to \End V$ such that $\phi(1_\Gamma) = \id_V$ and $\phi(gh) = \phi(g)\phi(h)$ for all $g,h \in \Gamma$. The representation is \emph{finite-dimensional} if $V$ is finite-dimensional. 
For every monoid $\Gamma$, there exists a representation, for example the \emph{trivial representation} $\Gamma \to \End \{0\}$ given by $g \mapsto \id_{\{0\}}$.

Let $\phi \colon \Gamma \to \End(V)$ and $\psi \colon \Gamma \to \End(W)$ be two representations over $\mathbb{K}$.
Then $\phi$ and $\psi$ are \emph{equivalent} if there exists a $\mathbb{K}$-vector space isomorphism $X \colon V \to W$ such that $X \phi(g) = \psi(g) X$ for all $g \in \Gamma$.
Moreover, $\phi$ is a \emph{subrepresentation} of $\psi$ if $V \leq W$ and $\psi(g)$ restricted to $V$ equals $\phi(g)$ for all $g \in \Gamma$.
A representation~$\phi$ is \emph{simple} if its only subrepresentations are the trivial representation and $\phi$ itself.
The \emph{direct sum} of $\phi$ and $\psi$ denoted by $\phi \oplus \psi \colon \Gamma \to \End(V \oplus W)$ is the representation that maps $g \in \Gamma$ to $\phi(g) \oplus \psi(g) \in \End(V) \oplus \End(W) \leq \End(V \oplus W)$.
A representation $\phi$ is \emph{semisimple} if it is the direct sum of simple representations.

Let $\phi \colon \Gamma \to \End V$ be a representation with subrepresentations $\psi' \colon \Gamma \to \End V'$ and $\psi'' \colon \Gamma \to \End V''$. 
The representation $\phi' \colon \Gamma \to \End(V' \cap V'')$ sending $g \in \Gamma$ to the restriction $\phi(g)|_{V' \cap V''} \in \End(V' \cap V'')$ of $\phi(g)$ to $V' \cap V''$ is called the \emph{intersection of $\psi'$ and $\psi''$}. For a set $S \subseteq V$, define the \emph{subrepresentation of $\phi$ generated by $S$} as the intersection of all subrepresentations $\psi' \colon \Gamma \to \End V'$ of $\phi$ such that $S \subseteq V'$.

The \emph{character} of a representation $\phi$ is the map $\chi_\phi \colon \Gamma \to \mathbb{K}$ defined as $g \mapsto \tr(\phi(g))$.
Its significance stems from the following theorem, which can be traced back to Frobenius and Schur~\cite{frobenius_uber_1906}. For a contemporary proof, consult \cite[Theorem~7.19]{lam_first_2001}.

\begin{theorem}[{Frobenius--Schur \cite{frobenius_uber_1906}}]
\label{thm:frobenius-schur}
Let $\Gamma$ be a monoid. Let $\phi \colon \Gamma \to \End(V)$ and $\psi \colon \Gamma \to \End(W)$ be finite-dimensional semisimple representations over $\mathbb{K}$. 
Then $\phi$ and $\psi$ are equivalent if and only if $\chi_\phi = \chi_\psi$.
\end{theorem}

The monoids studied in this work are equipped with an additional structure which ensures that their finite-dimensional representations are always semisimple: 
An \emph{involution monoid}\footnote{Involution monoids correspond to $*$-algebras: A $*$-algebra is an involution monoid with additional structure. Conversely, an involution monoid $\Gamma$ gives rise to the $*$-algebra $\mathbb{C}\Gamma$ of formal finite $\mathbb{C}$-linear combinations of elements in $\Gamma$. 
	Although $*$-algebras are much more studied than involution monoids,
	we choose to work with the latter in order to maintain a clear separation between algebraic and combinatorial  arguments in \cref{sec:three-variants,sec:graphclasses}, respectively.} is a monoid $\Gamma$ with a unary operation ${}^* \colon \Gamma \to \Gamma$ such that
\begin{equation} \label{eq:involution-laws}
	(gh)^* = h^* g^*  \quad \text{and} \quad (g^*)^* = g \quad \text{for all } g, h \in \Gamma.
\end{equation}
Note that $\End V$ is an involution monoid with the adjoint operation $X \mapsto X^*$.
Representations of involution monoids must preserve the involution operations.

\begin{lemma} \label{prop:semisimple}
	Every finite-dimensional representation of an involution monoid~$\Gamma$ over $\mathbb{K}$ is semisimple.
\end{lemma}
\begin{proof}
	Let $\phi \colon \Gamma \to \End V$ be a finite-dimensional representation of $\Gamma$. It suffices to show that for every subrepresentation $\psi \colon \Gamma \to \End W$ of $\phi$ there exists a subrepresentation $\psi' \colon \Gamma \to \End W'$ of $\phi$ such that $\phi = \psi \oplus \psi'$, i.e.\@ $\phi$ acts as $\psi$ on $W$ and as $\psi'$ on $W'$.
	Set $W'$ to be the orthogonal complement of $W$ in $V$. It has to be shown that $\phi(g) \in \End V$ for every $g \in \Gamma$ can be restricted to an endomorphism of $W'$. Let $w \in W$ and $w' \in W'$ be arbitrary. Then
	$
	\left< \phi(g) w', w \right>
	= \left< w', \phi(g)^* w \right>
	= \left< w', \phi(g^*)  w \right>
	= 0
	$
	since $\phi(g^*)$ maps $W \to W$ and $W \perp W'$. Hence, the image of $W'$ under $\phi(g)$ is contained in the orthogonal complement of $W$, which equals $W'$. Clearly, $\phi = \psi \oplus \psi'$.
\end{proof}

\subsection{Tensors, Tensor Maps, and Algebraic Operations}
\label{sec:tensors}
In this section, we define tensors and tensor maps as well as algebraic operations involving them.

Fix $\mathbb{K} \in \{\mathbb{Q}, \mathbb{R}, \mathbb{C}\}$.
For a finite set $I$ and $k, \ell \in \mathbb{N}$, the set of all functions $X\colon I^k \times I^\ell \to \mathbb{K}$ forms a $\mathbb{K}$-vector space denoted by $\mathbb{K}^{I^k \times I^\ell}$. 
We call the elements of $\mathfrak{T}_I(k, \ell) \coloneqq \mathbb{K}^{I^k \times I^\ell}$ the \emph{$(k,\ell)$-shaped tensors over~$I$}.
A $(k,0)$-shaped tensor is also called \emph{$k$-shaped} and $\mathfrak{T}_I(k)$ is defined as $\mathfrak{T}_I(k,0)$.
We identify $0$-shaped tensors with scalars, i.e.\@ $\mathfrak{T}_I(0,0) = \mathfrak{T}_I(0) = \mathbb{K}$. Furthermore, $1$-shaped tensors are vectors in $\mathfrak{T}_I(1) = \mathbb{K}^{I}$, 
$(1,1)$-shaped tensors are matrices in $\mathbb{K}^{I\times I}$, et cetera.

The set $\mathfrak{T}_I(k, \ell)$ forms a $\mathbb{K}$-vector space under point-wise operations, i.e.\@ for $\phi, \psi \in \mathfrak{T}_I(k, \ell)$ and $a,b \in \mathbb{K}$, $(a\phi + b\psi)(\boldsymbol{u}, \boldsymbol{v}) \coloneqq a\phi(\boldsymbol{u}, \boldsymbol{v}) + b\psi(\boldsymbol{u}, \boldsymbol{v})$
for $\vec v\in I^k$ and $\boldsymbol{u} \in I^\ell$.

The \emph{sum-of-entries} of $\phi \in \mathfrak{T}_I(k, \ell)$ is the $(0,0)$-shaped tensor $\soe \phi$ defined via $(\soe \phi) \coloneqq \sum_{\boldsymbol{u} \in I^k , \boldsymbol{v} \in I^\ell} \phi(\boldsymbol{u},\boldsymbol{v})$.
The \emph{trace} of $\phi \in \mathfrak{T}_I(k,k)$ is the $(0,0)$-shaped tensor $\tr \phi$ defined via 
$(\tr \phi) \coloneqq \sum_{\boldsymbol{u}\in I^k} \phi(\boldsymbol{u},\boldsymbol{u})$.

The \emph{tensor product} $\phi \otimes \psi \in \mathfrak{T}_I(k_1+k_2, \ell_1+\ell_2)$ for two tensors $\phi \in \mathfrak{T}_I(k_1, \ell_1)$, $\psi\in \mathfrak{T}_I(k_2, \ell_2)$ where $k_1,k_2,\ell_1,\ell_2 \in \mathbb{N}$ is defined by the equation 
$(\phi \otimes \psi)(\boldsymbol{u}\boldsymbol{u}', \boldsymbol{v}\boldsymbol{v}') \coloneqq \phi(\boldsymbol{u},\boldsymbol{v}) \psi(\boldsymbol{u}',\boldsymbol{v}')$ for $\boldsymbol{u} \in I^{k_1}$, $\boldsymbol{u}' \in I^{k_2}$, $\boldsymbol{v} \in I^{\ell_1}$, $\boldsymbol{v}' \in I^{\ell_2}$. 

The \emph{matrix-vector product} $\phi \cdot \psi \in \mathfrak{T}_I(k)$ for two tensors  $\phi \in \mathfrak{T}_I(k,\ell)$ and $\psi \in \mathfrak{T}_I(\ell)$ is the tensor defined by $(\phi \cdot \psi)(\vec{v}) \coloneqq \sum_{\vec{w} \in I^\ell} \phi(\vec{v},\vec{w}) \psi(\vec{w})$ and $\boldsymbol{v} \in I^{k}$.
Analogously, the \emph{matrix product} of $\phi \in \mathfrak{T}_I(k,\ell)$ and $\psi \in \mathfrak{T}_I(\ell, m)$ for $m \in \mathbb{N}$ denoted by $\phi \cdot \psi \in \mathfrak{T}_I(k,m)$ is defined as $(\phi \cdot \psi)(\vec{v},\vec{u}) \coloneqq \sum_{\vec{w} \in I^k} \phi(\vec{v},\vec{w}) \psi(\vec{w}, \vec{u})$  for  $\boldsymbol{v} \in I^{k}$ and $\boldsymbol{u} \in I^m$.

The \emph{Schur product} of $\phi, \psi \in \mathfrak{T}_I(k)$ is defined as $(\phi \odot \psi)(\vec{v}) \coloneqq \phi(\vec{v})\psi(\vec{v})$  for $\boldsymbol{v} \in I^{k}$. The \emph{inner-product} of $\phi, \psi \in \mathfrak{T}_I(k)$ is $\left< \phi, \psi \right> \coloneqq \sum_{\boldsymbol{v} \in I^k} \overline{\phi(\boldsymbol{v})} \psi(\boldsymbol{v})$ if $\mathbb{K} = \mathbb{C}$ and 
$\left< \phi, \psi \right> \coloneqq \sum_{\boldsymbol{v} \in I^k} \phi(\boldsymbol{v}) \psi(\boldsymbol{v})$ otherwise.

A \emph{$(k,\ell)$-shaped tensor map on graphs} is a function $\phi$ that maps a graph $G$ to a $(k,\ell)$-shaped tensor $\phi_G \in \mathfrak{T}_{V(G)}(k,\ell)$ over the set of vertices $V(G)$ of $G$.
A $(k,\ell)$-shaped tensor map $\phi$ is \emph{equivariant} if for all isomorphic graphs $G$ and $H$, all isomorphisms $f$ from $G$ to $H$, and all $\vec v\in V(G)^k$ and $\boldsymbol{u} \in V(G)^\ell$, it holds that $\phi(\vec v, \vec u)=\phi_H(f(\vec v), f(\vec u))$ where $f$ acts entry-wise on tuples.
We write $\mathfrak{F}(k, \ell)$ for the set of all equivariant $(k,\ell)$-shaped  tensor maps on graphs.

The operation defined on $\mathfrak{T}_I(k, \ell)$ above induce operations on $(k,\ell)$-shaped tensor maps on graphs.
It is easy to see that these operation preserve equivariance and thereby induce operations on $\mathfrak{F}(k,\ell)$.
For example, for $\phi \in \mathfrak{F}(k,\ell)$ and $\psi \in \mathfrak{F}(\ell, m)$, 
it is $\phi \cdot \psi \in \mathfrak{F}(k, m)$ with $(\phi \cdot \psi)_G(\boldsymbol{v}, \boldsymbol{u}) \coloneqq (\phi_G \cdot \psi_G)(\boldsymbol{v}, \boldsymbol{u}) = \sum_{\boldsymbol{w} \in V(G)^\ell} \phi_G(\boldsymbol{v}, \boldsymbol{w}) \psi_G(\boldsymbol{w}, \boldsymbol{u})$ for every graph $G$ and $\boldsymbol{u} \in V(G)^k$, $\boldsymbol{w} \in V(G)^m$.

\section{Three Variants of a Theorem by Specht and Wiegmann}
\label{sec:three-variants}

In this section, two novel variants of a classical theorem by Specht~\cite{specht_zur_1940} and Wiegmann~\cite{wiegmann_necessary_1961} are derived.
All results give criteria for simultaneous similarity of sequences of matrices in terms of character-like functions.

We fix throughout a possibly infinite set of indices $M$.
Let $\Gamma_M$ denote the set of all finite words over the alphabet $\{x_i, x^*_i \mid i \in M\}$.
Clearly, $\Gamma_M$ forms a monoid under concatenation and with the empty word $\epsilon$ as unit element.
Furthermore, it can be endowed with an involution defined by extending $x_i \mapsto x^*_i$ to $\Gamma_M$ in accordance with \cref{eq:involution-laws}.
In this way, $\Gamma_M$ can be thought of as a \emph{free involution monoid}.

Let $\matvec{A} = (A_i)_{i \in M}$ be a sequence of matrices in $\mathbb{K}^{I \times I}$ for some finite index set $I$.
For a word $w \in \Gamma_{M}$, let $w_{\matvec{A}}$ denote the matrix obtained by substituting $x_i \mapsto A_i$ and $x^*_i \mapsto A_i^*$ for all $i \in M$ and evaluating the matrix product.
Furthermore, $\epsilon_{\matvec{A}}$ is set to be the identity matrix in $\mathbb{K}^{I \times I}$.
Crucially, the words in $\Gamma_M$ are finite despite that the underlying alphabet is infinite. Hence, this map is well-defined.
The substitution $w \mapsto w_{\matvec{A}}$ is an involution monoid representation of~$\Gamma_{M}$.

\subsection{Unitary and Orthogonal Similarity}

We first recall the classical Specht--Wiegmann Theorem.
See also \cite{jing_unitary_2015,futorny_spechts_2017} for more recent accounts.

\begin{theorem}[Specht \cite{specht_zur_1940}, Wiegmann \cite{wiegmann_necessary_1961}] \label{thm:wiegmann}
	Let $\mathbb{K} \in \{\mathbb{R}, \mathbb{C}\}$.
	Let $I$ and $J$ be finite index sets. Let $M$ be any set.
	Let $\matvec{A} = (A_i)_{i \in M}$ and $\matvec{B} = (B_i)_{i \in M}$  be two sequences of matrices such that $A_i \in \mathbb{K}^{I \times I}$ and $B_i \in \mathbb{K}^{J \times J}$  for all $i \in M$. Then the following are equivalent: 
	\begin{enumerate}
		\item there exists a unitary $U \in \mathbb{K}^{J \times I}$ such that $U A_i  =  B_i U$ and $U A_i^*  =  B_i^*U $ for every $i \in M$,
		\item for every word $w \in \Gamma_{M}$, $\tr(w_{\matvec{A}}) = \tr(w_{\matvec{B}})$.\label{it:spw3}
	\end{enumerate} 
\end{theorem}
\begin{proof}As observed above, the maps $w \mapsto w_\matvec{A}$ and $w \mapsto w_\matvec{B}$ yield two representations of the involution monoid $\Gamma_{M}$.
	By \cref{prop:semisimple}, these representations are semisimple.
	If $\tr(w_\matvec{A}) = \tr(w_\matvec{B})$ for every word $w \in \Gamma_{M}$, these two representations have the same character, 
	and hence, by \cref{thm:frobenius-schur}, they are equivalent. Therefore, there exists an invertible matrix $X$ such that 
	$X^{-1} w_\matvec{B} X = w_\matvec{A}$ for every word $w \in \Gamma_{M}$. 
	
	The desired unitary matrix $U$ can then be recovered from the polar decomposition of $X= HU$, where $U$ is unitary and $H$ is positive semi-definite, cf.\@ \cite[p.\@~292]{lang_linear_1987} and \cite[Corollary~2.3]{jing_unitary_2015}.
	Since $XX^*$ commutes with $B_i$ for $i \in M$, so does $H$ and hence, $X^{-1}B_iX = U^*(H^{-1}B_iH)U = U^*B_iU = A_i$ for $i \in M$.
\end{proof}

In fact, it suffices to compare traces from finitely many words in $\Gamma_M$ to establish the assertions of \cref{thm:wiegmann}.
The following result is due to \cite{pearcy_complete_1962}. Tighter bounds are known \cite{pappacena_upper_1997}, a linear bound is conjectured. 
We include the following proof for completeness.

\begin{theorem}[{\cite[Theorem~1]{pearcy_complete_1962}}]	\label{lem:pearcy}
	Writing $n \coloneqq |I| = |J|$,
	the conditions of \cref{thm:wiegmann} are equivalent to the following: For every word $w \in \Gamma_{M}$ of length at most $2n^2 - 1$, $\tr(w_{\matvec{A}}) = \tr(w_{\matvec{B}})$.
\end{theorem}
\newcommand{\bothmat}{{\matvec{A} \oplus \matvec{B}}}
\begin{proof} 
	To ease notation, we suppose wlog that $I$ and $J$ are disjoint.
	For a word $w \in \Gamma_M$, define the block matrix $w_{\bothmat} \coloneqq \left( \begin{smallmatrix}
		w_{\matvec A} & 0 \\ 0 & w_{\matvec B}
	\end{smallmatrix}\right) \in \mathbb{K}^{(I \cup J) \times (I \cup J) }$.
	Observe that $w \mapsto w_{\bothmat}$ is an involution monoid representation. In particular, $(xy)_{\bothmat} = x_{\bothmat}y_{\bothmat}$ and $(x^*)_{\bothmat} = (x_\bothmat)^*$ for all $x,y\in \Gamma_M$.
	
	Let $S_{\bothmat} \leq \mathbb{K}^{(I \cup J) \times (I \cup J)}$ denote the vector space spanned by the $w_{\bothmat}$ for $w \in \Gamma_M$.
	Clearly, $S_{\bothmat}$ has dimension at most $2n^2$.
	Furthermore, for $\ell \geq 0$, write $S_{\bothmat}^{\leq \ell} \leq S_{\bothmat}$ for the vector space spanned by the $w_{\bothmat}$ for words $w \in \Gamma_M$ of length at most $\ell$.
	The space $S_{\bothmat}^{\leq 0}$ containing the identity matrix, has dimension~$1$.
	
	\begin{claim} \label{cl:space-collapse}
		If $S_{\bothmat}^{\leq \ell} = S_{\bothmat}^{\leq \ell +1}$ for some $\ell \in \mathbb{N}$ then $S_{\bothmat} = S_{\bothmat}^{\leq \ell}$.
		In particular, $S_{\bothmat}^{\leq 2n^2-1} = S_{\bothmat}$.
	\end{claim}
	\begin{claimproof}
		By induction on~$j \geq 1$, it is shown that $S_{\bothmat}^{\leq \ell+j} \leq S_{\bothmat}^{\leq \ell}$. The base case $j=1$ holds by assumption.
		Let $x \in \Gamma_{M}$ be a word of length $\ell+j+1$.
		Let $y$ denote the first character of $x$ and write $z$ for the  length-$(\ell+j)$ suffix of $x$, i.e.\@ $x = yz$.
		By assumption, there exist words $z^1, \dots, z^r$ of length at most $\ell$ and coefficients $\alpha_1,\dots, \alpha_r \in \mathbb{K}$ such that $z_\bothmat = \sum_{i=1}^r \alpha_i z_\bothmat^i$. 
		Hence, $x_\bothmat = y_\bothmat z_\bothmat = \sum_{i=1}^r \alpha_i y_\bothmat z_\bothmat^i = \sum_{i=1}^r \alpha_i (y z^i)_{\bothmat} \in S_\bothmat^{\leq \ell + 1}$.
		Thus, $S_\bothmat^{\leq \ell+j+1} \leq S_\bothmat^{\leq \ell + 1} \leq S_{\bothmat}^{\leq \ell}$, as desired.
	\end{claimproof}	
		
	Equipped with \cref{cl:space-collapse}, we prove the main claim.
	For a matrix $C \in  \mathbb{K}^{(I \cup J) \times (I \cup J)}$, write $\tr_{\matvec{A}} C \coloneqq \sum_{i \in I} C(i,i)$
	and analogously $\tr_{\matvec{B}} C \coloneqq \sum_{j \in J} C(j,j)$.
	Let $w \in \Gamma_{M}$ be arbitrary.
	By \cref{cl:space-collapse}, there exist $w^1, \dots, w^r \in \Gamma_M$ of length at most $2n^2-1$ and coefficients $\alpha_1, \dots, \alpha_r \in \mathbb{K}$ such that
	\( w_{\bothmat} = \sum_{i=1}^r \alpha_i w^i_{\bothmat}\).
	Hence,
	\[
	\tr(w_{\matvec{A}}) 
	= \tr_{\matvec{A}}(w_\bothmat) 
	= \sum \alpha_i\tr_{\matvec{A}}(w^i_{\bothmat})
	= \sum \alpha_i\tr(w^i_{\matvec{A}})
	= \sum \alpha_i\tr(w^i_{\matvec{B}})
	= \tr_{\matvec{B}}(w_\bothmat) 
	= \tr(w_{\matvec{B}}). 
	\]
	Thus, the assertion in \cref{lem:pearcy} implies \cref{it:spw3} of \cref{thm:wiegmann}.
\end{proof}

\subsection{Pseudo-Stochastic Similarity}

Our first variant of \cref{thm:wiegmann} establishes a criterion for simultaneous similarity w.r.t.\@ a pseudo-stochastic matrix. 
In this case, instead of traces, sums-of-entries have to be considered.

\begin{theorem}\label{thm:soe}
	Let $\mathbb{K} \in \{\mathbb{Q}, \mathbb{R}, \mathbb{C}\}$.
	Let $I$ and $J$ be finite index sets. Let $M$ be any set.
	Let $\matvec{A} = (A_i)_{i \in M}$ and $\matvec{B} = (B_i)_{i \in M}$  be two sequences of matrices such that $A_i \in \mathbb{K}^{I \times I}$ and $B_i \in \mathbb{K}^{J \times J}$  for $i \in M$. Then the following are equivalent: 
	\begin{enumerate}
		\item there exists a pseudo-stochastic matrix $X \in \mathbb{K}^{J \times I}$ such that $X A_i  =  B_iX $ and $X A_i^*  =  B_i^* X$ for all $i \in M$,
		\item for every word $w \in \Gamma_{M}$, $\soe(w_{\matvec{A}}) = \soe(w_{\matvec{B}})$.\label{it:thm:soe2} 
	\end{enumerate} 
\end{theorem}

\Cref{thm:soe} is implied by \cref{lem:invmodsoe}, which provides a sum-of-entries analogue of \cref{thm:frobenius-schur}. As it establishes a character-theoretic interpretation of the function $\soe$, it may be of independent interest.

\begin{lemma} \label{lem:invmodsoe}
	Let $\mathbb{K} \in \{\mathbb{Q}, \mathbb{R}, \mathbb{C}\}$.
	Let $\Gamma$ be an involution monoid.
	Let $I$ and $J$ be finite index sets.
	Let $\phi \colon \Gamma \to \mathbb{K}^{I \times I}$ and $\psi \colon \Gamma \to \mathbb{K}^{J \times J}$ be representations of $\Gamma$.
	Let $\phi' \colon \Gamma \to \End(V)$ and $\psi' \colon \Gamma \to \End(W)$ denote the subrepresentations of $\phi$ and of $\psi$ generated by $\allones_I$ and $\allones_J$, respectively.
	Then the following are equivalent:
	\begin{enumerate}
		\item for all $g \in \Gamma$, $\soe \phi(g) = \soe \psi(g)$,\label{rep1}
		\item there exists a unitary pseudo-stochastic $U \colon V \to W$ such that $U \phi'(g) = \psi'(g) U$ for all $g\in \Gamma$,\label{rep2}
		\item there exists a pseudo-stochastic $X \in \mathbb{K}^{J \times I}$ such that $X \phi(g) = \psi(g) X$ for all $g \in \Gamma$.\label{rep3}
	\end{enumerate}
\end{lemma}
\begin{proof}
	Suppose that \cref{rep1} holds.
	The space $V$ is spanned by the vectors $\phi(g) \allones_I$ for $g \in \Gamma$ while $W$ is spanned by the $\psi(g) \allones_J$ for $g \in \Gamma$. For $g, h \in \Gamma$, it holds that
	\[
	\left< \phi(g) \allones_I, \phi(h) \allones_I \right>
	= \left< \allones_I, \phi(g^* h) \allones_I \right>
	= \soe \phi(g^* h)
	= \soe \psi(g^* h)
	= \left< \psi(g) \allones_J, \psi(h) \allones_J \right>.
	\]
	Hence, $V$ and $W$ are spanned by vectors whose pairwise inner-products are respectively the same. Thus, by \cref{lemma:gs}, there exists a unitary $U \colon V \to W$ such that $U \phi(g) \allones_I = \psi(g)\allones_J$ for all $g \in \Gamma$.
	This immediately implies that $U \phi'(g) = \psi'(g) U$ for $g \in \Gamma$. 
	Furthermore, $U \allones_I = U \phi(1_\Gamma) \allones_I = \psi(1_\Gamma) \allones_J = \allones_J$ and $U^* \allones_J = \allones_I$ since $U$ is unitary.
	Thus, \cref{rep2} holds.
	
	Suppose now that \cref{rep2} holds. By \cref{prop:semisimple}, write $\phi = \phi' \oplus \phi''$ and $\psi = \psi' \oplus \psi''$. By assumption, there exists  a unitary $U \colon V \to W$ such that $U \phi'(g) = \psi'(g) U$ for all $g \in \Gamma$. Extend $U$ to $X$ by letting it annihilate $V^\perp$. Then
	$X \phi(g) = (U \oplus 0)(\phi' \oplus \phi'')(g) = U\phi'(g) \oplus 0 = \psi'(g) U \oplus 0 = \psi(g) X$ for all $g \in \Gamma$. Since $U$ is pseudo-stochastic and $\allones_I \in V$ and $\allones_J \in W$, $X$ is pseudo-stochastic as well. Hence, \cref{rep3} holds.
	That \cref{rep3} implies \cref{rep1} is immediate.
\end{proof}

The following \cref{prop:soe-quadratic-length} parallels the polynomial bound  from \cref{lem:pearcy} on the length of the words which need to be inspected.

\begin{theorem} \label{prop:soe-quadratic-length}
	Writing $n \coloneqq |I| = |J|$,
	the conditions of \cref{thm:soe} are equivalent to the following: For every word $w \in \Gamma_{M}$ of length at most $2n-1$, $\soe(w_{\matvec{A}}) = \soe(w_{\matvec{B}})$.
\end{theorem}
\begin{proof}
	Suppose wlog that $I$ and $J$ are disjoint.
	As in the proof of \cref{lem:pearcy}, write $w_{\bothmat} \coloneqq \left( \begin{smallmatrix}
		w_{\matvec A} & 0 \\ 0 & w_{\matvec B}
	\end{smallmatrix}\right) \in \mathbb{K}^{(I \cup J) \times (I \cup J) }$ for $w \in \Gamma_M$.
	Furthermore, write $\boldsymbol{1}$ for the all-ones vector in $\mathbb{K}^{I \cup J}$.
	Write $V_{\bothmat} \leq \mathbb{K}^{I \cup J}$ for the space spanned by the $w_{\bothmat} \boldsymbol{1}$ for all $w \in \Gamma_M$.
	Write $V_{\bothmat}^{\leq \ell} \leq V_{\bothmat}$ for $\ell \geq 0$ for the subspace spanned by the $w_{\bothmat} \boldsymbol{1}$ for $w \in \Gamma_M$ of length~$\leq \ell$.
	Clearly, $V_{\bothmat}^{\leq 0}$ containing $\boldsymbol{1}$ is one-dimensional.
	The space $V_{\bothmat}$ is at most $2n$-dimensional.
	
	\begin{claim} \label{cl:space-collapse-soe}
		If $V_{\bothmat}^{\leq \ell} = V_{\bothmat}^{\leq \ell +1}$ for some $\ell \in \mathbb{N}$ then $V_{\bothmat} = V_{\bothmat}^{\leq \ell}$.
		In particular, $V_{\bothmat}^{\leq 2n-1} = V_{\bothmat}$.
	\end{claim}
	\begin{claimproof}
		By induction on $j \geq 1$, we show that $V_\bothmat^{\leq \ell + j} \leq V_\bothmat^{\leq \ell}$.
		The base case $j=1$ holds by assumption.
		Let $x \in \Gamma_{M}$ be a word of length $\ell+j+1$. 
		Let $y$ denote the first character of $x$ and write $z$ for the  length-$(\ell+j)$ suffix of $x$, i.e.\@ $x = yz$.
		By assumption, there exist words $z^1, \dots, z^r$ of length at most~$\ell$ and coefficients $\alpha_1,\dots, \alpha_r \in \mathbb{K}$ such that $z_\bothmat \boldsymbol{1} = \sum_{i=1}^r \alpha_i z_\bothmat^i \boldsymbol{1}$. 
		Hence, $x_\bothmat\boldsymbol{1} = y_\bothmat z_\bothmat\boldsymbol{1} = \sum_{i=1}^r \alpha_i y_\bothmat z_\bothmat^i\boldsymbol{1} = \sum_{i=1}^r \alpha_i (y z^i)_{\bothmat}\boldsymbol{1} \in V_\bothmat^{\leq \ell + 1}$.
		Thus, $V_\bothmat^{\leq \ell +j+1} \leq V_\bothmat^{\leq \ell +1}  \leq V_\bothmat^{\leq \ell} $, as desired.
	\end{claimproof}
	
	Let $w \in \Gamma_M$ be arbitrary. By \cref{cl:space-collapse-soe}, there exist $w^i \in \Gamma_M$ of length at most $2n-1$ and coefficients $\alpha_i \in \mathbb{K}$ such that $w_{\bothmat} \allones = \sum \alpha_i w^i_{\bothmat} \allones$. Hence, writing $\boldsymbol{1}_A$ and $\boldsymbol{1}_B$ for the indicator vectors on $I$ and $J$ in $\mathbb{K}^{I \cup J}$,
	\[
	\soe(w_{\matvec{A}}) 
= \boldsymbol{1}_{\matvec{A}}^T w_{\bothmat} \boldsymbol{1}
		= \sum \alpha_i\boldsymbol{1}_{\matvec{A}}^T w^i_{\bothmat} \boldsymbol{1}
		= \sum \alpha_i \soe(w^i_{\matvec{A}})
		= \sum \alpha_i \soe(w^i_{\matvec{B}})
		= \soe(w_{\matvec{B}}). \qedhere
		\]
	\end{proof}

\subsection{Doubly Stochastic Similarity}

Our second variant of the Specht--Wiegmann Theorem gives a criterion for simultaneous doubly stochastic similarity. 
Recall that a real matrix is doubly stochastic if it is pseudo-stochastic and has non-negative entries. 
Since double stochasticity makes no sense over complex numbers, we restrict our attention to representations of involution monoids over real vector spaces.
In contrast to \cref{thm:wiegmann,thm:soe}, the criterion derived in this section does not involve words over some set of matrices but trees, defined as follows:

\begin{definition}[Trees over a monoid]
	Let $\Gamma$ be a monoid. A \emph{tree over $\Gamma$} is a tuple $t = (T, r, e)$ where $T$ is a finite tree, $r \in V(T)$, and $e \colon E(T) \to \Gamma$ is a map which assigns an element of $\Gamma$ to every edge of $T$. Write $T(\Gamma)$ for the set of trees over $\Gamma$.
\end{definition}

Two trees $t = (T, r, e)$ and $t' = (T', r', e')$ over $\Gamma$ are \emph{isomorphic}
if there is a graph isomorphism $h \colon T \to T'$ such that  $h(r) = r'$ and $e'(h(u)h(v))) = e(uv)$ for all $uv \in E(T)$.
We tacitly identify isomorphic trees over $\Gamma$ and
write $T(\Gamma)$ for the set of (isomorphism types) of trees over $\Gamma$.
We consider the following operations on $T(\Gamma)$.

\begin{definition} \label{def:trees-monoid-ops}
	Let $t = (T, r, e)$ and $t' = (T', r', e')$ be elements of $T(\Gamma)$. Let $g \in \Gamma$.
	\begin{enumerate}
		\item Define $t \odot t' \coloneqq (T'', r'', e'') \in T(\Gamma)$ where $T''$ is the tree obtained from taking the disjoint union of $T$ and $T'$ and identifying the vertices $r$ and $r'$. Furthermore, $r''$ is the identified vertex and $e'' \colon E(T'') \to \Gamma$ is such that $e''|_{E(T)} = e$ and $e''|_{E(T')} = e'$.
		\item Define $gt \coloneqq (T'', r'', e'') \in T(\Gamma)$ where $V(T'') \coloneqq V(T) \sqcup \{r''\}$, $E(T'') \coloneqq E(T) \sqcup \{r r''\}$, and $e'' \colon E(T'') \to \Gamma$ is such that $e''|_{E(T)} = e$  and $e''(rr'') = g$.
	\end{enumerate}
\end{definition}

The elements of $T(\Gamma)$ can be constructed from the tree over $\Gamma$ with only one vertex by the operations in \cref{def:trees-monoid-ops}, 
i.e.\ by gluing and by attaching a new root $s$ to a tree $t = (T, r, e)$ and picking an element $g \in \Gamma$ to associate with the new edge $sr$.

The set $T(\Gamma)$ forms a monoid under the operation $\odot$ of gluing two of its elements together at their roots.
Since isomorphic trees are identified in $T(\Gamma)$, the monoid is commutative.
Its unique neutral element is the one-vertex tree.
Endowing trees over monoids with a monoid structure is not a novel idea, cf.\ e.g.\ \cite{bojanczyk_forest_2008}.

A representation of a monoid $\Gamma$ induces a representation\footnote{In \cref{def:tree-rep},  $\mathbb{K}^I$ is understood as the monoid whose binary operation is the Schur product~$\odot$ and whose neutral element is the all-ones vector $\boldsymbol{1}_I$. Crucially, $\mathbb{K}^I$ is commutative. The map $\widehat{\phi} \colon T(\Gamma) \to \mathbb{K}^I$ defined there is strictly speaking only a monoid homomorphism and not a monoid representation as $\mathbb{K}^I$ is not an endomorphism monoid. However, $\mathbb{K}^I$ can be identified with the set of diagonal matrices in $\mathbb{K}^{I \times I}$ in which case the Schur product in $\mathbb{K}^I$ corresponds to the standard matrix product in $\mathbb{K}^{I \times I}$. Thus, calling $\widehat{\phi}$ a monoid representation is only slightly abusive. It allows us to make a distinction between matrices and vectors, which will in \cref{sec:graphclasses} amount to a distinction between bilabelled and labelled graphs.} of the monoid $T(\Gamma)$.

\begin{definition}\label{def:tree-rep}
	Let $\mathbb{K}$ be a field.
	Let $\Gamma$ be a monoid and let $I$ be a finite set.
	A monoid representation $\phi \colon \Gamma \to \mathbb{K}^{I \times I}$ of $\Gamma$ induces a monoid representation $\widehat{\phi} \colon T(\Gamma) \to \mathbb{K}^I$ of $T(\Gamma)$ defined inductively as follows:
	\begin{enumerate}
		\item $\widehat{\phi}(t) \coloneqq \boldsymbol{1}_I$ if $t$ has only one vertex,
		\item $\widehat{\phi}(t) \coloneqq \widehat{\phi}(t') \odot \widehat{\phi}(t'')$ if $t \coloneqq t' \odot t''$ for $t', t'' \in T(\Gamma)$ on more than one vertex,
		\item $\widehat{\phi}(t) \coloneqq \phi(g) \cdot \widehat{\phi}(t')$ if $t = gt'$ for $t' \in T(\Gamma)$ and $g \in \Gamma$.
	\end{enumerate}
\end{definition}

\begin{figure}
	\centering
	\includegraphics[page=22]{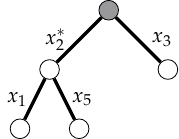}
	\caption[Example for a tree over an involution monoid.]{Example for a tree $t \in \Gamma_M$ for $M = [5]$. The root is depicted in grey. Here, $t_{\matvec A} = (A_2^*((A_1\boldsymbol{1}) \odot (A_5\boldsymbol{1}))) \odot (A_3 \boldsymbol{1})$.}
	\label{fig:example-tree-over-matrix}
\end{figure}

For the involution monoid $\Gamma_M$, we abbreviate the representation of $T(\Gamma_M)$ induced by $w \mapsto w_{\matvec{A}}$ as $t \mapsto t_{\matvec{A}}$.  
See \cref{fig:example-tree-over-matrix} for an example. 
The main result of this section is the following:

\begin{theorem}\label{thm:soe-pos}
	Let $\mathbb{K} \in \{\mathbb{Q}, \mathbb{R}\}$.
	Let $I$ and $J$ be finite index sets. Let $M$ be any set.
	Let $\matvec{A} = (A_i)_{i \in M}$ and $\matvec{B} = (B_i)_{i \in M}$  be two sequences of matrices such that $A_i \in \mathbb{K}^{I \times I}$ and $B_i \in \mathbb{K}^{J \times J}$  for $i \in M$. Then the following are equivalent: 
	\begin{enumerate}
		\item there exists a doubly stochastic matrix $X \in \mathbb{K}^{J \times I}$ such that $X A_i  =  B_iX $ and $X A_i^*  =  B_i^* X$ for all $i \in M$.
		\item for every $t \in T(\Gamma_{M})$, $\soe(t_{\matvec{A}}) = \soe(t_{\matvec{B}})$.\label{it:thm:soe-pos2}
	\end{enumerate} 
\end{theorem}

\Cref{thm:soe-pos} is implied by the following \cref{lem:invmodsoe-pos}.
Inspecting the proof of this lemma, shows that $\mathbb{K}$ can be replaced by $\mathbb{Q}$ in both \cref{thm:soe-pos,lem:invmodsoe-pos}.

\begin{lemma} \label{lem:invmodsoe-pos}
	Let $\mathbb{K} \in \{\mathbb{Q}, \mathbb{R}\}$.
	Let $\Gamma$ be an involution monoid.
	Let $I$ and $J$ be finite index sets.
	Let $\phi \colon \Gamma \to \mathbb{K}^{I \times I}$ and $\phi \colon \Gamma \to \mathbb{K}^{J \times J}$ be representations of $\Gamma$.
	Let $\widehat{\phi}$ and $\widehat{\psi}$ be the induced representations of $T(\Gamma)$.
	Then the following are equivalent:
	\begin{enumerate}
		\item there exists a doubly stochastic $X \in \mathbb{K}^{J \times I}$ such that $X \phi(g) = \psi(g) X$ for all $g \in \Gamma$,\label{p1}
		\item there exists a pseudo-stochastic $X \in \mathbb{K}^{J \times I}$ such that $X \phi(g) = \psi(g) X$ for all $g\in \Gamma$ and $X\widehat{\phi}(t) = \widehat{\psi}(t)$ for all $t \in T(\Gamma)$,\label{p2}
		\item for all $t \in T(\Gamma)$, $\soe \widehat{\phi}(t) = \soe \widehat{\psi}(t)$.\label{p3}
	\end{enumerate}
\end{lemma}

\begin{proof}
\Cref{p1} implies \cref{p2}: Let $X$ be as in \cref{p1}. 
It has to be shown that $X\widehat{\phi}(t) = \widehat{\psi}(t)$ for all $t \in T(\Gamma)$.
The proof of the following slightly stronger claim is guided by \cite[Lemma~1]{tinhofer_graph_1986}.

\begin{claim} \label{claim:trees}
For all $t \in T(\Gamma)$,
if $X(j, i) > 0$ for $i \in I$ and $j \in J$ then 
$\widehat{\phi}(t)(i) = \widehat{\psi}(t)(j)$.
\end{claim}

\Cref{claim:trees} implies that $X\widehat{\phi}(t) = \widehat{\psi}(t)$ for all $t \in T(\Gamma)$.
Indeed, for $t \in T(\Gamma)$ and $j \in J$,
\begin{align} 
(X\widehat{\phi}(t))(j) 
& = \sum_{\substack{i \in I, \\ X(j,i) > 0}}  X(j,i) \widehat{\phi}(t)(i) 
= \sum_{\substack{i \in I, \\ X(j,i) > 0}} X(j,i) \widehat{\psi}(t)(j) 
= \widehat{\psi}(t)(j) (X\boldsymbol{1}_I)(j)
= \widehat{\psi}(t)(j).\label{eq:consequence}
\end{align}
One may similarly see that $X^T\widehat{\psi}(t) = \widehat{\phi}(t)$ for all $t \in T(\Gamma)$.
This statement is applied in the inductive proof of \cref{claim:trees}.

\begin{claimproof}[Proof of \cref{claim:trees}]
The proof is by induction on the structure of the elements of $T(\Gamma)$, cf.\@ \cref{def:tree-rep}. For the single-vertex tree, the claim is vacuous.

For the induction step, the
two means of constructing more complex elements $t = (T, r, e) \in T(\Gamma)$ from \cref{def:trees-monoid-ops}  are considered. 
If $t = t' \odot t''$ for two non-trivial $t', t'' \in T(\Gamma)$, the claim is readily verified.
It remains to consider the case in which $r$ has a unique child $s$ in $T$. 
In this case,
write $t = gt'$ for some $t' \in T(\Gamma)$ and $g \in \Gamma$.

The vectors $\widehat{\phi}(t)$ and $\widehat{\psi}(t)$ satisfy the assumptions of \cref{lem:ds}. Indeed, by \cref{p1,eq:consequence},
\begin{equation}
X \widehat{\phi}(t) 
= X\phi(g)\widehat{\phi}(t') 
= \psi(g) X \widehat{\phi}(t') 
= \psi(g) \widehat{\psi}(t') 
= \widehat{\psi}(t),  \label{eq:xrg}
\end{equation}
and, alluding to the assumption that $\Gamma$ is an involution monoid,
\begin{align}
X^T \widehat{\psi}(t) 
&= X^T \psi(g)\widehat{\psi}(t') 
= (\psi(g^*) X)^T \widehat{\psi}(t') 
= (X \phi(g^*))^T \widehat{\psi}(t') \notag \\
&= \phi(g) X^T \widehat{\psi}(t') 
= \phi(g) \widehat{\phi}(t') 
= \widehat{\phi}(t). \label{eq:xrg2}
\end{align}
\Cref{eq:xrg,eq:xrg2} in conjunction with \cref{lem:ds} imply \cref{claim:trees}.
\end{claimproof}

\Cref{p2} implies \cref{p1}: 
Write $X' \in \mathbb{K}^{J \times I}$ for the matrix in \cref{p2}. 
That is, $X'$ is pseudo-stochastic and satisfies $X' \phi(g) = \psi(g) X'$ for all $g \in \Gamma$ 
as well as $X' \widehat{\phi}(t) = \widehat{\psi}(t)$ for all $t \in T(\Gamma)$.
Write $V \leq \mathbb{K}^I$ for the vector spaced spanned by the $\widehat{\phi}(t)$ for $t \in T(\Gamma)$
and $P \in \mathbb{K}^{I \times I}$ for the projection onto $V$.

We claim that the matrix $X \coloneqq X'P \in \mathbb{K}^{J \times I}$ is doubly stochastic and satisfies $X \phi(g) = \psi(g) X$ for all $g \in \Gamma$.
To that end, first consider the following claim:
\begin{claim} \label{cl:commute-p}
	For $g \in \Gamma$, it holds that $P \phi(g) = \phi(g) P$.
\end{claim}
\begin{claimproof}
	For $t \in T(\Gamma)$ and $g \in \Gamma$,
	it is $\phi(g) \widehat{\phi}(t) = \widehat{\phi}(gt)$.
	Hence,
	the image of $V$ under $\phi(g)$ is contained in $V$.
	Moreover, the image of the orthogonal complement $V^\perp$ of $V$ under $\phi(g)$ is contained in $V^\perp$.
	Indeed, for $t \in T(\Gamma)$ and $v \in V^\perp$,
	$
		\left< \phi(g) v, \widehat{\phi}(t) \right>
		=
		\left< v, \phi(g)^* \widehat{\phi}(t) \right>
		= \left< v, \widehat{\phi}(g^*t) \right> 
		= 0
	$
	since $\widehat\phi(g^*t) \in V$. 
\end{claimproof}
Given \cref{cl:commute-p}, it remains to argue that $X$ is doubly stochastic.
To that end, note that $V$ is closed under Schur products by construction. 
Consider the equivalence relation $\sim$ on~$I$ induced by $V$ via $i \sim i'$ iff $v(i) = v(i')$ for all $v \in V$.
By \cref{lemma:vd}, the indicator vectors of the equivalence classes of this relation are a basis of $V$.
Let $i \in I$ and let $C \subseteq I$ denote its equivalence class. Then $P e_{i} = \boldsymbol{1}_C/|C|$ for $e_{i} \in \mathbb{K}^{I}$ the standard basis vector corresponding to~$i$. It remains to compute $Xe_{i}$. Write $\boldsymbol{1}_C = \sum_r \alpha_r \widehat{\phi}(t_r)$ for some $t_r \in T(\Gamma)$ and $\alpha_r \in \mathbb{K}$. 
Then
\begin{align*}
(X \boldsymbol{1}_C) \odot (X \boldsymbol{1}_C) 
&= \sum_{r,s} \alpha_{r}\alpha_s (X \widehat{\phi}(t_r)) \odot (X \widehat{\phi}(t_s)) \\
&= \sum \alpha_{r}\alpha_s \widehat{\psi}(t_r) \odot \widehat{\psi}(t_s)\\
&= \sum \alpha_{r}\alpha_s \widehat{\psi}(t_r\odot t_s)\\
&= \sum \alpha_{r}\alpha_s X(\widehat{\phi}(t_r\odot t_s))\\
&= \sum \alpha_{r}\alpha_s X(\widehat{\phi}(t_r)\odot \widehat{\phi}(t_s))\\
&= X( \boldsymbol{1}_C \odot \boldsymbol{1}_C)\\
&= X\boldsymbol{1}_C
\end{align*}
Hence, $X\boldsymbol{1}_C$ has entries in $\{0,1\}$. 
Finally, observe that $X(j, i) = e_{j}^T X e_{i} = e_j^T X P e_i = e_j^T X \boldsymbol{1}_C/|C|$ is non-negative for  every $j \in J$.
This yields \cref{p1}. 

\Cref{p2} implies \cref{p3}: Let $t \in T(\Gamma)$. Then
\[
\soe \widehat{\phi}(t) 
= \left< \boldsymbol{1}_I, \widehat{\phi}(t) \right>
= \left< \boldsymbol{1}_J, X\widehat{\phi}(t) \right>
= \left< \boldsymbol{1}_J, \widehat{\psi}(t) \right>
= \soe \widehat{\psi}(t). 
\]

\Cref{p3} implies \cref{p2}: Write $V \leq \mathbb{K}^I$ and $W \leq \mathbb{K}^J$ for the vector spaced spanned by the $\widehat{\phi}(t)$ and $\widehat{\psi}(t)$ for $t \in T(\Gamma)$ respectively. Since for all $t, s \in T(\Gamma)$
\[
\left< \widehat{\phi}(t), \widehat{\phi}(s) \right>
= \soe(\widehat{\phi}(t) \odot \widehat{\phi}(s))
= \soe(\widehat{\phi}(t \odot s))
= \soe(\widehat{\psi}(t \odot s))
= \left< \widehat{\psi}(t), \widehat{\psi}(s) \right>,
\]
\Cref{lemma:gs} implies the existence of an orthogonal map $U \colon V \to W$ such that $U \widehat{\phi}(t) = \widehat{\psi}(t)$ for all $t \in T(\Gamma)$.
Extend $U$ to $X \in \mathbb{K}^{J \times I}$ by letting it annihilate $V^\perp$. This matrix satisfies all desired properties.
\end{proof}

In stark contrast to \cref{lem:pearcy,prop:soe-quadratic-length}, we give an exponential bound on the size of the trees in \cref{prop:soe-pos-exp-size}.
Lower bounds are discussed in \cref{rem:tight}.
First we bound the depth of the trees.
The \emph{depth} of a tree $t = (T, r, e) \in T(\Gamma_M)$ is defined as the maximal number of edges of any path on $T$ starting in~$r$.

\begin{lemma} \label{lem:soe-pos-exp-size}
	Writing $n \coloneqq |I| = |J|$,
	the conditions of \cref{thm:soe-pos} are equivalent to the following: For every tree $t \in T(\Gamma_{M})$ of depth $\leq n$, $\soe(t_{\matvec{A}}) = \soe(t_{\matvec{B}})$.
\end{lemma}
\begin{proof}
	For $d \geq 0$, write $T_{\matvec{A}}^{\leq d} \leq \mathbb{K}^I$ for the vector space spanned by the $t_{\matvec{A}}$ for all $t \in T(\Gamma_M)$ of depth $\leq d$.
	Write $T_{\matvec{A}}$ for the space spanned by all $t_{\matvec{A}}$ for $t \in T(\Gamma_M)$.
	Clearly, $T_{\matvec{A}}$ is at most $n$ dimensional. The space $T_{\matvec{A}}^{\leq 0}$ containing the all-ones vector is one-dimensional.
	\begin{claim} \label{cl:chain-tree}
		If $T_{\matvec{A}}^{\leq d} = T_{\matvec{A}}^{\leq d +1}$ for some $d \in \mathbb{N}$ then $T_{\matvec{A}} = T_{\matvec{A}}^{\leq d}$.
		In particular, $T_{\matvec{A}}^{\leq n - 1} = T_{\matvec{A}}$.
	\end{claim}
	\begin{claimproof}
		By induction on $j \geq 1$, it is shown that $T_{\matvec{A}}^{\leq d+j} \leq T_{\matvec{A}}^{\leq d}$.
		The base case $j=1$ holds by assumption.
		Let $t \in T(\Gamma_{M})$ be a tree of depth $d+j+1$. 
		If $t$ can be written as $t = t^1 \odot \dots \odot t^\ell$ for some trees $t^1,\dots, t^\ell \in T(\Gamma_M)$ whose roots have degree one then it suffices to show that $t^i_{\matvec{A}} \in T_{\matvec{A}}^{\leq d}$ for all $i \in [\ell]$ since $T_{\matvec{A}}^{\leq d}$ is closed under Schur products.
		Hence, it may be supposed that the root $r$ of $t$ has a single child $s$. Write $t'$ for the subtree of $t$ rooted at $s$ and $g \in \Gamma_M$ for the element associated with the edge $rs$.
		
		By assumption, there exist trees $x^1, \dots, x^m \in T(\Gamma_M)$ of depth at most $d$ and coefficients $\alpha_1,\dots, \alpha_r \in \mathbb{K}$ such that $t'_{\matvec{A}} = \sum_{i=1}^m \alpha_i x_{\matvec{A}}^i$. Then $t_{\matvec{A}} = \sum_{i=1}^m \alpha_i g_{\matvec{A}} x_{\matvec{A}}^i \in T_{\matvec{A}}^{\leq d+1}$.
		Thus, $T_{\matvec{A}}^{\leq d+j+1} \leq T_{\matvec{A}}^{\leq d +1} \leq T_{\matvec{A}}^{\leq d}$, as desired.
	\end{claimproof}
	
	By \cref{cl:chain-tree}, every $t_{\matvec{A}} \in T_{\matvec{A}}$ can be written as linear combination of some $t'_{\matvec{A}} \in T_{\matvec{A}}^{\leq n - 1}$. It remains to show that the coefficients in this linear combination are the same for $\matvec{A}$ and $\matvec{B}$.
	
	\begin{claim} \label{cl:reduce-depth}
		For every $t \in T(\Gamma_M)$, there exist coefficients $\alpha_1, \dots, \alpha_m \in \mathbb{K}$ 
		and trees $t^1, \dots, t^m \in T(\Gamma_M)$ of depth at most $n-1$
		such that
		\begin{equation} \label{eq:condTree}
			t_{\mathsf{A}} = \sum_i \alpha_i t^i_{\mathsf{A}} \quad \text{and} \quad
			t_{\mathsf{B}} = \sum_i \alpha_i t^i_{\mathsf{B}}.
		\end{equation}
	\end{claim}
	\begin{claimproof}
		By induction on the structure of $t$. If $t$ is the single vertex tree then the claim is vacuously true.
		If $t$ has a unique child, write $s$ for the subtree rooted at this child and $g \in \Gamma_M$ for the element associated to the edge~$ts$.
		Observe that $t_{\matvec A} = g_{\matvec A} s_{\matvec A}$
		and $t_{\matvec B} = g_{\matvec B} s_{\matvec B}$.
		The inductive hypothesis applies to $s$ yielding coefficients $\alpha_1, \dots, \alpha_m$ and trees $s^1, \dots, s^m \in T(\Gamma_M)$ of depth at most $n-1$ such that \cref{eq:condTree} holds.
		By \cref{cl:chain-tree}, for every $i \in [m]$, there exist coefficients $\beta_{ij}$ and trees $r^{ij}$ of depth at most $n-1$ such that
		$g_{\matvec{A}}s^i_{\matvec{A}} = \sum_j \beta_{ij} r^{ij}_{\matvec{A}}$.
		In order to conclude that the same identity holds for $\matvec{B}$,
		observe that 
		the tree represented by $g s^i$ has depth at most $n$.
		The same holds for all trees occurring in the following calculation:		
		\begin{align*}
			&\left< g_{\matvec{B}}s^i_{\matvec{B}} - \sum \beta_{ij} r^{ij}_{\matvec{B}}, g_{\matvec{B}}s^i_{\matvec{B}} - \sum \beta_{ij} r^{ij}_{\matvec{B}} \right> \\
			&= \soe(g_{\matvec{B}}s^i_{\matvec{B}} \odot g_{\matvec{B}}s^i_{\matvec{B}}) - 2 \sum \beta_{ij} \soe(g_{\matvec{B}}s^i_{\matvec{B}} \odot r^{ij}_{\matvec{B}})
			+ \sum \beta_{ij} \beta_{ik} \soe(r^{ij}_{\matvec{B}} \odot r^{ik}_{\matvec{B}}) \\
			&= \soe(g_{\matvec{A}}s^i_{\matvec{A}} \odot g_{\matvec{A}}s^i_{\matvec{A}}) - 2 \sum \beta_{ij} \soe(g_{\matvec{A}}s^i_{\matvec{A}} \odot r^{ij}_{\matvec{A}})
			+ \sum \beta_{ij} \beta_{ik} \soe(r^{ij}_{\matvec{A}} \odot r^{ik}_{\matvec{A}}) \\
				&=\left< g_{\matvec{A}}s^i_{\matvec{A}} - \sum \beta_{ij} r^{ij}_{\matvec{A}}, g_{\matvec{A}}s^i_{\matvec{A}} - \sum \beta_{ij} r^{ij}_{\matvec{A}} \right> \\			
			&= 0.
		\end{align*}
		Thus, $g_{\matvec{B}}s^i_{\matvec{B}} = \sum_j \beta_{ij} r^{ij}_{\matvec{B}}$, as desired.
		If $t$ is of the form $t^1 \odot \dots \odot t^r$ for some trees $t^1,\dots, t^r$ whose roots have degree one then the first case applies to each of these subtrees.
		The claim follows readily.
	\end{claimproof}
	Finally, for every tree $t \in T(\Gamma_M)$ of arbitrary depth, let  $\alpha_1, \dots, \alpha_m \in \mathbb{K}$ 
	and $t^1, \dots, t^m \in T(\Gamma_M)$  be as in \cref{cl:reduce-depth}. Then
	$
	\soe(t_\matvec{A}) = \sum \alpha_i \soe(t^i_\matvec{A}) = \sum \alpha_i \soe(t^i_\matvec{B}) = \soe(t_{\matvec{B}}).
	$
\end{proof}

We conclude this section by bounding the degree of the trees which need to be considered in \cref{thm:soe-pos}.
In particular, 
it suffices to consider trees on at most $\sum_{d = 0}^n (2n-1)^d 
\leq (2n)^{n+1}
$ vertices.
In \cref{rem:tight}, we comment on the tightness of the bounds in \cref{prop:soe-pos-exp-size}.

\begin{theorem} \label{prop:soe-pos-exp-size}
	Writing $n \coloneqq |I| = |J|$,
	the conditions of \cref{thm:soe-pos} are equivalent to the following: For every tree $t \in T(\Gamma_{M})$ of depth $\leq n$ and out-degree $\leq 2n-1$, $\soe(t_{\matvec{A}}) = \soe(t_{\matvec{B}})$.
\end{theorem}
\begin{proof}
	Given \cref{lem:soe-pos-exp-size}, it suffices to show that $\soe(t_{\matvec{A}}) = \soe(t_{\matvec{B}})$ for all trees $t \in T(\Gamma_{M})$ of depth at most~$n$.	
	To ease notation, we suppose wlog that $I$ and $J$ are disjoint.
	Similar to the set-up of the proof of \cref{lem:pearcy},
	define 
	for a tree $t \in T(\Gamma_M)$  the block vector $t_{\bothmat} \coloneqq \left( \begin{smallmatrix}
		t_{\matvec A} & t_{\matvec B}
	\end{smallmatrix}\right)^T \in \mathbb{K}^{I \cup J}$.
	We construct for every $d \geq 0$ a spanning set of vectors for the space $T_{\bothmat}^{\leq d}$ spanned by the $t_{\bothmat}$ for all $t \in T(\Gamma_M)$ of depth $\leq d$ and out-degree $\leq 2n-1$.
	For a tree $t \in T(\Gamma_M)$ and $i \geq 1$, write $t^{\odot i}$ for the tree obtained by gluing $i$ copies of $t$. Furthermore, let $t^{\odot 0} \coloneqq \boldsymbol{1}$.
	\begin{claim} \label{cl:tree-spanning}
		For every $d \geq 0$, there exist a set $T \subseteq  T(\Gamma_M)$ of trees of depth $\leq d$,
		whose roots have out-degree $\leq 1$, and all whose out-degrees are $\leq 2n-1$
		such that 
		\[
			T^{\leq d}_{\bothmat} = \spn \{t_\bothmat^{\odot i} \mid t \in T, 0 \leq i \leq 2n-1\}.
		\]
	\end{claim}
	\begin{claimproof}
		For $d = 0$, the singleton containing the one-vertex tree is as desired.
		For $d \geq 1$, 
		consider the following equivalence relation on $I \cup J$:
		Let $i \sim_d j$ if and only if $t_\bothmat(i) = t_\bothmat(j)$ for all $t \in T(\Gamma_M)$ of depth $\leq d$.
		Observe that if $i \not\sim_d j$ then there exists $s \in T(\Gamma_M)$ of depth $\leq d$ and with root of degree $1$ such that $s_\bothmat(i) \neq s_\bothmat(j)$.
		Indeed, if the root of a tree $t$ such that $t_\bothmat(i) \neq t_\bothmat(j)$ has higher degree then $t$ is the gluing product of multiple trees with root of degree one and one of these factors is as desired.
		
		Let $S \subseteq T(\Gamma_M)$ be a set of trees of depth at most~$d$ and with roots of degree~$1$ such that $i \sim_{d} j$ if and only if $s_\bothmat(i) = s_\bothmat(j)$ for all $s \in S$.
		By \cref{fact:vd},
		\[
		T^{\leq d}_{\bothmat} = \spn \{s_\bothmat^{\odot i} \mid s \in S, 0 \leq i \leq 2n-1\}.
		\]
		For $s \in S$, write $s'$ for the tree rooted at the unique child of the root of $s$. Moreover, write $g^s \in \Gamma_M$ for the element associated to the edge incident to the root of $s$.
		The tree $s'$ is of depth at most~$d-1$.
		Write $Q \subseteq T(\Gamma_M)$ for the set of trees of depth at most $d-1$ with roots of out-degree $\leq 1$, and all whose out-degrees are $\leq 2n-1$, which is guaranteed to exist by induction. Then, by linearity,
		\begin{align*}
			T^{\leq d}_{\bothmat} &= \spn \{s_\bothmat^{\odot i} \mid s \in S, 0 \leq i \leq 2n-1\} \\
			& = \spn \{(g^ss')_\bothmat^{\odot i} \mid s \in S, 0 \leq i \leq 2n-1\} \\
			& \leq \spn \{(g( q^{\odot j}))_\bothmat^{\odot i} \mid q \in Q, g \in \Gamma_M, 0 \leq i \leq 2n-1, 0 \leq j \leq 2n-1\}\\
			& \leq T^{\leq d}_{\bothmat}.
		\end{align*}
		Note that all trees $g(q^{\odot j})$ appearing in the final set are of depth $\leq d$, have roots of out-degree $\leq 1$, and only vertices of out-degree $\leq 2n-1$.
	\end{claimproof}
	
		For a vector $v \in  \mathbb{K}^{I \cup J}$, write $\soe_{\matvec{A}} v \coloneqq \sum_{i \in I} v(i)$
	and analogously $\soe_{\matvec{B}} v \coloneqq \sum_{j \in J} v(j)$.
	Let $t \in T(\Gamma_{M})$ be a tree of depth at most~$n$.
	By \cref{cl:tree-spanning}, there exist $t^1, \dots, t^r \in \Gamma_M$ of depth at most~$n$ and out-degrees at most~$2n-1$ and coefficients $\alpha_1, \dots, \alpha_r \in \mathbb{K}$ such that
	\( t_{\bothmat} = \sum_{i=1}^r \alpha_i t^i_{\bothmat}\).
	Hence,
	\[
	\soe(t_{\matvec{A}}) 
	= \soe_{\matvec{A}}(t_\bothmat) 
	= \sum \alpha_i\soe_{\matvec{A}}(t^i_{\bothmat})
	= \sum \alpha_i\soe(t^i_{\matvec{A}})
	= \sum \alpha_i\soe(t^i_{\matvec{B}})
	= \soe(t_{\matvec{B}}). 
	\]
	Thus, the assertion in \cref{prop:soe-pos-exp-size} implies the assertion of \cref{lem:soe-pos-exp-size}.
\end{proof}

\section{Bilabelled Graphs and Homomorphism Tensors}

In this section, bilabelled graphs and their homomorphism tensors are introduced.
These are the objects to which the results from \cref{sec:three-variants} are applied in the subsequent sections.
Bilabelled graphs and homomorphism tensors provide a language for constructing more complicated graphs from simple building blocks while keeping track of homomorphism counts.
They have also been used in the seminal paper \cite{mancinska_quantum_2020}.

\subsection{Bilabelled Graphs and Combinatorial Operations}

\label{sec:lblgraphs}
For $\ell \in \mathbb{N}$, an \emph{$\ell$-labelled}\footnote{Note that this definition is dual to what what is known as a  \emph{graph labelling}, where labels are assigned to vertices of a graph \cite{gallian_graph_2018}, rather than vertices to labels as here. However, our definition is in line with \cite{mancinska_quantum_2020,lovasz_operations_1967}.} graph $\boldsymbol{F}$ is a tuple $\boldsymbol{F} = (F, \vec{v})$ where $F$ is a graph and $\vec{v} \in V(F)^\ell$. The vertices in $\vec v$ are not necessarily distinct, i.e.\@ vertices may have several labels.
Write $\mathcal{G}(\ell)$ for the class of all $\ell$-labelled graphs.

The operation of \emph{gluing} two $\ell$-labelled graphs
$\boldsymbol{F} =(F,\vec{u})$ and $\boldsymbol{F}' = (F', \vec{u}')$ yields the $\ell$-labelled graph $\boldsymbol{F} \odot \boldsymbol{F}'$ obtained by taking the disjoint union of $F$ and $F'$ and pairwise identifying the vertices $u_i$ and $v_i$ to become the $i$-th labelled vertex, for $i \in [\ell]$, and removing any multiedges in the process. 
In fact, since we consider homomorphisms into simple graphs, multiedges can always be omitted. 
Likewise, self-loops can also be disregarded since the number of homomorphisms $F \to G$ where $F$ has a self-loop, and $G$ does not is always zero. 
We henceforth tacitly assume that all graphs are simple.

For $\ell_1, \ell_2 \in \NN$, an \emph{$(\ell_1,\ell_2)$-bilabelled graph} $\boldsymbol{F}$ is a tuple $(F, \vec{u}, \vec{v})$ for $\vec{u} \in V(F)^{\ell_1}$, $\vec{v} \in V(F)^{\ell_2}$. If $\vec{u} = (u_1, \dots, u_{\ell_1})$ and $\vec{v} = (v_1, \dots, v_{\ell_2})$, it is usual to say that the vertex $u_i$, resp.\@ $v_i$, is labelled with the $i$-th \emph{in-label}, resp.\@ \emph{out-label}.
Write $\mathcal{G}(\ell_1, \ell_2)$ for the class of all $(\ell_1, \ell_2)$-bilabelled graphs.

The \emph{reverse} of an $(\ell_1, \ell_2)$-bilabelled
graph $\boldsymbol{F} =(F,\vec{u}, \vec{v})$ is defined to be the $(\ell_2,\ell_1)$-bilabelled graph $\boldsymbol{F}^* = (F, \vec{v}, \vec{u})$ with roles of in- and out-labels interchanged. 
The \emph{concatenation} or \emph{series composition} of an $(\ell_1, \ell_2)$-bilabelled graph $\boldsymbol{F} = (F, \vec{u}, \vec{v})$ and an $(\ell_2, \ell_3)$-bilabelled graph $\boldsymbol{F}' = (F', \vec{u}', \vec{v}')$, $\ell_3 \in \mathbb{N}$, denoted by $\boldsymbol{F} \cdot \boldsymbol{F}'$ is the $(\ell_1, \ell_3)$-bilabelled graph obtained by taking the disjoint union of $F$ and $F'$ and identifying for all $i \in [\ell_2]$ the vertices $v_i$ and $u'_i$. The in-labels of $\boldsymbol{F} \cdot \boldsymbol{F}'$ lie on $\vec{u}$ while its out-labels are positioned on $\vec{v}'$. 
The \emph{parallel composition} of $(\ell_1, \ell_2)$-bilabelled graphs $\boldsymbol{F} = (F, \vec{u}, \vec{v})$ and $\boldsymbol{F}' = (F', \vec{u}', \vec{v}')$ denoted by $\boldsymbol{F} \odot \boldsymbol{F}'$ is obtained by taking the disjoint union of $F$ and $F'$ and identifying $u_i$ with $u'_i$, and $v_j$ with with $v'_j$ for $i \in [\ell_1]$ and $j \in [\ell_2]$.

\subsection{Homomorphism Tensors and Homomorphism Tensor Maps}

For graphs $F$ and $G$, let $\hom(F, G)$ denote the number of homomorphisms from $F$ to~$G$, i.e.\@ the number of mappings $h \colon V(F) \to V(G)$ such that $v_1v_2 \in E(F)$ implies $h(v_1)h(v_2) \in E(G)$. 
For a graph class $\mathcal{F}$ and graphs $G$ and $H$, write $G \equiv_{\mathcal{F}} H$ if $G$ and $H$ are \emph{homomorphism indistinguishable over $\mathcal{F}$}, i.e.\@ $\hom(F, G) = \hom(F, H)$ for all $F \in \mathcal{F}$.

For an $\ell$-labelled graph $\boldsymbol{F} = (F, \vec{v})$ and $\vec{w} \in V(G)^\ell$, let $\hom(\boldsymbol{F},G,\vec{w})$ denote the number of homomorphisms~$h$ from $F$ to $G$ such that $h(v_i) = w_i$ for all $i \in [\ell]$. Analogously, for an $(\ell_1,\ell_2)$-bilabelled graph $\boldsymbol{F}' = (F', \vec{u},\vec{v})$ and $\vec{x} \in V(G)^{\ell_1}$, $\vec{y} \in V(G)^{\ell_2}$, let $\hom(F',G,\vec{x},\vec{y})$ denote the number of homomorphisms $h \colon F' \to G$ such that $h(u_i) = x_i$  and $h(v_j) = y_j$ for all $i \in [\ell_1]$, $j \in [\ell_2]$.
More succinctly, we write $\boldsymbol{F}_G \in \mathbb{N}^{V(G)^\ell}$ for the \emph{homomorphism tensor} defined by letting $\boldsymbol{F}_G(\vec{w}) \coloneqq \hom(\boldsymbol{F},G,\vec{w})$ for all $\vec{w} \in V(G)^\ell$. 
Similarly, for a bilabelled graph~$\boldsymbol{F}'$, $\boldsymbol{F}'_G \in \mathbb{N}^{V(G)^{\ell_1} \times V(G)^{\ell_2}}$ is the matrix defined as 
$\boldsymbol{F}'_G(\vec{x},\vec{y}) \coloneqq \hom(\boldsymbol{F},G,\vec{x},\vec{y})$ for all $\vec{x} \in V(G)^{\ell_1}$, $\vec{y} \in V(G)^{\ell_2}$. 

Letting this construction range over all right-hand side graphs $G$, the map $G \mapsto \boldsymbol{F}_G$ becomes a tensor map, the \emph{homomorphism tensor map} induced by $\boldsymbol{F}$. It is easy to see that homomorphism tensor
maps are equivariant. 

	\begin{figure}
	\centering
	\begin{subfigure}[t]{.4\linewidth}
		\centering
		\includegraphics[page=16,scale=.9]{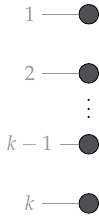}
		\caption{$k$-labelled graph $\boldsymbol{1} = \boldsymbol{1}^k$.}
		\label{fig:identity-graph-labelled}
	\end{subfigure}
	\begin{subfigure}[t]{.4\linewidth}
		\centering
		\includegraphics[page=17,scale=.9]{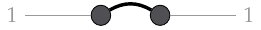}
		\caption{$(1,1)$-bilabelled graph $\boldsymbol{A}$.}
		\label{fig:adjacency11}
	\end{subfigure}
	\caption[The (bi)labelled graphs from \cref{ex:adjacency}]{The (bi)labelled graphs from \cref{ex:adjacency}  in wire notation of \cite{mancinska_quantum_2020}: A vertex carries in-label (out-label) $i$ if it is connected to the number $i$ on the left (right) by a wire. Actual edges and vertices of the graph are depicted in black.}
\end{figure}

\begin{example} \label{ex:adjacency}
	For $k \geq 1$, let $\boldsymbol{1}^k$ denote the labelled graph consisting of $k$~isolated vertices with distinct labels $(1, \dots, k)$.
	Then, $\boldsymbol{1}_G^k$ is the uniform tensor in $\mathbb{K}^{V(G)^k}$ where every entry is equal to $1$.
Let $\boldsymbol{A}$ denote the $(1,1)$-bilabelled graph consisting of a single edge whose endpoints each carry one label. For every graph $G$, the matrix $\boldsymbol{A}_G$ is the adjacency matrix of~$G$.
\end{example}

For $\mathbb{K} \in \{\mathbb{Q,R,C}\}$,
homomorphism tensors give rise to the $\mathbb{K}$-vector spaces of our main interest and their endomorphisms. For a set $\mathcal{R}$ of $\ell$-labelled graphs, the tensors $\boldsymbol{R}_G$ for $\boldsymbol{R} \in \mathcal{R}$ span a subspace of~$\mathbb{K}^{V(G)^\ell}$, which is denoted by $\mathbb{K}\mathcal{R}_G$. 
Moreover, the tensors $\boldsymbol{S}_G$ for an $(\ell,\ell)$-bilabelled graph~$\boldsymbol{S}$ induces an endomorphisms of $\mathbb{K}^{V(G)^\ell}$.

		\subsection{Algebraic and Combinatorial Operations on Homomorphism Tensor Maps}
		\label{sec:opcorrespond}

		As outlined in \cref{sec:tensors}, tensor maps naturally admit a plenitude of algebraic operations.
		Crucially, many operations when applied to homomorphism tensor maps correspond to operations on (bi)labelled graphs.
		This observation due to \cite{lovasz_contractors_2009,mancinska_quantum_2020} is illustrated by the following examples.
		\begin{description}[parsep=0pt]
			\item[Sum-of-Entries and dropping labels]  Given a $k$-labelled graph $\boldsymbol{F} = (F, \vec{u})$, let $\soe(\boldsymbol{F})$ denote the $0$-labelled graph $(F, ())$. Then, for every graph $G$, $\soe(\boldsymbol{F})_G  = \hom(F, G) = \sum_{\vec{v} \in V(G)^k} \boldsymbol{F}_G(\vec{v}) = \soe(\boldsymbol{F}_G)$.
			For an example, see \cref{fig:soe}.
			\item[Matrix Product and series composition] Let an $(\ell_1, \ell_2)$-bilabelled graph $\boldsymbol{F} = (F, \vec{u}, \vec{v})$ and an $(\ell_2, \ell_3)$-bilabelled graph $\boldsymbol{F}' = (F', \vec{u}', \vec{v}')$ be given. Then for every graph $G$, vertices $\vec{x} \in V(G)^{\ell_1}$, and $\vec{y} \in V(G)^{\ell_3}$,  $(\boldsymbol{F} \cdot \boldsymbol{F}')_G(\vec{x},\vec{y}) = \sum_{\vec{w} \in V(G)^{\ell_2}} \boldsymbol{F}_G(\vec{x}, \vec{w}) \boldsymbol{F}'_G(\vec{w}, \vec{y}) = (\boldsymbol{F}_G \cdot \boldsymbol{F}'_G)(\vec{x}, \vec{y})$.
			A similar operation corresponds to the matrix-vector product, where $\boldsymbol{F}'$ is assumed to be $\ell_2$-labelled.
			For an example, see \cref{fig:seriesprod}.
			\item[Schur Product and gluing] The gluing product $\boldsymbol{F} \odot \boldsymbol{F}'$ of two $k$-labelled graphs $\boldsymbol{F} = (F, \vec{u})$ and $\boldsymbol{F}' = (F', \vec{u}')$ corresponds to the Schur product of the homomorphism tensors. That is, for every graph $G$ and $\vec{v} \in V(G)^k$, $(\boldsymbol{F} \odot \boldsymbol{F}')_G(\vec{v}) = \boldsymbol{F}_G(\vec{v}) \boldsymbol{F}'_G(\vec{v}) = (\boldsymbol{F}_G \odot \boldsymbol{F}'_G)(\vec{v})$.
			Moreover, the \emph{inner-product} of $\ell$-labelled graphs $\boldsymbol{F}$, $\boldsymbol{F}'$ can be defined by $\left< \boldsymbol{F}, \boldsymbol{F}' \right> = \soe(\boldsymbol{F} \odot \boldsymbol{F}')$. It corresponds to the standard inner-product on the tensor space.
			\item[Tensor product and disjoint union] For a $k$-labelled graph $\boldsymbol{F}=(F,(u_1,\ldots,u_k))$ and an $\ell$-labelled graph $\boldsymbol{F'} =(F',(v_1,\ldots,v_\ell))$, the tensor map $\boldsymbol{F} \otimes \boldsymbol{F}'$ is the homomorphism tensor map corresponding to the $(k+\ell)$-labelled graph $(F \otimes F',(u_1,\ldots,u_k,v_{1},\ldots,v_\ell))$, where $F\otimes F'$ is the disjoint union of $F$ and $F'$. 
			\item[Traces and identifying and dropping labels] Given a $(k,k)$-bilabelled graph $\boldsymbol{F} = (F, \vec{u}, \vec{v})$, let $\tr(\boldsymbol{F})$ denote the $0$-labelled graph obtained from $\boldsymbol{F}$ by identifying $u_i$ with $v_i$ for $i \in [k]$ and dropping the labels. 
			Then, for every graph $G$, $\tr(\boldsymbol{F})_G = \hom(\tr \boldsymbol{F}, G) =  \sum_{\vec{v} \in V(G)^k} \boldsymbol{F}_G(\vec{v},\vec{v})$.	
			For an example, see \cref{fig:traces}.
		\end{description}

\section{Cycles, Paths, and Trees}

Two graphs $G$ and $H$ with adjacency matrices $\boldsymbol{A}_G$ and $\boldsymbol{A}_H$ are isomorphic if and only if there is a matrix~$X$ over the non-negative integers such that $X\boldsymbol{A}_G=\boldsymbol{A}_H X$ and $X\boldsymbol{1}=X^T\boldsymbol 1=\boldsymbol 1$, where $\boldsymbol{1}$ is the all-ones vector. Writing the constraints as linear equations whose variables are the entries of $X$, we obtain a system $\Fiso(G, H)$ that has a non-negative integer solution if and only if $G$ and $H$ are isomorphic. A combination of results from  \cite{tinhofer_note_1991,dvorak_recognizing_2010} shows that $\Fiso(G, H)$ has a non-negative rational solution if and only if $G$ and $H$ are homomorphism indistinguishable over the class of trees, and by \cite{dell_lovasz_2018}, 
$\Fiso(G, H)$ has an arbitrary rational solution if and only if $G$ and $H$ are homomorphism indistinguishable over the class of paths.

In this section, we reprove these results showcasing our \cref{thm:wiegmann,thm:soe,thm:soe-pos}. These theorems contain the algebraic core of the arguments while the correspondence between (bi)labelled graphs and their homomorphism tensors (\cref{sec:opcorrespond}) provides the necessary combinatorial insights.

\begin{figure}
\begin{subfigure}{\linewidth}
	\[
		\raisebox{-.2\height}{\includegraphics[page=17,scale=.9]{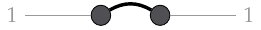}} \cdot 
		\raisebox{-.2\height}{\includegraphics[page=18,scale=.9]{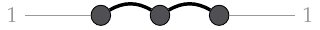}} = 
		\raisebox{-.2\height}{\includegraphics[page=19,scale=.9]{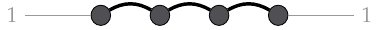}}
	\]
	\caption{Series products of bilabelled graphs correspond to matrix products of their homomorphism tensors. The left most $(1,1)$-bilabelled graph is $\boldsymbol{A}$, the one whose homomorphism tensor is the adjacency matrix. 
	The $(1,1)$-bilabelled paths with labels at vertices of degree at most $1$ form an involution monoid.}
	\label{fig:seriesprod}
\end{subfigure}
\begin{subfigure}{\linewidth}
	\[
	\tr\left(
	\raisebox{-.2\height}{\includegraphics[page=19,scale=.9]{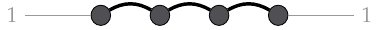}} 
	\right)
	=
	\raisebox{-.4\height}{\includegraphics[page=21,scale=.9]{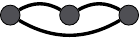}}
	\]
	\caption{Identifying opposing labels and unlabelling of a bilabelled graph corresponds to taking the trace of its homomorphism tensor.}
	\label{fig:traces}
\end{subfigure}
\begin{subfigure}{\linewidth}
	\[
	\soe\left(
	\raisebox{-.2\height}{\includegraphics[page=19,scale=.9]{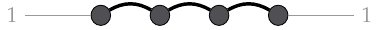}} 
	\right)
	=
	\raisebox{-.2\height}{\includegraphics[page=20,scale=.9]{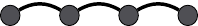}}
	\]
	\caption{Unlabelling a bilabelled graph corresponds to taking the sum-of-entries of its homomorphism tensor.}
	\label{fig:soe}
\end{subfigure}
\caption{Combinatorial operations on bilabelled graphs.}
\end{figure}

\begin{corollary}\label{cor:cycles}
	For graphs $n$-vertex simple graphs $G$ and $H$, the following are equivalent:
	\begin{enumerate}
		\item $G$ and $H$ are homomorphism indistinguishable over the class of cycles,
		\item $G$ and $H$ are homomorphism indistinguishable over the class of cycles on at most $2n^2-1$~vertices,
		\item there exists an orthogonal $X \in \mathbb{R}^{V(H) \times V(G)}$ such that $X\boldsymbol{A}_G = \boldsymbol{A}_H X$.
	\end{enumerate}
\end{corollary}
\begin{proof}
Apply \cref{thm:wiegmann} with $I \coloneqq V(G)$ and $J \coloneqq V(H)$, and $\matvec{A} \coloneqq (\boldsymbol{A}_G)$ and $\matvec{B} \coloneqq (\boldsymbol{A}_H)$. 
The series products of the $(1,1)$-bilabelled edge $\boldsymbol{A}$, cf.\ \cref{ex:adjacency}, with itself are precisely the $(1,1)$-bilabelled paths with labels at the vertices of degree~$\leq 1$, cf.\@ \cref{fig:seriesprod}. 
Taking the traces of their homomorphism matrices amounts to identifying the labels of these paths and thus counting homomorphisms from cycles into $G$ and $H$, cf.\@ \cref{fig:traces}.

By \cref{lem:pearcy}, it suffices to consider words in $\boldsymbol{A}_G$ and $\boldsymbol{A}_H$ of length at most $2n^2-1$. Each letter corresponds to an edge. 
Thus, homomorphism counts from cycles on at most $2n^2-1$ vertices suffice.
\end{proof}
By Newton's identities, cf.\ \cite[Proposition~1]{dawar_descriptive_2019}, considering cycles on at most~$n$ vertices suffices.
The bound in \cref{cor:cycles} is suboptimal as it is derived from the more general \cref{lem:pearcy}, which contrary to Newton's identities gives a criterion of simultaneous orthogonal similarity of multiple matrices, cf.\ \cite{pappacena_upper_1997}.
For paths, we obtain an analogous result:

\begin{corollary} \label{cor:paths}
	For $n$-vertex  simple graphs $G$ and $H$, the following are equivalent:
	\begin{enumerate}
		\item $G$ and $H$ are homomorphism indistinguishable over the class of paths,
		\item $G$ and $H$ are homomorphism indistinguishable over the class of paths on at most $2n$~vertices,
		\item there exists a pseudo-stochastic $X \in \mathbb{Q}^{V(H) \times V(G)}$ such that $X\boldsymbol{A}_G = \boldsymbol{A}_H X$.
	\end{enumerate}
\end{corollary}
\begin{proof}
Recall the  proof of \cref{cor:cycles} and invoke \cref{thm:soe}. Taking sums-of-entries of homomorphism matrices of $(1,1)$-bilabelled paths amounts to counting homomorphisms from the underlying unlabelled paths into $G$ and $H$, cf.\@ \cref{fig:soe}.
By \cref{prop:soe-quadratic-length}, it suffices  to consider words in $\boldsymbol{A}_G$ and $\boldsymbol{A}_H$ of length at most $2n-1$.
Each letter corresponds to an edge.
Thus, homomorphism counts from paths on at most $2n$ vertices suffice.
\end{proof}

The classical characterisation \cite{tinhofer_graph_1986} of homomorphism indistinguishability over trees involves a non-negativity condition on the matrix $X$.
While such an assumption appears natural from the viewpoint of solving the system of equations for fractional isomorphism, it lacks an algebraic or combinatorial interpretation. Using \cref{thm:soe-pos}, we reprove this known characterisation and give an alternative description that emphasises its graph-theoretic origin.
Here, the \emph{depth} of a rooted tree $(T,r)$, $r \in V(T)$,  is the maximum number of edges on any path from $r$ to a leaf.
\begin{corollary} \label{cor:trees}
	For $n$-vertex simple graphs $G$ and $H$, the following are equivalent:
	\begin{enumerate}
		\item $G$ and $H$ are homomorphism indistinguishable over the class of trees,
		\item $G$ and $H$ are homomorphism indistinguishable over all trees $T$ for which there exists $r \in V(T)$ such that $(T, r)$ is of depth at most~$n$ and maximum out-degree at most~$2n-1$,
		\item $G$ and $H$ are homomorphism indistinguishable over all trees on at most~$(2n)^{n+1}$ vertices,
		\item there exists a pseudo-stochastic matrix $X \in \mathbb{Q}^{V(H) \times V(G)}$ satisfying $X \boldsymbol{A}_G = \boldsymbol{A}_H X$  and one of the following equivalent conditions holds:
		\begin{enumerate}
			\item all entries of $X$ are non-negative,\label{trees}
			\item $X\boldsymbol{T}_G = \boldsymbol{T}_H$ for all $1$-labelled trees $\boldsymbol{T} \in \mathcal{T}$,\label{trees1}
			\item $X$ preserves the Schur product on $\mathbb{R}\mathcal{T}_G$, the space spanned by the  $\boldsymbol{T}_G$ for $\boldsymbol{T} \in \mathcal{T}$, i.e.\@ $X(u \odot v) = (Xu) \odot (Xv)$ for all $u, v \in \mathbb{R}\mathcal{T}_G$.\label{trees2}
		\end{enumerate}
	\end{enumerate}
\end{corollary}
\begin{proof}
	By \cref{prop:soe-pos-exp-size}, the first two assertions are equivalent. 
	A tree as in the second assertion has at most $\sum_{d=0}^n (2n-1)^d \leq (2n)^{n+1}$ vertices.
	Hence, the third assertion implies the second.
	Clearly, the first assertion implies the third.
	
	For the last assertion, consider the following argument:
	The equivalence of \cref{trees,trees1} is immediate from \cref{lem:invmodsoe-pos}. 
	Assuming \cref{trees1}, \cref{trees2}  follows since $X(\boldsymbol{T}_G \odot \boldsymbol{S}_G) = X((\boldsymbol{T} \odot \boldsymbol{S})_G) = (\boldsymbol{T} \odot \boldsymbol{S})_H = \boldsymbol{T}_H \odot \boldsymbol{S}_H$ for all $\boldsymbol{T}, \boldsymbol{S} \in \mathcal{T}$. 
	Conversely, by induction on the structure of $\boldsymbol{T} \in \mathcal{T}$, if $\boldsymbol{T} = \boldsymbol{A} \cdot \boldsymbol{S}$  for some $\boldsymbol{S} \in \mathcal{T}$  then $X \boldsymbol{T}_G = \boldsymbol{A}_H  X  \boldsymbol{S}_G = \boldsymbol{T}_H$ by the assumption $X \boldsymbol{A}_G = \boldsymbol{A}_H X$.
	If $\boldsymbol{T} = \boldsymbol{R} \odot \boldsymbol{S}$ for some $\boldsymbol{R}, \boldsymbol{S} \in \mathcal{T}$ then the claim follows immediately from \cref{trees2}.
\end{proof}

We finally comment on the optimality of the bounds in \cref{cor:trees}.

	\begin{remark} \label{rem:tight}
	In \cref{sec:bounded-degree-trees}, we argue that homomorphism counts of constant degree trees are not as expressive as homomorphism counts from all trees.
	In particular, by \cref{thm:degreeGH}, the bound on the maximum degree in \cref{cor:trees} cannot be replaced with a constant.
	Furthermore, by \cite{furer_weisfeiler-lehman_2001}, there exist graphs $G$ and $H$ on $n$~vertices which are distinguished by Colour Refinement but only in $\Theta(n)$ iterations.
	Thus, by \cite{dell_lovasz_2018}, these graphs are homomorphism indistinguishable over all trees $T$ for which there exists $r \in V(T)$ such that the rooted tree $(T, r)$ is of depth $\Theta(n)$.
	Thereby, the bound in \cref{cor:trees} on the depth of the trees is asymptotically tight.
\end{remark}

\section{Cyclewidth, Pathwidth, Treewidth, Treedepth, and Trees of Bounded Degree}
\label{sec:graphclasses}

In \cref{cor:cycles,cor:paths,cor:trees}, the machinery from \cref{sec:three-variants} was applied to involution monoids which are generated by a single non-trivial generator, namely the bilabelled edge, cf.\ \cref{fig:adjacency11}.
In this section, we consider involution monoids which are generated by more than one non-trivial generator.
In the language of \cref{thm:wiegmann,thm:soe,thm:soe-pos}, this amounts to considering multiple matrices.
As before, the matrices are homomorphism matrices of bilabelled graphs. 
Using multiple such graphs permits the treatment of more complicated graph classes such the classes of graphs of bounded cycle-, path-, treewidth, and treedepth.

The following subsections feature four different algebro-combinatorial setups, which are summarised in \cref{fig:interplay}. The algebraic structure of the considered class of (bi)labelled graphs determines the domain of the matrix variables in the matrix equations whose feasibility is equivalent to homomorphism indistinguishability over the family of underlying unlabelled graphs.
Domains covered by our results are unitary matrices, pseudo-stochastic matrices, and doubly stochastic matrices.
In some cases, feasibility over two of these possible domains coincides (\cref{sec:treedepth}).

\begin{figure}
	\centering
	\begin{tikzpicture}
		\draw [rounded corners, draw=lightgray, dashed, ultra thick] (-6, 2) rectangle (-.1, -4.5);
		\draw [rounded corners, draw=lightgray, dashed, ultra thick] (6, 2) rectangle (.1, -4.5);
		\node [align=center] at (-3, 1.5) {\textbf{pseudo-stochastic solutions}};
		\node [align=center] at (3, 1.5) {\textbf{doubly stochastic solutions}};
		
		\draw [rounded corners, draw=lightgray, dashed, ultra thick] (3, 2.3) rectangle (-3, 4);
		\node [align=center] at (0, 3.7) {\textbf{orthogonal solutions}};
		\node (cyc) [text width=6cm, align=center] at (0, 2.9) {cyclic $\tr(\mathcal{S})$ \\ \smaller e.g.\  cycles, cycle\-width  (\cref{sec:path-cycle})};
		
		\node (e) [text width=5cm, align=center, fill=white] at (0, .5) {$\mathcal{R} = \mathcal{S}\boldsymbol{1}$ gluing-closed  \\ \smaller e.g.\  treedepth  (\cref{sec:treedepth})};
		\node (a) [text width=5cm, align=center] at (-3, -1.5) {$\mathcal{R} = \mathcal{S}\boldsymbol{1}$ \\ \smaller e.g.\  paths, pathwidth (\cref{sec:path-cycle}) };
		\node (c) [text width=5cm, below of=a, align=center, yshift=-1cm] {$\mathcal{R}$ inner-product compatible \\ \smaller e.g.\  bounded degree trees  (\cref{sec:bounded-degree-trees})};
		
		\node (b) [text width=5cm, align=center] at (3, -1.5) {$\mathcal{R}$ agglutinatively generated by $\mathcal{S}$ \\ \smaller e.g.\  trees, treewidth (\cref{sec:trees})};
		\node (d) [text width=5cm, below of=b, yshift=-1cm, align=center] {$\mathcal{R}$ gluing-closed \\ \smaller e.g.\  Eilenberg--Moore categories of comonads  (\cref{sec:treedepth})};
		\draw [->, thick] (a) -- (c);
		\draw [->, thick] (d) -- (c);
		\draw [->, thick] (b) -- (d);
		\draw [->, thick] (e) -- (a); 
		\draw [->, thick] (e) -- (b);
	\end{tikzpicture}
	\caption[Interplay of a family of labelled graphs $\mathcal{R} \subseteq \mathcal{G}(k)$ and a family of bilabelled graphs~$\mathcal{S} \subseteq \mathcal{G}(k,k)$.]{Interplay of a family of labelled graphs $\mathcal{R} \subseteq \mathcal{G}(k)$ and a family of bilabelled graphs~$\mathcal{S} \subseteq \mathcal{G}(k,k)$ yielding matrix equations with orthogonal, pseudo-stochastic, or doubly stochastic solutions. Arrows indicate implications, e.g.\ every gluing-closed family of labelled graphs $\mathcal{R}$ is inner-product compatible.}
	\label{fig:interplay}
\end{figure}

\subsection{Preliminaries: Tree, Path, and Cycle Decompositions}

In preparation for the following sections,
we recall the well-known notion of a tree decomposition in the following slightly more general form.

\begin{definition} \label{def:decomposition}
	Let $F$ be a graph.
	An \emph{$F$-decomposition} of a graph $G$ is a pair $(F, \beta)$ and $\beta$ is map $V(F) \to 2^{V(G)}$ such that
	\begin{enumerate}
		\item the union of the $\beta(v)$ for $v \in V(F)$ is equal to $V(G)$,
		\item for every edge $e \in E(G)$ there exists $v \in V(F)$ such that $e \subseteq \beta(v)$,
		\item for every vertex $u \in V(G)$ the set of vertices $v \in V(F)$ such that $u \in \beta(v)$ is connected in~$F$.
	\end{enumerate}
\end{definition}
The sets $\beta(v)$ for $v \in V(F)$ are called the \emph{bags of $(F, \beta)$}.
The \emph{width} of $(F, \beta)$ is the maximum over all $\abs{\beta(v)} - 1$ for $v \in V(F)$.
An $F$-decomposition is called a \emph{tree decomposition} if $F$ is a tree, a \emph{path decomposition} if $F$ is a path, and a \emph{cycle decomposition} if $F$ is a cycle.\footnote{For convenience, we regard the cliques $K_1$ and $K_2$ as cycles.}
The \emph{tree-} / \emph{path-} / \emph{cyclewidth} of a graph $G$ is the minimum width of a tree/path/cycle decomposition of~$G$.
The following \cref{lem:bodlaender8} generalises \cite[Lemma~8]{bodlaender_partial_1998}.

\begin{lemma} \label{lem:bodlaender8}
	Let $k \geq 1$ and $F$ be a connected graph.
	If a graph $G$ has an $F$-decomposition of width at most $k$ and $|V(G)| \geq k+1$
	then there is an $F'$-decomposition $\beta \colon F' \to 2^{V(G)}$ of $G$ such that
	\begin{enumerate}
		\item $|\beta(t)| = k+1$ for all $t \in V(F')$, and
		\item $|\beta(s) \cap \beta(t)| = k$ for all $st \in E(F')$.
	\end{enumerate}
	The graph $F'$ can be obtained from $F$ by contracting and/or subdividing edges.
\end{lemma}
\begin{proof}
	If $|V(G)| = k+1$ then $F'$ can be taken to be the single vertex graph.
	If $|V(G)| > k+1$ then $F$ must contain at least one edge.
	Let $\beta \colon F \to 2^{V(G)}$ be the $F$-decomposition of width at most $k$.
	We repeatedly apply the following steps:
	\begin{itemize}
		\item If $st \in E(F)$ is such that $\beta(s) \subseteq \beta(t)$ or $\beta(t) \subseteq \beta(s)$ then the edge $st$ in $F$ can be contracted and the set $\beta(s) \cup \beta(t)$  can be taken to be the bag at the vertex obtained by contraction.
		\item If $st \in E(F)$ and $|\beta(s)| < k+1$ and $\beta(t) \not\subseteq \beta(s)$ then $\beta(s)$ can be enlarged by a vertex $v \in \beta(t) \setminus \beta(s)$.
		\item If $st \in E(F)$ and $|\beta(s)| = |\beta(t)| = k+1$ and $|\beta(s) \cap \beta(t)| < k$ then subdivide the edge $st$ in $F$ by introducing a fresh vertex $r$.
		Choose vertices $v \in \beta(s) \setminus \beta(t)$ and $w \in \beta(t) \setminus \beta(s)$ and let $\beta(r) \coloneqq (\beta(s) \setminus \{v\}) \cup \{w\}$.
	\end{itemize}
	If none of these operations can be applied, the decomposition is as desired.
\end{proof}

Since the class of trees (paths) is closed under contracting and subdividing edges, tree (path) decompositions satisfying the assertions of \cref{lem:bodlaender8} can be found for all graphs of bounded treewidth (pathwidth).

\subsection{Pathwidth and Cyclewidth: Generators for Involution Monoids}
\label{sec:path-cycle}

The families of bilabelled graphs considered in this section all are involution monoids in the following sense.
For $k \geq 1$, the class $\mathcal{G}(k,k)$ of all $(k,k)$\nobreakdash-bilabelled graphs forms an involution monoid whose binary operation is series composition, whose involution operation is reversal, and whose neutral element is the \emph{identity graph} $\boldsymbol{I} = (I, (1,\dots, k), (1, \dots, k))$ with $V(I) = [k]$ and $E(I) = \emptyset$, cf.\@ \cref{fig:identitygraph}.
This section features  subclasses  $\mathcal{S} \subseteq \mathcal{G}(k,k)$ 
which also form  involution monoids.
That is, they satisfy the following properties:
\begin{enumerate}
	\item $\boldsymbol{I} \in \mathcal{S}$,
	\item $\boldsymbol{S}^* \in \mathcal{S}$ for all $\boldsymbol{S} \in \mathcal{S}$,
	\item $\boldsymbol{S} \cdot \boldsymbol{S}' \in \mathcal{S}$ for all $\boldsymbol{S},\boldsymbol{S}' \in \mathcal{S}$.
\end{enumerate}

An example of an involution monoid of $(1,1)$-bilabelled graphs is the \emph{path monoid} of all $(1,1)$-bilabelled paths with labels at opposing ends, cf.\@ \cref{fig:seriesprod}.
In order to derive systems of equations with finitely many equations, we consider finite generating sets of involution monoids:

\begin{definition} \label{def:basal} \label{def:generates}
	Let $\mathcal{S}$ be an involution monoid. A set $\mathcal{B} \subseteq \mathcal{S}$ \emph{generates} $\mathcal{S}$ if
	\begin{enumerate}
		\item $\boldsymbol{I} \in \mathcal{B}$,\label{it:basal1}
		\item $\boldsymbol{B}^* \in \mathcal{B}$ for all $\boldsymbol{B} \in \mathcal{B}$,\label{it:basal2}
		\item for all $\boldsymbol{S} \in \mathcal{S}$ there exist $\boldsymbol{B}^1, \dots, \boldsymbol{B}^r \in \mathcal{B}$ such that $\boldsymbol{S} = \boldsymbol{B}^1 \cdots \boldsymbol{B}^r$.\label{it:basal3}
	\end{enumerate}
\end{definition}
For example, the path monoid is generated by the $(1,1)$-bilabelled graph $\boldsymbol{A}$ depicted in \cref{fig:adjacency11,fig:seriesprod} and the identity graph $\boldsymbol{I}$.

For a graph class $\mathcal{F}$ and $N \in \mathbb{N}$, 
write $\mathcal{F}_{\leq N} \coloneqq\{  F \in \mathcal{F} \mid  \lvert V(F) \rvert \leq N\}$.
For a class $\mathcal{S} \subseteq \mathcal{G}(k,k)$ of $(k,k)$-bilabelled graphs where $k \geq 1$, write
\(
	\soe(\mathcal{S}) \coloneqq \{\soe \boldsymbol{S} \mid \boldsymbol{S} \in \mathcal{S}\} 
\) and
\(
	\tr(\mathcal{S}) \coloneqq \{\tr \boldsymbol{S} \mid \boldsymbol{S} \in \mathcal{S}\}.
\)
Both \(\soe(\mathcal{S}) \) and \(\tr(\mathcal{S}) \)  are classes of unlabelled graphs.
The following \cref{thm:meta-paths-cycles} is immediate from \cref{thm:soe,thm:wiegmann,lem:pearcy,prop:soe-quadratic-length}.

\begin{theorem} \label{thm:meta-paths-cycles}
	Let $k \geq 1$.
	Let $\mathcal{S} \subseteq \mathcal{G}(k,k)$ be an involution monoid generated by $\mathcal{B} \subseteq \mathcal{S}$. 
	Let $G$ and $H$ be $n$-vertex graphs.
	Suppose that every graph in $\mathcal{B}$ has at most $b \in \mathbb{N} \cup \{\infty\}$ vertices 
	and let $N_1 \coloneqq 2n^{k}b \in \mathbb{N} \cup \{\infty\}$ and $N_2 \coloneqq 2n^{2k}b \in \mathbb{N} \cup \{\infty\}$.
	Then the following are equivalent:
	\begin{enumerate}
		\item $G$ and $H$ are homomorphism indistinguishable over $\soe(\mathcal{S})$,
		\item $G$ and $H$ are homomorphism indistinguishable over $\soe(\mathcal{S})_{\leq N_1}$,
		\item there exists a pseudo-stochastic $X \in \mathbb{Q}^{V(H)^k \times V(G)^k}$ such that $X\boldsymbol{B}_G = \boldsymbol{B}_H X$ for all $\boldsymbol{B} \in \mathcal{B}$.
	\end{enumerate}
	Furthermore, the following are equivalent:
	\begin{enumerate}
		\item $G$ and $H$ are homomorphism indistinguishable over $\tr(\mathcal{S})$,
		\item $G$ and $H$ are homomorphism indistinguishable over $\tr(\mathcal{S})_{\leq N_2}$,
		\item there exists an orthogonal $U \in \mathbb{R}^{V(H)^k \times V(G)^k}$ such that $U\boldsymbol{B}_G = \boldsymbol{B}_H U$ for all $\boldsymbol{B} \in \mathcal{B}$.
	\end{enumerate}
\end{theorem}
\begin{proof}
	In the set-up of \cref{thm:soe,thm:wiegmann}, let $I \coloneqq V(G)^k$, $J \coloneqq V(H)^k$, and $M \coloneqq \mathcal{B}$.
	Furthermore, let $\matvec{A}$ (respectively, $\matvec{B}$) be the sequence of homomorphism tensors $\boldsymbol{B}_G$ (respectively, $\boldsymbol{B}_H$) for $\boldsymbol{B} \in \mathcal{B}$.
	Words $w \in \Gamma_M$ corresponds to bilabelled graphs from $\mathcal{S}$ and vice-versa.
	The matrices $w_{\matvec{A}}$ and $w_{\matvec{B}}$ are homomorphism tensors of such a bilabelled graph.
	With these observations, \cref{thm:meta-paths-cycles} is immediate from \cref{thm:soe,thm:wiegmann,lem:pearcy,prop:soe-quadratic-length}.
\end{proof}

We remark that we are only interested in the order of magnitude of the parameters~$N_1$ and~$N_2$.
In order to state \cref{thm:meta-paths-cycles} more clearly, we chose to be slightly wasteful compared to \cref{prop:soe-quadratic-length,lem:pearcy}.

The remainder of this section features an application of \cref{thm:meta-paths-cycles} to homomorphism indistinguishability over graphs of bounded pathwidth and cyclewidth.
The prototypical example of an involution monoid is the family of graphs of pathwidth at most $k$.

\begin{definition} \label{def:pwk}
	Let $\mathcal{PW}^k$ denote the family of all $(k+1,k+1)$-bilabelled graphs $\boldsymbol{F} = (F, \vec{u}, \vec{v})$ such that $F$ admits a path decomposition $(P, \beta)$ of width at most $k$ with $u,v \in V(P)$ satisfying
	\begin{enumerate}
		\item $\beta(u) = \{u_1, \dots, u_{k+1}\}$ and $\beta(v) = \{v_1, \dots, v_{k+1}\}$,
		\item if $u \neq v$ then $\deg_P(u) = \deg_P(v) = 1$ and if $u = v$ then $\deg_P(u) = \deg_P(v) = 0$,
		\item $|\beta(s)| = k+1$ for all $s \in V(P)$ and $|\beta(s) \cap \beta(t)| = k$ for all $st \in E(P)$.\label{def:pwk4}
	\end{enumerate}
\end{definition}

The first two axioms of \cref{def:pwk} prescribe where in a path decomposition the labelled vertices have to be placed.
The last axiom makes subsequent arguments easier and does not constitute a loss of generality:
	It is easy to see that $\mathcal{PW}^k$ contains $\boldsymbol{I}$, is closed under reversal and series composition. 
Hence, it is an involution monoid.

\begin{lemma} \label{lem:pw-soe-tr}
	Let $k \geq 1$.
	\begin{enumerate}
		\item The class $\soe(\mathcal{PW}^k)$ is the class of all graphs of pathwidth at most $k$ on at least $k+1$ vertices.
		\item The class $\tr(\mathcal{PW}^k)$ is the class of all graphs of cyclewidth at most $k$ on at least $k+1$ vertices.
	\end{enumerate}
\end{lemma}
\begin{proof}
	By \cref{def:pwk}, every $F \in \soe(\mathcal{PW}^k)$ has pathwidth at most $k$ and at least $k+1$ vertices.
	Conversely, if $F$ has pathwidth at most $k$ and at least $k+1$ vertices then, by \cref{lem:bodlaender8}, there exists a path decomposition $(P, \beta)$ of $F$ satisfying \cref{def:pwk4} of \cref{def:pwk}. The labels can be arbitrarily placed on $F$ in accordance with \cref{def:pwk}.
	
	For the second claim,
	let $\boldsymbol{F} = (F, \vec{u}, \vec{v}) \in \mathcal{PW}^k$ be a bilabelled graph with path decomposition $(P, \beta)$ and $u, v\in V(P)$ as in \cref{def:pwk}.
	Let $C$ denote the cycle obtained from $P$ by making a fresh vertex $w$ adjacent to $u$ and $v$. Extend $\beta(w) \coloneqq \beta(u) \cup \beta(v)$. 
	In $\tr(\boldsymbol{F})$, the vertices $u_i$ and $v_i$ for $i \in [k+1]$ are respectively identified. 
	Hence, $(C,\beta)$ gives rise to a cycle decomposition of $\tr(\boldsymbol{F})$ of width at most $k$.
	
	Conversely, by \cref{lem:bodlaender8}, if $F$ has cyclewidth at most $k$ and at least $k+1$ vertices then there is a cycle decomposition $(C, \beta)$ of $F$ satisfying \cref{def:pwk4} of \cref{def:pwk}.
	Pick any vertex $w \in V(C)$. Let $P$ denote the path obtained from $C$ by replacing $w$ with two fresh vertices $u$ and $v$ adjacent to the two neighbours of $w$ respectively.
	Write $X \subseteq \beta(w)$ for the set of all vertices $x \in \beta(w)$ which do not appear in all bags of $(C, \beta)$, i.e.\@ $x \not\in \beta(y)$ for some $y \in V(C)$.
	For $x \in X$, let $c_1, \dots, c_r$ and $d_1, \dots, d_s$ denote the vertices of $C$ whose bags contain~$x$. 
	Suppose that $d_sd_{s-1}\dots d_1 w c_1 \dots c_r$ is a trail in $C$ and that $u$ was made adjacent to $c_1$, and $v$ to $d_1$.
	Construct a graph $F'$ from $F$ by replacing every $x \in X$ by two vertices $x'$ and $x''$ and making $x'$ adjacent to all neighbours of $x$ in $\beta(c_1) \cup \dots \cup \beta(c_r)$ and $x''$ adjacent to all neighbours of $x$ in $\beta(d_1) \cup \dots \cup \beta(d_s)$.
	Define a path decomposition $(P, \gamma)$ of $F'$ by letting $\gamma(u) \coloneqq (\beta(w) \setminus X) \cup \{x' \mid x\in X\}$ and $\gamma(v) \coloneqq (\beta(w) \setminus X) \cup \{x'' \mid x\in X\}$.
	Every other bag $\gamma(z)$ for $z \in V(P) \setminus \{u, v\}$ is obtained from $\beta(z)$ by replacing $x \in X$ by $x'$ or $x''$ depending on whether $z$ is among the $c_1, \dots, c_r$ or the $d_1, \dots, d_s$.
	
	Let $\vec{u}, \vec{v} \in V(F')^{k+1}$ be tuples comprised of the vertices of $\gamma(u)$ and $\gamma(v)$ such that if $u_i = x'$ for some $x \in X$ and $i \in [k+1]$ then $v_i = x''$.
	Also, if $u_i \in \beta(w) \setminus X$ for $i \in [k+1]$, then we require that $u_i = v_i$.
	Let $\boldsymbol{F}' \coloneqq (F', \vec{u}, \vec{v})$.
	Taking the trace of $\boldsymbol{F}'$ has the effect that for every $x \in X$ the copies $x'$ and $x''$ are identified. 
	Hence, $\tr \boldsymbol{F}' \cong F$, as desired.
\end{proof}

To apply \cref{thm:meta-paths-cycles}, it remains to give a set of generators for $\mathcal{PW}^k$, cf.\@ \cref{fig:tw-basal}.

\begin{figure}
	\centering
	\begin{subfigure}[b]{.24\linewidth}
		\centering
		\includegraphics[page=9,scale=.8]{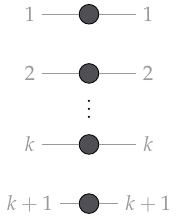}
		\caption{identity graph $\boldsymbol{I}$.}
\end{subfigure}
	\begin{subfigure}[b]{.24\linewidth}
		\centering
		\includegraphics[page=8,scale=.8]{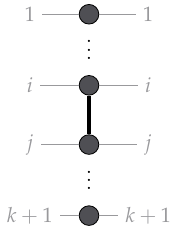}
		\caption{adjacency graph $\boldsymbol{A}^{ij}$.}
	\end{subfigure}
\begin{subfigure}[b]{.24\linewidth}
		\centering
		\includegraphics[page=10,scale=.8]{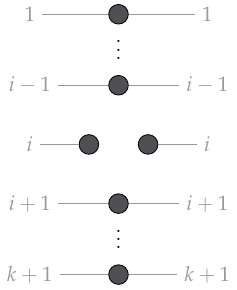}
		\caption{forgetting graph $\boldsymbol{J}^i$.}
	\end{subfigure}
	\begin{subfigure}[b]{.24\linewidth}
	\centering
	\includegraphics[page=15,scale=.8]{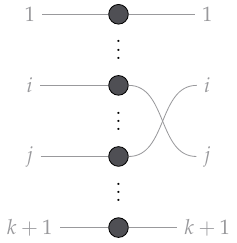}
	\caption{swap graph $\boldsymbol{S}^{ij}$.}
\end{subfigure}
	\caption{Bilabelled graphs in $\mathcal{B}^k$ as defined in \cref{lem:basal}.}
	\label{fig:tw-basal}
\end{figure}

\begin{lemma} \label{lem:basal} \label{ssec:basal1}
	The set $\calB^k$ consisting of the following $(k+1,k+1)$-bilabelled graphs generates $\mathcal{PW}^k$. For $1 \leq i \neq j \leq k+1$,
	\begin{itemize}
		\item the \emph{identity graph} $\boldsymbol{I} = (I, (1, \dots, k+1), (1, \dots, k+1))$ with $V(I) = [k+1]$, $E(I) = \emptyset$,
		\item the \emph{adjacency graphs} $\boldsymbol{A}^{ij} = (A^{ij}, (k+1), (k+1))$ with $V(A^{ij}) = [k+1]$ and $E(A) = \left\{ij \right\}$,
\item the \emph{forgetting graphs} $\boldsymbol{J}^i = (F^i, (1, \dots, k+1), (1, \dots, i-1, i', i+1, \dots, k+1))$ with $V(F^i) = [k+1] \cup \{i'\}$ and $E(F^i) = \emptyset$,
		\item the \emph{swap graphs} $\boldsymbol{S}^{ij} = (S^{ij}, (1, \dots, k+1), (1, \dots, i-1, j, i+1, \dots, j-1, i, j+1, \dots, k+1))$ with $V(S^{ij}) = [k+1]$ and $E(S^{ij}) = \emptyset$.
	\end{itemize}
\end{lemma}
\begin{proof}
	Clearly, $\mathcal{B}^k \subseteq \mathcal{PW}^k$.
	\Cref{it:basal1,it:basal2} of \cref{def:basal} are immediate.
	It remains to verify \cref{it:basal3}:
	To that end, let $F$ be an arbitrary graph with a path decomposition $(P, \beta)$ with vertices $u, v\in V(P)$ and $\vec{u}, \vec{v} \in V(F)^{k+1}$ as in \cref{def:pwk}.
	The proof is by induction on $|V(P)|$.
	
	If $|V(P)| = 1$ then $u = v$ and $\{{u}_1, \dots, {u}_{k+1}\} = \{{v}_1, \dots, {v}_{k+1}\}$. Hence, there exists a permutation $\sigma \colon [k+1] \to [k+1]$ such that ${u}_i = {v}_{\sigma(i)}$ for all $i \in [k+1]$. Write $\sigma = \tau_1 \cdots \tau_r$ as product of transpositions.
	Then $\boldsymbol{F} = (F, \vec{u}, \vec{v})$ is equal to
	\[
	\boldsymbol{S}^{\tau_1} \cdots \boldsymbol{S}^{\tau_r} \cdot 
	\prod_{\substack{1 \leq i \neq j \leq k+1\\ {v}_i{v}_j \in E(F)}} \boldsymbol{A}^{ij},
	\]
	a product of graphs in $\mathcal{B}^k$.
	
	If $|V(P)| \geq 2$, let $w \in V(P)$ denote the unique neighbour of $u$. Let $P' \coloneqq P - u$. 
	The subgraph $F'$ of $F$ induced by $\bigcup_{p \in V(P')} \beta(p)$ satisfies \cref{def:pwk} with the path decomposition $(P',\beta|_{V(P')})$, 
	the vertices $w, v$, the tuple $\vec{v}$ and some tuple $\vec{w} \in V(F')^{k+1}$ such that $\beta(w) = \{{w}_1,\dots, {w}_{k+1}\}$ and $w_i = u_i$ for all $i \in [k+1] \setminus \{\ell\}$ for some $\ell \in [k+1]$. 
	Let $\boldsymbol{F}' \coloneqq (F', \vec{w}, \vec{v})$.
	Then
	\[
		\boldsymbol{F} = \prod_{\substack{1 \leq i \neq j \leq k+1\\ {u}_i{u}_j \in E(F)}} \boldsymbol{A}^{ij} \cdot \boldsymbol{J}^\ell \cdot \boldsymbol{F}'.
	\]
	The claim follows inductively.
\end{proof}

In order to avoid the technicalities of working with small graphs, we record the following \cref{lem:padding}, which describes a padding trick.

\begin{lemma} \label{lem:padding}
	Let $\mathcal{F}' \subseteq \mathcal{F}$ be graph classes such that 
	$nK_1 \in \mathcal{F}'$ for some $n \geq 1$.
	Suppose that for all $F \in \mathcal{F} \setminus \mathcal{F}'$ it holds that $F + \ell K_1 \in \mathcal{F}'$ for some $\ell \geq 1$.
	Then for all graphs $G$ and $H$ it holds that $G \equiv_{\mathcal{F}} H$ if and only if $G \equiv_{\mathcal{F}'} H$.
\end{lemma}
\begin{proof}
	Since $\mathcal{F}' \subseteq \mathcal{F}$, it suffices to argue that the backward implication holds.
	Suppose that $G$ and $H$ are homomorphism indistinguishable over $\mathcal{F}'$. 
	Since $nK_1 \in \mathcal{F}'$ for some $n \geq 1$, it holds that $\lvert V(G) \rvert^n = \hom(nK_1, G) = \hom(nK_1, H) = \lvert V(H) \rvert^n$.
	Hence, $G$ and $H$ have the same number of vertices.
	Suppose wlog that this number is non-zero.
	Let $F \in \mathcal{F} \setminus \mathcal{F}'$. Then $F + \ell K_1 \in \mathcal{F}'$ for some $\ell \geq 1$. Hence, by \cite[(5.28)]{lovasz_large_2012},
	\[
		\hom(F, G) = \frac{\hom(F + \ell K_1, G)}{\hom(\ell K_1, G)} = \frac{\hom(F + \ell K_1, H)}{\hom(\ell K_1, H)} = \hom(F, H),
	\]
	which implies that $G$ and $H$ are homomorphism indistinguishable over $\mathcal{F}$.
\end{proof}

For the case of the classes of graphs of pathwidth at most~$k$ or cyclewidth at most~$k$, 
we apply \cref{lem:padding} with $\mathcal{F}$ being the respective graph class and
$\mathcal{F}_{\geq k} \coloneqq \{F \in \mathcal{F} \mid |V(F)| \geq k\}$ assuming the role of $\mathcal{F}'$.
With this choice of $\mathcal{F}'$,
the somewhat cumbersome assumptions of \cref{lem:padding} can be alleviated for graph classes which are minor-closed and closed under disjoint unions.
For example, by \cref{lem:pw-soe-tr}, it holds that two graphs $G$ and $H$ are homomorphism indistinguishable over the class of graphs of pathwidth at most~$k$ if and only if $G \equiv_{\soe(\mathcal{PW}^k)} H$.
In contrast, the class  of graphs of cyclewidth at most~$k$ is not closed under disjoint unions but satisfies the weaker assumptions of \cref{lem:padding}.
Indeed, if a graph $F$ has less at most $k+1$ vertices then it has cycle decomposition with a single bag and all graphs $F + nK_1$ for $n \in \mathbb{N}$ are also of cyclewidth at most~$k$.
Thus \cref{lem:padding,lem:pw-soe-tr} yield that two graphs $G$ and $H$ are homomorphism indistinguishable over the class of graphs of cyclewidth at most $k$ if and only if $G \equiv_{\tr(\mathcal{PW}^k)} H$.

This concludes the preparation for obtaining a system of matrix equations characterising homomorphism indistinguishability over graphs of bounded pathwidth and cyclewidth via \cref{thm:meta-paths-cycles}.
For later reference in \cref{sec:sa}, we denote the system of linear equations in \cref{it:pwk} of \cref{thm:hom-pw} by $\PW^{k+1}(G, H)$.

\begin{theorem} \label{thm:hom-pw}
	Let $k \geq 1$. For $n$-vertex simple graphs  $G$ and $H$, the following are equivalent:
	\begin{enumerate}
		\item $G$ and $H$ are homomorphism indistinguishable over graphs of pathwidth at most $k$,\label{it:pw1}
		\item $G$ and $H$ are homomorphism indistinguishable over graphs of pathwidth at most $k$ on at most $2n^{k+1}(k+2)$ vertices,\label{it:pw2}
		\item there exists a pseudo-stochastic $X \in \mathbb{Q}^{V(H)^{k+1} \times V(G)^{k+1}}$ such that $X \boldsymbol{B}_G = \boldsymbol{B}_H X$ for all $\boldsymbol{B} \in \mathcal{B}^k$.\label{it:pwk}
	\end{enumerate}
\end{theorem}
\begin{proof}
	In virtue of \cref{lem:basal},
	we apply \cref{thm:meta-paths-cycles} to the involution monoid $\mathcal{PW}^k$ with generating set $\mathcal{B}^k$.
	Write $N \coloneqq 2n^{k+1}(k+2)$. Consider the following additional assertions:
	\begin{enumerate}[start=4]
		\item $G$ and $H$ are homomorphism indistinguishable over graphs of pathwidth at most $k$ on at least~$k+1$ vertices,\label{it:pw4}
		\item $G$ and $H$ are homomorphism indistinguishable over graphs of pathwidth at most $k$ on at least~$k+1$ vertices and at most~$N$ vertices.\label{it:pw5}
	\end{enumerate}
	By \cref{thm:meta-paths-cycles,lem:pw-soe-tr}, \cref{it:pwk,it:pw4,it:pw5} are equivalent.
	By containment of the respective graph classes, \cref{it:pw1} implies \cref{it:pw2}, which implies \cref{it:pw5}.
	By \cref{lem:padding}, \cref{it:pw1,it:pw4} are equivalent. This closes a cycle of implications.
\end{proof}

Analogously, we obtain the following theorem:
\begin{theorem}
	Let $k \geq 1$. For $n$-vertex simple graphs  $G$ and $H$, the following are equivalent:
	\begin{enumerate}
		\item $G$ and $H$ are homomorphism indistinguishable over the class of graphs of cyclewidth at most $k$,
		\item $G$ and $H$ are homomorphism indistinguishable over the class of graphs of cyclewidth at most $k$ on at most $2n^{2(k+1)}(k+2)$ vertices,
		\item there exists an orthogonal $U \in \mathbb{R}^{V(H)^{k+1} \times V(G)^{k+1}}$ such that $U \boldsymbol{B}_G = \boldsymbol{B}_H U$ for all $\boldsymbol{B} \in \mathcal{B}^k$.
	\end{enumerate}
\end{theorem}

\subsection{Treewidth: Agglutinative Generation}
\label{sec:trees}

The class of graphs of treewidth at most $k$ is generated by the same generators as $\mathcal{PW}^k$, cf.\@ \cref{lem:basal}.
However, instead of only considering series composition, we require the gluing product as well.
To that end, recall the $k$-labelled version $\boldsymbol{1}$ of the $(k,k)$-bilabelled graph~$\boldsymbol{I}$ from \cref{ex:adjacency,fig:identity-graph-labelled}.

\begin{definition} \label{def:gluten-generation}
	Let $k \geq 1$.
	Let $\mathcal{S} \subseteq \mathcal{G}(k,k)$ be an involution monoid.
	The set $\mathcal{X} \subseteq \mathcal{G}(k)$ of \emph{graphs agglutinatively generated by $\mathcal{S}$} is inductively defined as follows:
	\begin{enumerate}
		\item $\boldsymbol{1} \in \mathcal{X}$,
		\item $\boldsymbol{S} \cdot \boldsymbol{X} \in \mathcal{X}$ for $\boldsymbol{S} \in \mathcal{S}$ and $\boldsymbol{X} \in \mathcal{X}$,
		\item $\boldsymbol{X} \odot \boldsymbol{X}' \in \mathcal{X}$ for $\boldsymbol{X},\boldsymbol{X}' \in \mathcal{X}$.
	\end{enumerate}
\end{definition}
For a class $\mathcal{R} \subseteq \mathcal{G}(k)$ where $k \geq 1$, write $\soe(\mathcal{R}) \coloneqq \{\soe \boldsymbol{R} \mid \boldsymbol{R} \in \mathcal{R}\}$.
For agglutinatively generated graph classes, the following general theorem follows from \cref{thm:soe-pos,prop:soe-pos-exp-size}:
\begin{theorem} \label{thm:meta-trees}
	Let $k \geq 1$.
	Let $\mathcal{S} \subseteq \mathcal{G}(k,k)$ be an involution monoid generated by $\mathcal{B} \subseteq \mathcal{S}$. 
	Let $\mathcal{X} \subseteq \mathcal{G}(k)$ be the class agglutinatively generated by $\mathcal{S}$.
	Let $G$ and $H$ be $n$-vertex graphs.
	Suppose every graph in $\mathcal{B}$ has at most $b \in \mathbb{N} \cup \{\infty\}$ vertices and let $N \coloneqq (2n^k)^{n^k +1 }b \in \mathbb{N} \cup \{\infty\}$.
	Then the following are equivalent:
	\begin{enumerate}
		\item $G$ and $H$ are homomorphism indistinguishable over $\soe(\mathcal{X})$,
		\item $G$ and $H$ are homomorphism indistinguishable over $\soe(\mathcal{X})_{\leq N}$,
		\item there exists a doubly stochastic $X \in \mathbb{Q}^{V(H)^k \times V(G)^k}$ such that $X\boldsymbol{B}_G = \boldsymbol{B}_H X$ for all $\boldsymbol{B} \in \mathcal{B}$.
	\end{enumerate}
\end{theorem}

Equipped with this general theorem, we now turn to establishing that the class of graphs of bounded treewidth is subject to it.

\begin{lemma} \label{lem:tw-soe}
	Let $k \geq 1$.
	Let $\mathcal{TW}^k \subseteq \mathcal{G}(k+1)$ denote the class which is agglutinatively generated by~$\mathcal{PW}^k$.
	Then $\soe(\mathcal{TW}^k)$ is the class of all graphs of treewidth at most $k$ on at least $k+1$ vertices.
\end{lemma}
\begin{proof}
	We first show that all graphs in $\boldsymbol{F} = (F, \boldsymbol{u}) \in \mathcal{TW}^k$ admit a tree decomposition $(T, \beta)$ of width at most $k$ such that there is $r \in V(T)$ with $\beta(r) = \{u_1, \dots, u_{k+1}\} $. We call $r$ the \emph{root} of the decomposition.
	The hypothesis clearly holds for $\boldsymbol{1}$.
	
	By structural induction, suppose that $\boldsymbol{F} = \boldsymbol{S} \cdot \boldsymbol{X}$ for $\boldsymbol{X} = (X, \boldsymbol{x}) \in \mathcal{TW}^k$ of lesser complexity and $\boldsymbol{S} = (S, \boldsymbol{u}, \boldsymbol{v}) \in \mathcal{PW}^k$.
	Let $(P, \beta)$ denote the path decomposition of $\boldsymbol{S}$ with vertices $u,v\in V(P)$ as stipulated in \cref{def:pwk}.
	Let $(T, \gamma)$ denote the tree decomposition of $\boldsymbol{X}$ with root $r$ whose existence is guaranteed by the inductive hypothesis.
	Define a tree $Q$ by taking the disjoint union of $P$ and $T$ and identifying $v$ and $r$. Define $\alpha \colon V(Q) \to 2^{V(F)}$ via
	\[
		\alpha(q) = \begin{cases}
			\beta(q), & \text{if } q \in V(P), \\
			\gamma(q), & \text{if } q \in V(T).
		\end{cases}
	\]
	Since $\beta(v) = \gamma(r)$ implicitly, the map $\alpha$ is a tree decomposition of width at most $k$ of $F$.
	By construction, all labelled vertices in $\boldsymbol{F}$ lie in the same bag. 
	
	If $\boldsymbol{F} = \boldsymbol{X} \odot \boldsymbol{X}'$, a tree decomposition for $\boldsymbol{F}$ can be constructed from the tree decompositions of $\boldsymbol{X}$ and $\boldsymbol{X}'$ by taking the disjoint union of the decomposition trees and identifying the roots.
	
	Conversely, we consider graphs of treewidth at most $k$ with tree decompositions as in \cref{lem:bodlaender8}.
	By induction on the size of decomposition tree~$T$, 
	we show that for every graph $F$ with tree decomposition $(T, \beta)$ of width at most $k$ as in \cref{lem:bodlaender8}
	and every $r \in V(T)$ and $\boldsymbol{u} \in V(F)^{k+1}$ such that $\beta(r) = \{u_1, \dots, u_{k+1}\}$, 
	it holds that $\boldsymbol{F} \coloneqq (F, \boldsymbol{u}) \in \mathcal{TW}^k$.
	
	If $\lvert V(T) \rvert = 1$ then $T$ is a path. 
	Clearly, $\boldsymbol{F}' \coloneqq (F, \boldsymbol{u}, \boldsymbol{u}) \in \mathcal{PW}^k$ and $\boldsymbol{F} = \boldsymbol{F}' \boldsymbol{1} \in \mathcal{TW}^k$, as desired.
	
	If $\lvert V(T) \rvert > 1$, distinguish two cases:
	First suppose that $r$ has only one neighbour $r'$.
	Define $T'$ as the tree obtained from $T$ by deleting $r$ and write $F'$ for the subgraph of $F$ induced by $\bigcup_{t' \in V(T')} \beta(t')$.
	Let $\beta'$ denote the restriction of $\beta$ to $V(T')$. 
	By \cref{lem:bodlaender8}, there is a unique index $i \in [k+1]$ such that $u_i \in \beta(r) \setminus \beta(r')$.
	Write $x$ for the unique vertex in $\beta(r') \setminus \beta(r)$.
	Then the inductive hypothesis applies to $F'$, $(T', \beta')$, $r'$ and $\boldsymbol{v} \coloneqq u_1 \dots u_{i-1} x u_{i+1} \dots u_{k+1} \in V(F')^{k+1}$.
	Then 
	\[
		\boldsymbol{F} = \prod_{\ell, j \in [k+1] \text{ s.t. } u_\ell u_j \in E(F)} \boldsymbol{A}^{\ell j} \cdot \boldsymbol{J}^i \cdot (F', \boldsymbol{v}) \in \mathcal{TW}^k. 
	\]
	
	Finally, suppose that $r$ has multiple neighbours $r'_1, \dots, r'_m$.
	For $i \in [m]$, write $T_i$ for the connected component of $r'_i$ in the forest obtained from $T$ by deleting $r'_1, \dots, r'_{i-1}, r'_{i+1}, \dots, r'_m$.
	Write $F'_i$ for the subgraph of $F$ induced by $\bigcup_{t' \in V(T_i)} \beta(t')$ and $\beta'_i$ for the restriction of $\beta$ to $V(T'_i)$.
	Observe that $r$ is of degree one in all graphs $T_1, \dots, T_m$. 
	By the previous case, $(F'_i, \boldsymbol{u}) \in \mathcal{TW}^k$.
	Clearly, $\boldsymbol{F} = \bigodot_{i=1}^m (F'_i, \boldsymbol{u}) \in \mathcal{TW}^k$.
\end{proof}

This concludes the preparation for the proof of the main theorem of this section.

\begin{theorem} \label{thm:trees}
	Let $k \geq 1$.
	Let $G$ and $H$ be $n$-vertex simple graphs.
	Then the following are equivalent:
	\begin{enumerate}
		\item $G$ and $H$ are homomorphism indistinguishable over the class of graphs of treewidth at most $k$,\label{it:tw1}
		\item $G$ and $H$ are homomorphism indistinguishable over the class of graphs of treewidth at most $k$ on at most $(2n^{k+1})^{n^{k+1} +1}(k+2)$ vertices,\label{it:tw2}
		\item there exists a doubly stochastic $X \in \mathbb{Q}^{V(H)^{k+1} \times V(G)^{k+1}}$ such that $X\boldsymbol{B}_G = \boldsymbol{B}_H X$ for all $\boldsymbol{B} \in \mathcal{B}^k$.\label{it:tw3}
	\end{enumerate}
\end{theorem}
\begin{proof}
	In virtue of \cref{lem:basal},
	we apply \cref{thm:meta-trees} to $\mathcal{TW}^k$, which is agglutinatively generated by the involution monoid $\mathcal{PW}^k$ with generating set $\mathcal{B}^k$.
	Write $N \coloneqq (2n^{k+1})^{n^{k+1} + 1}(k+2)$. Consider the following additional assertions:
	\begin{enumerate}[start=4]
		\item $G$ and $H$ are homomorphism indistinguishable over the class of graphs of treewidth at most $k$ on at least~$k+1$ vertices,\label{it:tw4}
		\item $G$ and $H$ are homomorphism indistinguishable over the class of graphs of treewidth at most $k$ on at least~$k+1$ vertices and at most~$N$ vertices.\label{it:tw5}
	\end{enumerate}
	By \cref{thm:meta-trees,lem:tw-soe}, \cref{it:tw3,it:tw4,it:tw5} are equivalent.
	By containment of the respective graph classes, \cref{it:tw1} implies \cref{it:tw2}, which implies \cref{it:tw5}.
	By \cref{lem:padding}, \cref{it:tw1,it:tw4} are equivalent. This closes a cycle of implications.
\end{proof}

\subsection{Treedepth: Generation by an Involution Monoid and Closure under Gluing}
\label{sec:treedepth}

In order to infer a system of equations characterising homomorphism indistinguishability over graphs of bounded treedepth, we consider another type of interaction between labelled and bilabelled graphs. Let $\mathcal{S} \subseteq \mathcal{G}(k,k)$ be an involution monoid. Write $\mathcal{S}\boldsymbol{1} \coloneqq \{\boldsymbol{S} \cdot \boldsymbol{1} \mid \boldsymbol{S} \in \mathcal{S}\} \subseteq \mathcal{G}(k)$. This section is concerned with the case when $\mathcal{S}\boldsymbol{1}$ is gluing-closed.

\begin{definition}
	Let $k \geq 1$.
	A set $\mathcal{R} \subseteq \mathcal{G}(k)$ is \emph{gluing-closed} if $\boldsymbol{R} \odot \boldsymbol{R}' \in \mathcal{R}$ for all $\boldsymbol{R}, \boldsymbol{R}' \in \mathcal{R}$.
\end{definition}

For example, the class of graphs which admit a coalgebra w.r.t.\@ a fixed comonad is gluing-closed \cite{rattan_weisfeiler_2023,dawar_lovasz_2021}.
Note that any agglutinatively generated graph class is gluing-closed.
If $\mathcal{S}\boldsymbol{1}$ is gluing-closed then pseudo-stochastic solutions exist if and only if doubly stochastic solutions exist:

\begin{theorem} \label{thm:meta-treedepth}
	Let $k\geq 1$.
	Let $\mathcal{S} \subseteq \mathcal{G}(k,k)$ be an involution monoid generated by $\mathcal{B} \subseteq \mathcal{S}$. 
	Suppose that $\mathcal{S} \boldsymbol{1}$ is gluing-closed.
	Suppose that every graph in $\mathcal{B}$ has at most $b \in \mathbb{N} \cup \{\infty\}$ vertices and let $N \coloneqq 2n^{k}b \in \mathbb{N} \cup \{\infty\}$.
	Then for graphs $G$ and $H$, the following are equivalent:
	\begin{enumerate}
		\item $G$ and $H$ are homomorphism indistinguishable over $\soe(\mathcal{S})$,\label{it:td1}
		\item $G$ and $H$ are homomorphism indistinguishable over $\soe(\mathcal{S})_{\leq N}$,\label{it:td2}
		\item there exists a pseudo-stochastic matrix $X \in \mathbb{Q}^{V(H)^k \times V(G)^k}$ such that $X \boldsymbol{B}_G = \boldsymbol{B}_H X$ for all $\boldsymbol{B} \in \mathcal{B}$,\label{it:td3}
		\item there exists a doubly stochastic matrix $X \in \mathbb{Q}^{V(H)^k \times V(G)^k}$ such that $X \boldsymbol{B}_G = \boldsymbol{B}_H X$ for all $\boldsymbol{B} \in \mathcal{B}$.\label{it:td4}
	\end{enumerate}
\end{theorem}
\begin{proof}
	By \cref{thm:meta-paths-cycles}, \cref{it:td1,it:td2,it:td3} are equivalent.
	Let $\mathcal{X} \subseteq \mathcal{G}(k)$ denote the class agglutinatively generated by~$\mathcal{S}$, cf.\@ \cref{def:gluten-generation}.
	Since $\mathcal{S} \boldsymbol{1}$ is gluing-closed, it follows inductively that $\mathcal{X} \subseteq \mathcal{S} \boldsymbol{1}$.
	Conversely, $\mathcal{S} \boldsymbol{1} \subseteq \mathcal{X}$ since taking product of bilabelled graphs in $\mathcal{S}$ is one particular operation listed in \cref{def:gluten-generation}.
	Hence, $\mathcal{S}\boldsymbol{1} = \mathcal{X}$.
	\Cref{thm:meta-trees} yields that \cref{it:td1,it:td4} are equivalent.
\end{proof}

In \cite{grohe_counting_2020}, it was shown that homomorphism indistinguishable over graphs of treedepth at most $k$ corresponds to equivalence over the quantifier-rank-$k$
fragment of first-order logic with counting quantifiers. 
We extend this characterisation by proposing a linear system of equations very similar to the one for bounded pathwidth. 

First, we recall the following notions including the definition of treedepth as introduced in \cite{nesetril_tree-depth_2006}. See also \cite{nesetril_sparsity_2012} for further background.
A \emph{forest order} on a set $X$ is a partial order~$\leq$ on $X$ such that, for every $x \in X$, the set $\{y \in X \mid y \leq x\}$ is finite and totally ordered \cite[Section~II.C]{dawar_lovasz_2021}.
Let $F$ be a graph.
An \emph{elimination forest} of $F$ is a forest order $\leq$ on $V(F)$ such that, for every edge $uv \in E(F)$, the vertices $u$ and $v$ are \emph{comparable} with respect to to $\leq$, 
i.e.\ it holds that $u \leq v$ or $v \leq u$.
The \emph{height} of $\leq$ is the size of the largest subset of $V(F)$ which is totally ordered under $\leq$.
The \emph{treedepth}  of a graph $F$ is the minimum height of an elimination forest of~$F$.

A \emph{leaf} of an elimination forest $\leq$ on $F$ is a vertex $v \in V(F)$ such that for all $w \in V(F)$ it holds that $v \leq w $ implies that $  v=w$.
Dually, a \emph{root} is a vertex $v \in V(F)$ such that for all $w \in V(F)$ it holds that $w \leq v $ implies that $  v = w$.

The following \cref{def:tdk} introduces a suitable class of labelled graphs of bounded treedepth.

\begin{definition} \label{def:tdk}
	Let $k \geq 1$.
	The set $\mathcal{TD}^k$ is the set of all $(k,k)$-bilabelled graph $\boldsymbol{F} = (F, \vec{u}, \vec{v})$ such that there exists a rooted forest $\leq$ on $F$ satisfying
	\begin{enumerate}
		\item every edge $uv \in E(F)$ is such that $u \leq v$ or $v \leq u$,\label{b1}
		\item for every leaf $x \in V(F)$ of $\leq$, the set $\{ z \in V(F) \mid z \leq x\}$ has size $k$, \label{b1a}
		\item $\vec{u}$ and $\vec{v}$ are paths from a root to a leaf,
		i.e.\@ $u_1 < u_2 < \dots < u_k$ and $v_1 < v_2 < \dots < v_k$.\label{b2}
		\item the leaves $u_{k}$ and $v_{k}$ have the least number of common ancestors among all pairs of leaves of $F$, i.e.\@ writing\label{b3}
		\begin{equation}\label{eq:gca}
			\gca(x,y) \coloneqq \abs{\{ z \in V(F) \mid z \leq x \land z \leq y\}}
		\end{equation}
		for $x, y \in V(F)$, it holds that $\gca(x, y) \geq \gca(u_{k}, v_{k})$ for all pairs of leaves $x,y \in V(F)$.
	\end{enumerate}
\end{definition}

\cref{b1,b1a} ensure that the underlying unlabelled graph $F$ is of treedepth at most~$k$.
\cref{b2} guarantees that this property is preserved under series composition.
The remaining \cref{b3} helps to establish that $\mathcal{TD}^k$ is finitely generated.

\begin{lemma} \label{lem:tdk-involution-monoid}
	Let $k \geq 1$.
	$\mathcal{TD}^k$ is an involution monoid.
\end{lemma}
\begin{proof}
	Clearly, $\mathcal{TD}^k$ is closed under taking reverses and contains $\boldsymbol{I}$.
	Given $\boldsymbol{F} = (F, \vec{u}, \vec{v})$ and $\boldsymbol{F}' = (F', \vec{u}', \vec{v}')$ with rooted forest $\leq$ and $\leq'$, define a rooted forest $\leq''$ on $\boldsymbol{F} \cdot \boldsymbol{F}'$ by
	letting $x \leq'' y$ if $x, y\in V(F)$ and $x \leq y$ or if $x, y \in V(F')$ and $x \leq' y$. Since $\vec{v}$ and $\vec{u}'$ form paths from roots to leaves, this is a well-defined rooted forest.
	Also, in $\leq''$, $\vec{u}$ and $\vec{v}'$ are paths from roots to leaves. 
	
	It remains to check that \cref{b3} is satisfied.
	First observe that $\gca(u_{k}, v_{k}), \gca(u'_{k}, v'_{k}) \geq \gca(u_{k}, v'_{k})$. Indeed, any common ancestor of $u_{k}$ and $v'_{k}$ must be an ancestor of $v_{k}$, which is identified with $u'_{k}$. Let $z$ denote the maximal vertex such that $z \leq'' u_{k}, v'_{k}$. 
	Note that $z \leq'' v_k, u'_k$.
	Any leaf $x$ of $\boldsymbol{F}$ satisfies $z \leq x$. Indeed, if $z$ and $x$ were incomparable w.r.t.\@ $\leq$ then $\gca(x, u_{k}) < \gca(u_{k}, v'_{k})$ which contradicts the previous observation since $\gca(x, u_{k}) \geq \gca(u_{k}, v_{k})$. 
	The same argument applies to leaves of $\boldsymbol{F}'$. 
	It follows that $\gca(x, y) \geq \gca(u_{k}, v'_{k})$ for every pair of leaves $x,y$ in $\boldsymbol{F} \cdot \boldsymbol{F}'$.
\end{proof}

The following \cref{lem:tdk-gluing-closed} establishes another assumption of \cref{thm:meta-treedepth}:

\begin{lemma} \label{lem:tdk-gluing-closed}
	Let $k \geq 1$.
	$\mathcal{TD}^k\boldsymbol{1}$ is gluing-closed.
\end{lemma}
\begin{proof}
	Given $\boldsymbol{F} = (F, \vec{u})$ and $\boldsymbol{F}' = (F', \vec{u}')$ with rooted forest $\leq$ and $\leq'$, define a rooted forest $\leq''$ on $\boldsymbol{F} \odot \boldsymbol{F}'$ by
	letting $x \leq'' y$ if $x, y\in V(F)$ and $x \leq y$ or $x, y \in V(F')$ and $x \leq' y$.
	Since $\vec{u}$ and $\vec{u}'$ form paths from roots to leaves, this is a well-defined rooted forest.
	The other conditions in \cref{def:tdk} are easily verified.
\end{proof}

It remains to define generators for $\mathcal{TD}^k$. The graphs featured in \cref{lem:tdk-generates} are depicted by \cref{fig:td-basal}.
\begin{figure}
	\centering
	\begin{subfigure}[b]{.3\linewidth}
		\centering
		\includegraphics[page=1,scale=.9]{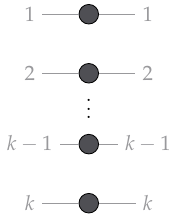}
		\caption{identity graph $\boldsymbol{I}$}
		\label{fig:identitygraph}
	\end{subfigure}
	\begin{subfigure}[b]{.3\linewidth}
		\centering
		\includegraphics[page=2,scale=.9]{figure-basal}
		\caption{adjacency graph $\boldsymbol{A}^{ij}$}
	\end{subfigure}
	\begin{subfigure}[b]{.3\linewidth}
		\centering
		\includegraphics[page=4,scale=.9]{figure-basal}
		\caption{join graph $\boldsymbol{J}^\ell$}
	\end{subfigure}
\caption{The graphs in $\mathcal{TDB}^k$ as defined in \cref{lem:tdk-generates}.}
	\label{fig:td-basal}
\end{figure}

\begin{lemma} \label{lem:tdk-generates}
	Let $k \geq 1$. Define the set $\mathcal{TDB}^k$ as the set of the following $(k,k)$-bilabelled graphs:
	\begin{itemize}
		\item the \emph{identity graph} $\boldsymbol{I} = (I, (k), (k))$ with $V(I) = [k]$ and $E(I) = \emptyset$,
		\item the \emph{adjacency graphs} $\boldsymbol{A}^{ij} = (A^{ij}, (k), (k))$ with $V(A^{ij}) = [k]$ and $E(A^{ij}) = \{ij\}$ for $1 \leq i < j \leq k$,
		\item the \emph{join graphs} $\boldsymbol{J}^\ell = (J^\ell, (k), (1, \dots, \ell, (\ell+1)', \dots, k'))$ with  $V(J^\ell) = \{1, \dots, k, (\ell+1)', \dots, k'\}$ and $E(J^\ell) = \emptyset$ for $0 \leq \ell < k$,
\end{itemize}
	Then $\mathcal{TDB}^k$ generates $\mathcal{TD}^k$.
\end{lemma}
\begin{proof}
	Clearly, $\mathcal{TDB}^k$ is closed under taking reverses and contains $\boldsymbol{I}$.
	For the second assertions, let $\boldsymbol{F} = (F, \vec{u}, \vec{v}) \in \mathcal{TD}^k$. Let $\leq$ denote a rooted forest for $F$.
	The proof is by induction on the number of leaves in $\leq$.
	
	If there is only one leaf then $\boldsymbol{F}$ is the product of the $\boldsymbol{A}^{ij}$ such that $1 \leq i \neq j \leq k$ and $ u_iu_j \in E(F)$.
	
	Now suppose that there are at least two leaves in $\leq$.
	Write $X$ for the set of all leaves $x$ of $\leq$ other than $u_{k}$ such that $\gca(u_{k}, x)$ is maximal, i.e.\@ such that $\gca(u_{k}, x) \geq \gca(u_{k}, y)$ for all leaves $y \neq u_{k}$. Let $x \in X$ be such that $\gca(x, v_{k})$ is minimal, i.e.\@ $\gca(x, v_{k}) \leq \gca(y, v_{k})$ for all $y \in X$.
	Write $D$ for the set of vertices $z \in V(F)$ such that $z \leq u_{k}$ and $z$ and $x$ are incomparable. Note that $D$ forms a chain in $\leq$.

	Let $F'$ be the graph obtained from $F$ by deleting all vertices in $D$.
	The rooted forest $\leq$ restricts to a rooted forest $\leq'$ of $F'$.
	Let $F''$ be the subgraph of $F$ induced by the vertices $z \leq u_k$. Let $\vec{w} \in V(F')^k$ be the tuple satisfying $w_1 < w_2 < \dots < w_k = x$.
	
	For \cref{b3}, distinguish cases: If $\gca(u_{k}, x) > \gca(u_{k}, v_{k})$ then there is a vertex $y$ such that $y$ is a common ancestor of $u_{k}$ and $x$ but $y$ is not an ancestor of $v_{k}$. This implies that every common ancestor of $v_{k}$ and $x$ is comparable with $y$ and hence $\gca(x, v_{k}) = \gca(u_{k}, v_{k})$.
	Hence, $\gca(a, b) \geq \gca(u_{k}, v_{k}) = \gca(x, v_{k})$ for any pair of leaves $a,b$ in $F'$.
	However, if $\gca(u_{k}, x) = \gca(u_{k}, v_{k})$ then all leaves $a$ of $F'$ are in fact in $X$ and the claim follows readily.
	
	The bilabelled graphs $\boldsymbol{F}' \coloneqq (F', \vec{w}, \vec{v})$ and $\boldsymbol{F}'' \coloneqq (F'', \vec{u}, \vec{u})$ are in $\mathcal{TD}^k$ and have less leaves than $\boldsymbol{F}$. The claim follows inductively, observing that $\boldsymbol{F} = \boldsymbol{F}'' \cdot \boldsymbol{J}^\ell \cdot \boldsymbol{F}'$ for $0\leq \ell < k$ minimal such that $w_{\ell+1} \not\in \{u_1, \dots, u_k\}$.
\end{proof}

The above observations yield the following \cref{thm:treedepth-first}, which implies one of the equivalences of \cref{thm:treedepth}:

\begin{theorem} \label{thm:treedepth-first}
	Let $k \geq 1$.
	Let $G$ and $H$ be $n$-vertex simple graphs.
	Then the following are equivalent:
	\begin{enumerate}
		\item $G$ and $H$ are homomorphism indistinguishable over the class of graphs of treedepth at most $k$,\label{it:tdf1}
		\item $G$ and $H$ are homomorphism indistinguishable over the class of graphs of treedepth at most $k$ on at most $4kn^{k}$ vertices,\label{it:tdf2}
		\item there exists a doubly stochastic $X \in \mathbb{Q}^{V(H)^k \times V(G)^k}$ such that $X\boldsymbol{B}_G = \boldsymbol{B}_H X$ for all $\boldsymbol{B} \in \mathcal{TDB}^k$,\label{it:tdf3}
		\item there exists a pseudo-stochastic $X \in \mathbb{Q}^{V(H)^k \times V(G)^k}$ such that $X\boldsymbol{B}_G = \boldsymbol{B}_H X$ for all $\boldsymbol{B} \in \mathcal{TDB}^k$.\label{it:tdf4}
	\end{enumerate}
\end{theorem}
\begin{proof}
	In virtue of \cref{lem:tdk-generates,lem:tdk-gluing-closed,lem:tdk-involution-monoid},
	we apply \cref{thm:meta-treedepth} to the gluing-closed $\mathcal{TD}^k\boldsymbol{1}$ and the generating set $\mathcal{TDB}^k$.
	Write $N \coloneqq 4kn^{k}$. Consider the following additional assertions:
	\begin{enumerate}[start=5]
		\item $G$ and $H$ are homomorphism indistinguishable over the class of graphs which admit a rooted forest satisfying \cref{b1,b1a} of \cref{def:tdk},\label{it:tdf5}
		\item $G$ and $H$ are homomorphism indistinguishable over the class of graphs on at most~$N$ vertices which admit a rooted forest satisfying \cref{b1,b1a} of \cref{def:tdk}.\label{it:tdf6}
	\end{enumerate}
	Observe that for every graph $F$ admitting a rooted forest satisfying \cref{b1,b1a} of \cref{def:tdk} one can pick $\boldsymbol{u}, \boldsymbol{v} \in V(F)^k$ such that \cref{b2,b3} of \cref{def:tdk} are satisfied as well.
	Hence, by \cref{thm:meta-treedepth}, \cref{it:tdf3,it:tdf4,it:tdf5,it:tdf6} are equivalent.
	By containment of the respective graph classes, \cref{it:tdf1} implies \cref{it:tdf2}, which implies \cref{it:tdf6}.
	
	By adding isolated vertices, any graph of treedepth at most $k$ can be turned into a graph with rooted forest satisfying \cref{b1,b1a} of \cref{def:tdk}.
	Thus, by \cref{lem:padding}, \cref{it:tdf1,it:tdf5} are equivalent. This closes a cycle of implications.
\end{proof}

\subsection{Bounded Degree Trees: Inner-Product Compatibility}
\label{sec:bounded-degree-trees}

In this final subsection, we turn to a class of labelled graphs which is endowed with less algebraic structure than the graph classes considered above.
The class of suitably labelled trees of bounded degree trees is inner-product compatible.
This is the most general property which ensures amenability to our approach.

\begin{definition}
	Let $k \geq 1$.
	A set $\mathcal{R} \subseteq \mathcal{G}(k)$ is \emph{inner-product compatible} if for all $\boldsymbol{R}, \boldsymbol{R}' \in \mathcal{R}$
	there exists $\boldsymbol{R}'' \in \mathcal{R}$ such that $\left< \boldsymbol{R}, \boldsymbol{R}' \right> = \soe(\boldsymbol{R}'')$.
\end{definition}

Since $\left< \boldsymbol{R}, \boldsymbol{R}' \right> = \soe(\boldsymbol{R} \odot \boldsymbol{R}')$ for all $\boldsymbol{R}, \boldsymbol{R}' \in \mathcal{G}(k)$, all gluing-closed families are also inner-product compatible. Similarly, if $\mathcal{R} = \mathcal{S}\boldsymbol{1}$ for some involution monoid $\mathcal{S}$ then $\mathcal{R}$ is inner-product compatible.
This holds since $\left< \boldsymbol{S}_1 \boldsymbol{1}, \boldsymbol{S}_2 \boldsymbol{1} \right> = \left<  \boldsymbol{1}, \boldsymbol{S}_1^*\boldsymbol{S}_2 \boldsymbol{1} \right> = \soe(\boldsymbol{S}_1^*\boldsymbol{S}_2 \boldsymbol{1})$ for $\boldsymbol{S}_1, \boldsymbol{S}_2 \in \mathcal{S}$.
\begin{theorem} \label{thm:meta}
	Let $k \geq 1$.
	Let $\mathcal{R} \subseteq \mathcal{G}(k)$ be inner-product compatible containing $\boldsymbol{1}$. 
	Then for graphs $G$ and $H$, the following are equivalent:
	\begin{enumerate}
		\item $G$ and $H$ are homomorphism indistinguishable over $\soe(\mathcal{R})$,
		\item there exists a pseudo-stochastic $X \in \mathbb{Q}^{V(H)^k \times V(G)^k}$ such that $X \boldsymbol{R}_G =\boldsymbol{R}_H$ for all $\boldsymbol{R} \in \mathcal{R}$.
	\end{enumerate}
\end{theorem}
\begin{proof}
	The backward direction is immediate. For the forward direction, consider the spaces $\mathbb{Q}\mathcal{R}_G$ and $\mathbb{Q}\mathcal{R}_H$ spanned by the $\boldsymbol{R}_G$ and $\boldsymbol{R}_H$ respectively for $\boldsymbol{R} \in \mathcal{R}$. By inner-product compatibility of~$\mathcal{R}$, for any $\boldsymbol{R}, \boldsymbol{R}' \in \mathcal{R}$ there exists $\boldsymbol{R}'' \in \mathcal{R}$ such that $\left< \boldsymbol{R}_G, \boldsymbol{R}'_G \right> = \soe(\boldsymbol{R}''_G) = \soe(\boldsymbol{R}''_H) = \left< \boldsymbol{R}_H, \boldsymbol{R}'_H \right>$. 
	By \cref{lemma:gs}, there exists an orthogonal map $U \colon \mathbb{Q}\mathcal{R}_G \to \mathbb{Q}\mathcal{R}_H$ such that $U\boldsymbol{R}_G = \boldsymbol{R}_H$ for all $\boldsymbol{R} \in \mathcal{R}$. In particular, $U\boldsymbol{1}_G = \boldsymbol{1}_H$ and $U^T \boldsymbol{1}_H = \boldsymbol{1}_G$. Hence, $U$ is pseudo-stochastic. 
	
	Define $X \colon \mathbb{Q}^{V(G)^k} \to \mathbb{Q}^{V(H)^k}$ as the map coinciding with $U$ on $\mathbb{Q}\mathcal{R}_G$ and annihilating $(\mathbb{Q}\mathcal{R}_G)^\bot$.
	Then $X\boldsymbol{R}_G = U \boldsymbol{R}_G = \boldsymbol{R}_H$ for all $\boldsymbol{R} \in \mathcal{R}$. Furthermore, for all $v \in \mathbb{Q}\mathcal{R}_G$,
	$\left< v, X^T \boldsymbol{1}_H \right> = \left< Uv,  \boldsymbol{1}_H \right> = \left< v,  U^T\boldsymbol{1}_H \right> = \left< v,  \boldsymbol{1}_G \right>$.
	For all $v \in (\mathbb{Q}\mathcal{R}_G)^\bot$, $\left< v, X^T \boldsymbol{1}_H \right> = 0 = \left< v,  \boldsymbol{1}_G \right>$.
	Hence, $X^T \boldsymbol{1}_H = \boldsymbol{1}_G$.
\end{proof}

As noted above, all graph classes considered in previous sections enjoy properties stronger than inner-product compatibility.
The class of bounded degree trees however does not satisfy any of the aforementioned properties.

\begin{example}
	A \emph{$d$-ary tree} is a tree whose vertices have degree at most $d+1$.
	For $d \geq 1$, the family of $1$-labelled $d$-ary trees $\mathcal{T}^d$ with label at a vertex of degree at most one is inner-product compatible. 
\end{example}

The set $\mathcal{T}^d$ is closed under \emph{guarded Schur products}, i.e.\@ the $d$-ary operation~$\circledast^d$ defined as
$\circledast^d(\boldsymbol{R}^1,\dots,\boldsymbol{R}^d) \coloneqq \boldsymbol{A} \cdot (\boldsymbol{R}^1 \odot \cdots \odot \boldsymbol{R}^d)$ for $\boldsymbol{R}^1, \dots, \boldsymbol{R}^d \in \mathcal{T}^d$. This operation induces a $d$-ary multilinear map on $\mathbb{Q}^{V(G)}$ for every graph $G$, i.e.\@ $\circledast^d_G(u_1,\dots,u_d) \coloneqq \boldsymbol{A}_G ( u_1 \odot \cdots \odot u_d)$ for $u_1, \dots, u_d \in \mathbb{Q}^{V(G)}$.

\begin{theorem} \label{cor:dary}
	Let $d \geq 1$. For simple graphs $G$ and $H$, the following are equivalent:
	\begin{enumerate}
		\item $G$ and $H$ are homomorphism indistinguishable over the class of $d$-ary trees,\label{bt1}
		\item there exists a pseudo-stochastic matrix $X \in \mathbb{Q}^{V(H) \times V(G)}$ such that $X \boldsymbol{T}_G = \boldsymbol{T}_H$ for all $\boldsymbol{T} \in \mathcal{T}^d$,\label{bt2}
		\item there exists a pseudo-stochastic matrix $X \in \mathbb{Q}^{V(H) \times V(G)}$ such that $X$ preserves $\circledast^d$ on $\mathbb{Q}\mathcal{T}^d_G$, i.e.\@ $X(\circledast^d_G(u_1,\dots,u_d)) = \circledast^d_H (Xu_1, \dots, Xu_d)$ for all $u_1,\dots,u_d \in \mathbb{Q}\mathcal{T}^d_G$.\label{bt3}
	\end{enumerate}
\end{theorem}
\begin{proof}That \cref{bt1,bt2} are equivalent follows directly from \cref{thm:meta}. 
	
	Assuming \cref{bt2}, let $X \in \mathbb{Q}^{V(H) \times V(G)}$ be pseudo-stochastic such that $X \boldsymbol{T}_G = \boldsymbol{T}_H$ for all $\boldsymbol{T} \in \mathcal{T}^d$.
	Then for all $\boldsymbol{T}^1, \dots, \boldsymbol{T}^d \in \mathcal{T}^d$, $X (\circledast^d_G(\boldsymbol{T}^1_G, \dots, \boldsymbol{T}^d_G)) = X (\circledast^d(\boldsymbol{T}^1, \dots, \boldsymbol{T}^d))_G = (\circledast^d(\boldsymbol{T}^1, \dots, \boldsymbol{T}^d))_H = \circledast^d_H(X\boldsymbol{T}^1_G, \dots, X\boldsymbol{T}^d_G)$.
	
	Finally, \cref{bt1} follows inductively from \cref{bt3} observing that every $\boldsymbol{1} \neq \boldsymbol{T} \in \mathcal{T}^d$ can be written as $\boldsymbol{T} = \circledast^d(\boldsymbol{S}^1, \dots, \boldsymbol{S}^d)$ for some $ \boldsymbol{S}^1, \cdots, \boldsymbol{S}^d \in \mathcal{T}^d$ of lower depth.
\end{proof}

\section{Comparison to Known Systems of Equations}

\label{sec:lkiso}

In this section, the systems of equations obtained in \cref{thm:hom-pw,thm:trees,thm:treedepth-first} are compared to system of equations that have been previously studied.
This yields proofs of \cref{thm:pw-informal,thm:treedepth} and thereby a proof a conjecture from \cite{dell_lovasz_2018}.

We first recall Tinhofer's \textsmaller{LP} for graph isomorphism and its relaxations.
For graphs $G$ and $H$,
the standard \textsmaller{LP} relaxation for the graph isomorphism problem~\cite{tinhofer_graph_1986}, denoted by $\Fiso(G,H)$, is defined as follows. 
Its variables are $X_{wv}$ for every $v \in V(G)$ and $w \in V(H)$. 

\begin{center}
	\begin{fminipage}[.95\linewidth]
		$\Fiso(G,H)$
		\begin{align*}
			\sum_{v \in V(G)} X_{wv} &= 1 && \parbox[t]{.3\linewidth}{for all $w \in V(H)$,} \\
			\sum_{w \in V(H)} X_{wv} &= 1 && \parbox[t]{.3\linewidth}{for all $v \in V(G)$,} \\
			\sum_{v' \in V(G)} X_{wv'} \boldsymbol{A}_G(v',v) - \sum_{w' \in V(H)} \boldsymbol{A}_H(w,w') X_{w'v}&=0 && \parbox[t]{.3\linewidth}{for all $v \in V(G)$ \\ and $w \in V(H)$.}
		\end{align*}
	\end{fminipage}
\end{center}
The last equation can be concisely stated as a matrix equation $X \boldsymbol{A}_G = \boldsymbol{A}_H X$, where $\boldsymbol{A}_G$ and $\boldsymbol{A}_H$ are the adjacency matrices of $G$ and $H$ respectively. 

The existence of a solution of $\Fiso(G,H)$ over $\{0,1\}$  is equivalent to $G$ and $H$ being isomorphic. 
The existence of a non-negative rational solution in turn is equivalent to indistinguishability under the one-dimensional Weisfeiler--Leman algorithm~\cite{tinhofer_graph_1986}.

The system of equations $\Liso[k+1](G,H)$ is described next. Instead of variables $X_{wv}$ for vertices $v\in V(G)$, $w\in V(H)$, as in the system $\Fiso{}(G,H)$, the new system has variables $X_{\pi}$ for $\pi\subseteq V(H) \times V(G)$ of size $\abs{\pi} \le k$. 
We call $\pi=\{(w_1,v_1),\ldots,(w_\ell,v_\ell)\} \subseteq V(H)\times V(G)$ a
\emph{partial bijection} if $v_i=v_j\iff w_i=w_j$ for all $i,j$, and 
we call it a \emph{partial isomorphism} if in addition $v_iv_j\in
E(G)\iff w_iw_j\in E(H)$. Now consider the following
system of linear equations.

\begin{center}
	\begin{fminipage}[.95\linewidth]
		$\Liso[k+1](G,H)$
		\begin{align}
			\sum_{v\in V(G)} X_{\pi\cup\{(w, v)\}}&= X_\pi &&\parbox[t]{7.5cm}{for
				all $\pi\subseteq
				V(H)\times V(G)$ of size
				$\abs{\pi} \le k$ and
				all $w\in V(H)$,}\tag{L1}\label{sa1}\\
			\sum_{w\in V(H)} X_{\pi\cup\{(w, v)\}}&=X_\pi &&\parbox[t]{7.5cm}{for
				all $\pi\subseteq
				V(H)\times V(G)$ of size
				$\abs{\pi}\le k$ and
				all $v\in V(G)$,} \tag{L2}\label{sa2}\\
			X_\emptyset&=1 \tag{L3} \label{sa3}\\
			X_\pi&=0 &&\parbox[t]{7.5cm}{for all $\pi\subseteq
				V(H)\times V(G)$ of size
				$\abs{\pi}\le k+1$ such that
				$\pi$ is not a partial isomorphism.}\tag{L4}\label{sa4}
		\end{align}
	\end{fminipage}
\end{center}

Borrowing the terminology from \cite{grohe_pebble_2015},  \cref{sa1} and \cref{sa2} are referred to as \emph{continuity equations} while \cref{sa3} and \cref{sa4} are referred to as \emph{compatibility equations}. 

By \cite{malkin_sheraliadams_2014}, for $k \geq 1$,
$\Liso[k+1](G,H)$ has a non-negative rational solution if and only if
$G$ and $H$ are not distinguished by $k$\nobreakdash-dimensional Weisfeiler--Leman algorithm.
In particular, $\Fiso{}(G, H)$ has a non-negative rational solution if and only if $\Liso[2](G,H)$ has a non-negative rational solution.
By
\cite{dvorak_recognizing_2010,dell_lovasz_2018}, the existence of a non-negative rational solution to $\Liso[k+1](G,H)$ is thus equivalent to homomorphism indistinguishability over the class of graphs of treewidth at most~$k$. 

The system $\Liso[k+1](G,H)$ is closely related to the Sherali--Adams relaxations of
Tinhofer's~\textsmaller{LP} $\Fiso{}(G,H)$: every solution for the level-$(k+1)$ Sherali--Adams relaxation of $\Fiso(G,H)$ yields a solution to $\Liso[k+1](G,H)$, and every solution to $\Liso[k+1](G,H)$ yields a solution to the level-$k$ Sherali--Adams relaxation of $\Fiso(G,H)$ \cite{atserias_sherali-adams_2013,grohe_pebble_2015}.

\subsection{Sherali--Adams without Non-Negativity Constraints}

\label{sec:sa}

Dropping non-negativity constraints in $\Fiso(G,H)$ yields a system of linear equations whose feasibility characterises homomorphism indistinguishability over the class of paths~\cite{dell_lovasz_2018}. It was conjectured ibidem that dropping non-negativity constraints in $\Liso[k+1](G,H)$ analogously characterises homomorphism indistinguishability over graphs of pathwidth at most $k$.
One of the conjectured implications was already shown in \cite{dell_lovasz_2018}: the existence of a rational solution to $\Liso[k+1](G,H)$ implies homomorphism indistinguishability over graphs of pathwidth at most $k$. 

We resolve the aforementioned conjecture by showing that the system of equations $\Liso[k+1](G,H)$ is feasible if and only if the system of equations $\PW^{k+1}(G,H)$ stated in \cref{thm:hom-pw} is feasible.
The proof repeatedly makes use of the observation that the equations in $\Liso[k+1](G,H)$ can be viewed as equations in $\PW^{k+1}(G, H)$ where certain $(k+1,k+1)$-bilabelled graphs model the continuity and compatibility equations of $\Liso[k+1](G,H)$. 
Building on \cref{thm:hom-pw}, we obtain the following \cref{thm:sa} implying \cref{thm:pw-informal}.

\begin{theorem} \label{thm:sa}
	For $k \geq 1$ and simple graphs $G$ and $H$, the following are equivalent:
	\begin{enumerate}
		\item $G$ and $H$ are homomorphism indistinguishable over the class of graphs of pathwidth at most~$k$,\label{it:sa1}
		\item the system of equations $\mathsf{PW}^{k+1}(G, H)$ has a rational solution,\label{it:sa3}
		\item the system of equations $\Liso[k+1](G,H)$ has a rational solution.\label{it:sa2}
	\end{enumerate}
\end{theorem}
\begin{proof}
	Given \cref{thm:hom-pw}, it suffices to show that \cref{it:sa2,it:sa3} are equivalent.
	First consider the forward implication.
	Since $X$ already denotes the solution to the system $\PW^{k+1}(G,H)$ in \cref{thm:hom-pw}, the variables of $\Liso[k+1](G,H)$ will be denoted by $Y_\pi$ instead of $X_\pi$.
	Let $\mathfrak{S}_k$ denote the symmetric group acting on $k$ letters. For a vector $\vec{v} \in V(G)^k$ and $\sigma \in \mathfrak{S}_k$, write $\sigma(\vec{v})$ for the vector $v_{\sigma(1)} \dots v_{\sigma(k)}$.
	\begin{claim} \label{lem:symmetrisation}
		Let $k \geq 2$.
		Let $X$ denote a solution to $\PW^{k+1}(G,H)$. 
		Then $X(\sigma(\boldsymbol{w}), \sigma(\boldsymbol{v})) = X(\boldsymbol{w}, \boldsymbol{v})$ for all $\boldsymbol{v} \in V(G)^{k+1}$, $\boldsymbol{w} \in V(H)^{k+1}$, and $\sigma \in \mathfrak{S}_{k+1}$.
	\end{claim}
	\begin{claimproof}
		Recall the swap graph $\boldsymbol{S}^{ij}$ from \cref{lem:basal}.
		Let $\tau \in \mathfrak{S}_{k+1}$ denote the transposition~$(ij)$.
		The equation $\boldsymbol{S}^{ij}_HX = X\boldsymbol{S}^{ij}_H$ of $\PW^{k+1}(G, H)$ is equivalent to $X(\tau(\boldsymbol{w}), \boldsymbol{v}) = X(\boldsymbol{w}, \tau(\boldsymbol{v}))$ for all $\boldsymbol{v} \in V(G)^{k+1}$ and $\boldsymbol{w} \in V(H)^{k+1}$.
		Hence, writing $\sigma \in \mathfrak{S}_{k+1}$ as product of transpositions $\sigma = \tau_1 \cdots \tau_r$, it holds that
		$X(\sigma(\boldsymbol{w}), \boldsymbol{v}) = X(\tau_1 \cdots \tau_r(\boldsymbol{w}), \boldsymbol{v}) = X(\tau_2 \cdots \tau_r(\boldsymbol{w}), \tau_1(\boldsymbol{v})) = \dots = X(\boldsymbol{w}, \tau_r \cdots \tau_1(\boldsymbol{v})) = X(\boldsymbol{w}, \sigma^{-1}(\boldsymbol{v}))$, as desired.
	\end{claimproof}
	
	The following \cref{lem:induction} shows that \cref{sa1,sa2} hold.
	
	\begin{claim} \label{lem:induction}
		Let $\ell \geq 2$. Let $X$ denote a solution to $\PW^{\ell}(G,H)$. 
		Then
		\begin{equation}  \label{eq:continuity}
			\sum_{v' \in V(G)} X(\vec{w}w, \vec{v}v') =  \sum_{w' \in V(H)} X(\vec{w}w', \vec{v}v) \eqqcolon \breve{X}(\vec{w},\vec{v})
		\end{equation}
		for all $\boldsymbol{v} \in V(G)^\ell$, $\boldsymbol{w} \in V(H)^\ell$, $v \in V(G)$, and $w \in V(H)$.
		Furthermore, the matrix $\breve{X}$ defined above is a solution to $\PW^{\ell-1}(G, H)$.
	\end{claim}
	\begin{claimproof}
		The first equality in \cref{eq:continuity} is equivalent to $\boldsymbol{J}^{\ell}_H X = X \boldsymbol{J}^\ell_G$ where $\boldsymbol{J}^\ell$ denotes the forgetting graph from \cref{lem:basal}. 
		Clearly, $\breve{X}$ is pseudo-stochastic.
		It remains to argue that $\breve{X}$ is a solution to $\PW^{\ell-1}(G, H)$.
		
		Let $\boldsymbol{J}$ denote the $(1,1)$-bilabelled edgeless $2$-vertex graph whose labels reside on distinct vertices. For every $\boldsymbol{B} \in \mathcal{B}^{\ell-2}$, the $(\ell,\ell)$-bilabelled graph $\boldsymbol{B} \otimes \boldsymbol{J}$ can be written as series product of graphs in $\mathcal{B}^{\ell-1}$.
		Then for arbitrary $w \in V(H)$ and $v \in V(G)$
		\begin{align*}
			(\breve{X} \boldsymbol{B}_G)(\vec{w},\vec{v}) 
			= (X (\boldsymbol{B}_G \otimes \boldsymbol{J}_G))(\vec{w}w,\vec{v}v)
			= ((\boldsymbol{B}_H \otimes \boldsymbol{J}_H) X)(\vec{w}w,\vec{v}v)
			= (\boldsymbol{B}_H \breve{X})(\vec{w}, \vec{v}).
		\end{align*}
		This concludes the proof.
	\end{claimproof}
	
	It remains to consider \cref{sa3}:
	
	\begin{claim} \label{lem:iso}
		Let $\ell \geq 2$. Let $X$ be a solution to $\PW^\ell(G, H)$. Let $\pi = (\vec{w},\vec{v}) \subseteq V(H) \times V(G)$ be arbitrarily ordered of length $\ell$. Suppose that $\pi$ is not a partial isomorphism. Then $X(\vec{w},\vec{v}) = 0$.
	\end{claim}
	\begin{claimproof}
		If $\pi$ is not a partial isomorphism then there exists $i\neq j \in [\ell]$ such that wlog $\{v_i, v_j \} \in E(G)$ but $\{w_i, w_j \} \not\in E(H)$. This implies that $\mathbf{A}^{ij}_G(\vec{v},\vec{v}) = 1$ while $\mathbf{A}^{ij}_H(\vec{w},\vec{w}) = 0$. Moreover, $X(\vec{w},\vec{v})\mathbf{A}^{ij}_G(\vec{v},\vec{v}) =  (X\mathbf{A}^{ij}_G)(\vec{w},\vec{v}) = (\mathbf{A}^{ij}_HX)(\vec{w},\vec{v}) = \mathbf{A}^{ij}_H(\vec{w},\vec{w})X(\vec{w},\vec{v})$ since in- and out-labels coincide. Hence, $X(\vec{w},\vec{v}) = 0$.
	\end{claimproof}
	
	This concludes the preparations for proving that \cref{it:sa3} implies \cref{it:sa2}.
	Let $X$ denote a solution of $\PW^{k+1}(G, H)$.
	Construct via \cref{lem:induction} solutions $X^\ell$ of $\PW^\ell(G, H)$ for $1 \leq \ell \leq k+1$ satisfying \cref{eq:continuity}. 
	For $\pi = (\vec{w}, \vec{v}) \subseteq V(H) \times V(G)$ define $Y_\pi = X^{\abs{\pi}}(\vec{w}, \vec{v})$. 
	By \cref{lem:symmetrisation}, the entries of $X$ do not depend on the ordering of the indices. Thus, $Y$ is well-defined. 
	Let finally $Y_\emptyset = 1$. Then \cref{sa1,sa2} hold by \cref{eq:continuity}. 
	\Cref{sa4} holds by \cref{lem:iso}, and \cref{sa3} by definition. Thus, \cref{it:sa2} holds.
	
	That \cref{it:sa2} implies \cref{it:sa1} was shown by \cite{dell_lovasz_2018}.
	For the sake of completeness, we prove that \cref{it:sa2} implies \cref{it:sa3}.
	To that end, let $Y$ denote a solution of $\Liso[k+1](G,H)$.
	Define a candidate solution $X$ for $\PW^{k+1}(G, H)$ by letting $X(\vec{w}, \vec{v}) \coloneqq Y_{(\vec{w}, \vec{v})}$. By repeatedly applying \cref{sa1,sa2} it follows that $X$ is pseudo-stochastic. For commutation with the graphs from \cref{lem:basal}, four cases have to be considered: Let $\vec{v} \in V(G)^{k+1}$, $\vec{w} \in V(H)^{k+1}$. Recall \cref{sa3}.
	\begin{enumerate}
		\item Let $\boldsymbol{A}^{ij}$ be an adjacency graph. If $(\vec{w}, \vec{v})$ is a partial isomorphism then $\boldsymbol{A}^{ij}_G(\vec{v},\vec{v}) = \boldsymbol{A}^{ij}_H(\vec{w}, \vec{w})$ and thus $X$ and $\boldsymbol{A}^{ij}$ commute. If $(\vec{w}, \vec{v})$ is not a partial isomorphism then $X(\vec{w} , \vec{v}) = 0$ and the tensors commute as well.
		\item Let $\boldsymbol{J}^\ell$ be a forgetting graph. To ease notation, suppose $\ell = k+1$. Then by \cref{sa1,sa2}, writing $\vec{v} = v_1\dots v_{k+1}$, and $\vec{w} = w_1 \dots w_k$,
		\begin{align*}
			(X\boldsymbol{J}^{k+1}_G)(\vec{w},\vec{v})
			&= \sum_{x \in V(G)} X(\vec{w}, v_1 \dots v_k x)  
			= \sum_{x \in V(G)} Y_{\{(w_1, v_1), \dots, (w_k, v_k), (w_{k+1}, x)\}} \\
			&\stackrel{\eqref{sa1}}{=} Y_{\{(w_1, v_1), \dots, (w_k, v_k) \}} 
			\stackrel{\eqref{sa2}}{=} \sum_{y \in V(H)} Y_{\{(w_1, v_1), \dots, (w_k, v_k), (y, v_{k+1})\}} \\
			&=\sum_{y \in V(H)} X(w_1 \dots w_k y, \vec{v}) 
			= (\boldsymbol{J}^{k+1}_HX)(\vec{w},\vec{v}).
		\end{align*}
		\item For the swap graphs, the statement follows as in \cref{lem:symmetrisation} since the value of $X$ does not depend on the ordering of the indices.\qedhere
	\end{enumerate}
\end{proof}

As a corollary, we show that $\mathsf{PW}^{k+1}(G, H)$ has a non-negative rational solution if and only if $\Liso[k+1](G,H)$ has a non-negative rational solution. Consequently, the system of linear equations $\mathsf{PW}^{k+1}(G, H)$ has a non-negative rational solution if and only if $G$ and $H$ are homomorphism indistinguishable over graphs of treewidth at most $k$. Hence, the systems of equations 
$\mathsf{PW}^{k+1}(G, H)$, for $k \in \mathbb{N}$, form an alternative well-motivated hierarchy of linear programming relaxations of the graph isomorphism problem.
Examining the proof of \cref{thm:sa}, we obtain the following corollary. 

\begin{corollary} \label{coro:sa}
	Let $k \geq 1$.  Let $G$ and $H$ be simple graphs. Then the following are equivalent:
	\begin{enumerate}
		\item $G$ and $H$ are homomorphism indistinguishable over the class of graphs of treewidth at most~$k$,\label{corit:sa1}
		\item $\mathsf{PW}^{k+1}(G, H)$ has a non-negative rational solution,\label{corit:sa3}
		\item $\Liso[k+1](G,H)$ has a non-negative rational solution.\label{corit:sa2}
	\end{enumerate}
\end{corollary}
\begin{proof}The equivalence of \cref{corit:sa1,corit:sa3} follows from \cref{thm:trees}.
	A combination of results of \cite{atserias_sherali-adams_2013, grohe_pebble_2015} and \cite{dvorak_recognizing_2010, dell_lovasz_2018} yields that \cref{corit:sa1,corit:sa2} are equivalent.
	For a self-contained argument, observe that the transformations devised in \cref{lem:symmetrisation,lem:induction} preserve non-negativity.
\end{proof}

\subsection{An Ordered Variant of Sherali--Adams and Bounded Treedepth}

In this section, we reinterpret the systems of equations in \cref{thm:treedepth-first} as ordered variant of $\Liso[k+1](G,H)$.
In contrast to $\Liso[k+1](G,H)$, it is equivalent for this system to have an arbitrary rational solution and a non-negative rational solution.

Let $k \geq 1$. For graphs $G$ and $H$, consider the following system of equations $\mathsf{TD}^k(G, H)$ with variables $X(\vec{w}, \vec{v})$ for every pair of tuples $\vec{w} \in V(H)^\ell$ and $\vec{v} \in V(G)^\ell$ for $0 \leq \ell \leq k$. A length-$\ell$ pair $(\vec{w}, \vec{v})$ is said to be a \emph{partial strong homomorphisms} if $v_iv_j \in E(G) \iff w_iw_j \in E(H)$ for all $i, j \in [\ell]$. 

\begin{center}
	\begin{fminipage}[.95\linewidth]
		$\mathsf{TD}^k(G,H)$
		\begin{align}
			\sum_{v' \in V(G)} X(\vec{w}w, \vec{v}v') = X(\vec{w}, \vec{v}) &&
			\parbox[t]{6cm}{for all $w \in V(H)$ and $\vec{v} \in V(G)^\ell$, $\vec{w} \in V(H)^\ell$ where $0 \leq \ell < k$.} \tag{TD1}\label{eq:td1} \\
			\sum_{w' \in V(H)} X(\vec{w}w', \vec{v}v) = X(\vec{w}, \vec{v}) &&
			\parbox[t]{6cm}{for all $v \in V(G)$ and $\vec{v} \in V(G)^\ell$, $\vec{w} \in V(H)^\ell$ where $0 \leq \ell < k$.} \tag{TD2}\label{eq:td2} \\
			X((),()) = 1 && \tag{TD3}\label{eq:td3} \\
			X(\vec{w},\vec{v}) = 0 && \parbox[t]{6cm}{whenever $(\vec{w}, \vec{v}) \in V(H)^\ell \times V(G)^\ell$ for $1 \leq \ell \leq k$ is not a partial strong homomorphism} \tag{TD4}\label{eq:td4}
		\end{align}
	\end{fminipage}
\end{center}

The main result of this section is the following:

\treedepth*
\begin{proof}
	Given \cref{thm:treedepth-first}, it suffices to show that $\mathsf{TD}^k(G, H)$ has a (non-negative) rational solution if and only if there is a (doubly stochastic) pseudo-stochastic matrix $X$ such that  $X \boldsymbol{B}_G = \boldsymbol{B}_H X$ for all $\boldsymbol{B} \in \mathcal{TDB}^k$.
	
	We first consider the backward implication.
	Given $X \in \mathbb{Q}^{V(H)^k \times V(G)^k}$ such that $X \boldsymbol{B}_G = \boldsymbol{B}_H X$ for all $\boldsymbol{B} \in \mathcal{TDB}^k$,  observe that if $(\vec{w}, \vec{v}) \in V(H)^k \times V(G)^k$ is not a partial strong homomorphism then $X(\vec{w}, \vec{v}) = 0$.
	Indeed, in this case there exist $1 \leq i \neq j \leq k$ such that $w_iw_j \in E(F) \not\Leftrightarrow v_iv_j \in E(F)$. In particular, precisely one of $\boldsymbol{A}^{ij}_G(\vec{v}, \vec{v})$ and $\boldsymbol{A}^{ij}_H(\vec{w}, \vec{w})$ is zero. Then $\boldsymbol{A}^{ij}_H(\vec{w}, \vec{w}) X(\vec{w}, \vec{v}) = X(\vec{w},\vec{v}) \boldsymbol{A}^{ij}_G(\vec{v}, \vec{v})$ implies that $X(\vec{w},\vec{v}) = 0$. Hence \cref{eq:td4} holds for $\ell = k$.
	With this observation at hand, define a solution to $\mathsf{TD}^k(G, H)$ by invoking the following claim whose proof is analogous to the proof of \cref{lem:induction}.
	\begin{claim} \label{claim:treedepth-induction}
		Let $k \geq 1$.
		If $X \in \mathbb{Q}^{V(H)^{k+1} \times V(G)^{k+1}}$
		is doubly stochastic and such  that $X \boldsymbol{B}_G = \boldsymbol{B}_H X$ for all $\boldsymbol{B} \in \mathcal{TDB}^{k+1}$ then 
		\[
		\sum_{v' \in V(G)} X(\vec{w}w, \vec{v}v') =  \sum_{w' \in V(H)} X(\vec{w}w', \vec{v}v) \eqqcolon \breve{X}(\vec{w},\vec{v})
		\]
		for all $v \in V(G)$, $w \in V(H)$, and $\vec{v} \in V(G)^k$, $\vec{w} \in V(H)^k$. Furthermore, $\breve{X} \in \mathbb{Q}^{V(H)^k \times V(G)^k}$ is doubly stochastic and such that  $\breve{X} \boldsymbol{B}_G = \boldsymbol{B}_H \breve{X}$ for all $\boldsymbol{B} \in \mathcal{TDB}^k$. 
	\end{claim}
A solution to $\mathsf{TD}^k(G, H)$ can now be defined inductively invoking \cref{claim:treedepth-induction}.
	\cref{eq:td1,eq:td2} are immediate from \cref{claim:treedepth-induction}. \cref{eq:td3} follows since double stochasticity is preserved throughout the induction. \cref{eq:td4} follows as observed initially.

	Conversely, define a matrix $X \in \mathbb{Q}^{V(H)^k \times V(G)^k}$ by extracting the values on $V(H)^k \times V(G)^k$ from a solution $Y$ to $\mathsf{TD}^k(G, H)$. By \cref{eq:td1,eq:td2,eq:td3}, $X$ is pseudo-stochastic. It remains to consider commutation with the graphs from \cref{lem:tdk-generates}: Let $(\vec{w}, \vec{v}) \in V(H)^k \times V(G)^k$.
	\begin{enumerate}
		\item If $(\vec{w}, \vec{v})$ is a partial strong homomorphism then $\boldsymbol{A}_G^{ij}(\vec{v}, \vec{v})  = \boldsymbol{A}_H^{ij}(\vec{w}, \vec{w})$ for all $1 \leq i \neq j \leq k$. Hence, $\boldsymbol{A}^{ij}_H(\vec{w}, \vec{w}) X(\vec{w}, \vec{v}) = X(\vec{w},\vec{v}) \boldsymbol{A}^{ij}_G(\vec{v}, \vec{v})$.
		If $(\vec{w}, \vec{v})$ is not a partial strong homomorphism then the same assertions follows readily from \cref{eq:td4}.
		\item For $0 \leq \ell < k$, by repeatedly applying \cref{eq:td1,eq:td2},
		\[
		(X\boldsymbol{J}_G^\ell)(\vec{w},\vec{v})
		= \sum_{v_{\ell+1},\dots, v_k \in V(G)} X(\vec{w}, v_1\dots v_\ell v_{\ell+1}\dots v_k)
		= Y(w_1\dots w_\ell, v_1\dots v_\ell),
		\]
		which in turn is equal to $(\boldsymbol{J}_H^\ell X)(\vec{w},\vec{v})$. \qedhere
	\end{enumerate}
\end{proof}

\section{Homomorphism Counts from Trees Cannot Be Inferred from Homomorphism Counts from Bounded Degree Trees}

In this section, we utilise the techniques introduced in \cref{sec:bounded-degree-trees} to show that homomorphism indistinguishability over bounded degree trees is a strictly finer relation than homomorphism indistinguishability over all trees.
As a consequence \cite{cai_optimal_1992,dvorak_recognizing_2010}, it is not possible to simulate the $1$-dimensional Weisfeiler--Leman algorithm (also known as Colour Refinement)
by counting homomorphisms from trees of any fixed bounded degree.

\degreeGH*

After this work was first presented, Roberson~\cite{roberson_oddomorphisms_2022} showed the following stronger statement: For every $d$, the class of trees of degree at most $d$ is \emph{homomorphism distinguishing closed}, i.e.\@ for every graph $F$ not from this class there exists graphs $G$ and $H$ which are homomorphism indistinguishable over trees of degree at most $d$ but $\hom(F, G) \neq \hom(F, H)$.
Roberson's proof intricately exploits combinatorial properties of \textsmaller{CFI} graphs \cite{cai_optimal_1992}.
Despite being weaker, our \cref{thm:degreeGH} is proven by introducing a novel construction of the graphs $G$ and $H$ which is not akin to the ubiquitous \textsmaller{CFI} construction. 

Towards \cref{thm:degreeGH}, we first show that the nested subspaces $\mathbb{Q}\mathcal{T}_G^d$ for $d \geq 1$ by which the systems of equations in \cref{cor:dary} are parametrised form a strict chain.

\begin{theorem}\label{thm:degree}
	For every integer $d \geq 1$, there exists a simple graph $H$ such that 
	$\mathbb{Q}\mathcal{T}_{H}^d \neq \mathbb{Q}\mathcal{T}_{H}^{d+1}$.
\end{theorem}

\subsection{Proof of \cref{thm:degree}}
\label{sec:proof:thm:degree}

The graph in \cref{thm:degree} will be constructed from a vector of distinct integers. As a first step, a multigraph adjacency matrix is constructed in \cref{prop:degree}. This multigraph is then turned into a simple graph in \cref{lem:multigraph} yielding the object with the desired properties (\cref{lem:spaces}).

All constructions in this section take place in $\mathbb{Q}^n$ for some integer $n$ to be fixed. Let $\allones = \allones_n$ denote the all-ones vector in this space.  Recall from \cref{sec:linalg} that $\odot$ denotes the vector Schur product.
The following \cref{lem:principal} follows by applying arguments similar to \cref{lemma:gs,fact:vd}.

\begin{lemma}[Principal Sequence] \label{lem:principal}
	Let $n \geq 3$.
	Let $a \in \mathbb{N}^n$ be a vector with distinct entries. Let
	\begin{equation} \label{eq:principal-sequence1}
		u_0 \coloneqq \allones + a, \quad u_1 \coloneqq \allones -a, \quad u_i \coloneqq a^{\odot i},
	\end{equation}
	for $2 \leq i \leq n-1$. Define the \emph{principal sequence of $a$} as the sequence of the vectors
	\begin{equation} \label{eq:principal-sequence2}
		v_0 \coloneqq u_0, \quad 
		v_{i+1} \coloneqq u_{i+1} - \sum_{j = 0}^i \frac{\left< u_{i+1}, v_j \right>}{\left< v_j, v_j \right>} v_j
	\end{equation}
	for $0 \leq i \leq n - 2$. Then
	\begin{enumerate}
		\item for $1 \leq i \leq n-1$, the vectors $v_0, \dots, v_i$ span the same subspace as $u_0, \dots, u_i$,
		\item the vectors $v_0, \dots, v_{n-1}$ are mutually orthogonal,
		\item the vectors $v_0, \dots, v_{n-1}$ lie in $\mathbb{Q}^n$.
	\end{enumerate}
\end{lemma}

\begin{lemma}\label{prop:degree}
	Let $d \geq 1$ and $n > d+1$.
	Given a vector $a \in \NN^n$ with distinct entries, 
	there exists a symmetric matrix $M \in \mathbb{N}^{n \times n}$ satisfying the following properties:
	\begin{enumerate}
		\item \label{prop:degree:part1} if $x_1,\dots,x_d \in S_{\leq 1}$ then $x_1 \odot \cdots \odot x_d \in S_{\leq d}$,
		\item \label{prop:degree:part2} if $x_1,\dots,x_d \in S_{\leq 1}$ then $M (x_1 \odot \cdots \odot x_d) \in S_{\leq 1}$,
		\item \label{prop:degree:part3} it holds that $M((M \allones)^{\odot (d+1)}) \not\in S_{\leq d}$. 
	\end{enumerate}
	Here, $S_{\leq i}$ denotes the subspace of $\mathbb{Q}^n$ spanned by the first $i+1$ vectors of the principal sequence of $a$.
\end{lemma}

\begin{proof}
	Construct invoking \cref{lem:principal} the principal sequence $v_0, \dots, v_{n-1} \in \mathbb{Q}^n$ of the vector $a$.
	Define the matrix $M \in \mathbb{N}^{n\times n}$ as follows. First let 
	\begin{equation} \label{eq:matrixm}
		M_\lambda \coloneqq \lambda v_0 {v_0}^T + v_1 {v_1}^T + v_{d+1}{v_{d+1}}^T
	\end{equation} 
	where $\lambda \in \NN$ is chosen to be a sufficiently large positive integer such that $M_\lambda$ has only positive entries, and $\boldsymbol{1}$ is not an eigenvector of $M_\lambda$.
	This is possible because the matrix $v_0 {v_0}^T$ defined via $v_0 = a + \boldsymbol{1}$ has only positive entries, and $\boldsymbol{1}$ is none of its eigenvectors.
	Define $M \in \NN^{n\times n}$ as the matrix obtained from $M_\lambda$ for appropriate $\lambda$ by clearing all denominators, i.e.\ by multiplying with a sufficiently large $\lambda' \in \mathbb{N}$.
	Observe that $M$ is a symmetric matrix with positive integral entries. Moreover, the rank of $M$ is exactly $3$. 
	
	Having defined $M$, it remains to verify the assertions.
	For \cref{prop:degree:part1}, if $x_1,\dots,x_d \in S_{\leq 1}$ then each of these vectors is a linear 
	combination of $\boldsymbol{1}$ and $a$. By bilinearity of the Schur product, the product 
	$x_1 \odot \cdots \odot x_d$ must be a linear combination of the vectors $\boldsymbol{1}, \dots, a^{\odot d}$. Hence, $x_1 \odot \cdots \odot x_d \in S_{\leq d}$. 
	
	For \cref{prop:degree:part2},  if $x_1,\dots,x_d \in S_{\leq 1}$ then $x_1 \odot \cdots \odot x_d \in S_{\leq d}$. Since $M$ annihilates all vectors in $S_{\leq d}$ except for those in $S_{\leq 1}$, it holds that $M(x_1 \odot \cdots \odot x_d) \in S_{\leq 1}$. 
	
	For \cref{prop:degree:part3}, write $M \boldsymbol{1} = \alpha \boldsymbol{1} + \beta a$, for some $\alpha,\beta \in \mathbb{Q}$. It holds that $\beta \neq 0$ because $\boldsymbol{1}$ is not an eigenvalue of $M$ by construction. The resulting expansion of $(M\allones)^{\odot (d+1)}$ is a linear combination of the vectors $\boldsymbol{1}, a , a \odot a, \dots, a^{\odot (d+1)}$. Moreover, in this binomial expansion, the coefficient of $a^{\odot (d+1)}$ is $\beta^{d+1}$, which is non-zero. The remaining terms in this expansion lie in $S_{\leq d}$ already. Hence, $(M\allones)^{\odot (d+1)} \not\in S_{\leq d}$. Furthermore, the coefficient of $a^{\odot (d+1)}$ in $M(M\allones)^{\odot (d+1)}$ is also non-zero because $v_{d+1}$ is a non-vanishing eigenvector of $M$. Hence, $M(M\allones)^{\odot (d+1)} \not\in S_{\leq d}$.
\end{proof} 

Since the matrix $M$ is symmetric and has non-negative integral entries, it can be thought of as the adjacency matrix of a multigraph, i.e.\@ an undirected graph with multiedges. The objective is to construct a simple undirected graph~$H$ such that the eigensystem of the adjacency matrix $H$ is closely related to that of~$M$.

\begin{lemma}\label{lem:multigraph}
	For every symmetric matrix~$M \in \NN^{n \times n}$, there exists an integer $N$ and a simple graph $H$ with vertex set $[n] \times [N]$ such that $\boldsymbol{A}_H\widehat{z} = (Mz) \otimes \boldsymbol{1}_{N}$ for all $z \in \mathbb{Q}^n$ with $\widehat{z} \coloneqq z \otimes \allones_N$.
\end{lemma}
\begin{proof}
The graph $H$ is constructed from $M$ as follows. 
	Fix a large positive integer~$N$ such that it is divisible by every entry of~$M$. 
	The vertex set of $H$ is defined to be $[n] \times [N]$. 
	The edge set of $H$ is determined as follows.
	\begin{itemize}
		\item For each $u \in [n]$, the graph induced on the vertex subset $\{u\} \times [N]$ is set to be an arbitrary regular graph of degree $M_{uu}$.
		\item For distinct $u, v \in [n]$, the bipartite graph induced between the vertex subsets $\{u\} \times [N]$ and $\{v\} \times [N]$ is set to be an arbitrary bi-regular graph with left and right degree $M_{uv} = M_{vu}$.
	\end{itemize}
	The generous choice of $N$ ensures that the arbitrary regular and bi-regular graphs stipulated above exist, cf.\@ e.g.\@ \cite[Lemma~3.2]{kiefer_graphs_2022}.
	Write $A \coloneqq \boldsymbol{A}_H$ for the adjacency matrix of $H$.
	The entry of $\boldsymbol{A}_H(z \otimes \boldsymbol{1}_{N})$ corresponding to a vertex $(u,j) \in V(H)$ is equal to 
	\begin{align*}
		\sum_{(u',j') \in V(H)} A_{(u,j),(u',j')} (z \otimes \boldsymbol{1}_{N})_{(u',j')}  
		&= \sum_{u' \in [n]} \sum_{j' \in [N]}  A_{(u,j),(u',j')} (z \otimes \boldsymbol{1}_{N})_{(u',j')} \\
		&= \sum_{u' \in [n]} \sum_{j' \in [N]}  A_{(u,j),(u',j')} z_{u'}\\
		&= \sum_{u' \in [n]} z_{u'} \sum_{j' \in [N]}  A_{(u,j),(u',j')}  \\
		&= \sum_{u' \in [n]} M_{uu'} z_{u'}  \\
		&= (Mz)_u 
	\end{align*}
	which does not depend on $j$.
\end{proof}

Let $\boldsymbol{S}^{d+1} \coloneqq \boldsymbol{A} \cdot (\boldsymbol{A}\boldsymbol{1})^{\odot d+1}$ denote the labelled star graph with label at one of its $d+2$ leaves. Clearly, $\boldsymbol{S}^{d+1} \in \mathcal{T}^{d+1} \setminus \mathcal{T}^d$.

\begin{lemma}\label{lem:spaces}
	Let $d \geq 1$ and $n > d+1$.
	Let $a \in \NN^n$ be a vector with distinct entries and principal sequence $(v_0, \dots, v_{n-1})$.
	Let $H$ denote the graph constructed from $a$ via \cref{prop:degree,lem:multigraph}.
	Then
	\begin{enumerate}
		\item for every labelled tree $\boldsymbol{B} \in \mathcal{T}^d$, the vector $\boldsymbol{B}_{H}$ belongs to the span of $\{ \widehat{v_0}, \widehat{v_1} \}$.\label{claim:daryspan}
		\item the vector $\boldsymbol{S}^{d+1}_H$ does not belong to the span of $\{ \widehat{v_0}, \widehat{v_1} \}$.\label{claim:starspan}
	\end{enumerate}
\end{lemma}
\begin{proof}
	Observe that for $x, y \in \mathbb{Q}^n$, $\widehat{x} \odot \widehat{y} = (x \otimes \boldsymbol{1}_N) \odot (y \otimes \boldsymbol{1}_N) = (x \odot y) \otimes \boldsymbol{1}_N$. So the assertions of \cref{prop:degree} on the layout of the eigenspaces of $M$ carry over to the eigenspaces of $\boldsymbol{A}_H$.
	
	First \cref{claim:daryspan} is proved by structural induction on $\boldsymbol{B}$.
	For the base case, $\boldsymbol{B}$ is the $1$-labelled one-vertex graph, and hence, $\boldsymbol{B}_{H} = \allones_{H} = \allones_{n} \otimes \allones_{N}$.
	Since $\allones_n$ lies in the span of $\{ v_0, v_1 \}$, $\boldsymbol{B}_{H}$ lies in the span of $\{ \widehat{v_0}, \widehat{v_1}\}$. For the inductive step, suppose $\boldsymbol{B} = \circledast^d(\boldsymbol{C}^1, \dots, \boldsymbol{C}^d) = \boldsymbol{A}(\boldsymbol{C}^1 \odot \cdots \odot \boldsymbol{C}^d)$ for smaller $\boldsymbol{C}^1, \dots,\boldsymbol{C}^d \in \mathcal{T}^d$. By \cref{prop:degree,lem:multigraph}, and the inductive hypothesis, the $\boldsymbol{C}^i_H$ lie in the span of 
	$\{ \widehat{v_0}, \widehat{v_1}\}$, and $\boldsymbol{A}_H(\boldsymbol{C}^1_H \odot \cdots \odot\boldsymbol{C}^d_H)$ also lies in the span of $\{ \widehat{v_0}, \widehat{v_1}\}$.

	For \cref{claim:starspan}, observe that by \cref{lem:multigraph},
	\begin{align}
		\boldsymbol{S}^{d+1}_H &= \boldsymbol{A}_H (\boldsymbol{A}_H \allones_H)^{\odot (d+1)}
		= \boldsymbol{A}_H( M\allones_n \otimes \allones_N)^{\odot (d+1)}
		= \boldsymbol{A}_H ( (M\allones_n)^{\odot (d+1)} \otimes \allones_N ) \nonumber \\
		&= (M(M\allones_n)^{\odot (d+1)}) \otimes \allones_N. \label{eq:homstar}
	\end{align}
	By \cref{prop:degree:part3} of \cref{prop:degree}, $\boldsymbol{S}^{d+1}_H$ is not contained in the span of $\{ \widehat{v_0}, \widehat{v_1} \}$.
\end{proof}

This completes the preparations for the proof of \cref{thm:degree}.

\begin{proof}[Proof of \Cref{thm:degree}]
	For the given degree $d$, choose $n > d+1$ and an arbitrary $a \in \NN^n$ with distinct entries. Construct the desired graph $H$ via \cref{prop:degree,lem:multigraph}. \Cref{thm:degree} then immediately follows from \cref{lem:spaces}.
\end{proof}

\subsection{Proof of \Cref{thm:degreeGH}}
\label{sec:proof:thm:degreeGH}

For a vector $x \in \mathbb{Q}^n$ and $p \in \NN$ write $\norm{x}_p = \left(\sum_{i=1}^n \abs{x_i}^p \right)^{1/p}$ for the \emph{$p$-norm} of $x$.

\begin{lemma} \label{lem:2graphs}
	Let $d \geq 1$ and $n > d+1$.
	Let $a, b \in \NN^n$ be vectors with distinct entries. Suppose that $\norm{a}_i = \norm{b}_i$ for all $1 \leq i \leq 2d$, and that there exists $\ell$ such that $\norm{a}_\ell \neq \norm{b}_\ell$. Then there exist simple graphs $G$ and $H$ such that
	\begin{enumerate}
		\item $G$ and $H$ are homomorphism indistinguishable over the class of $d$-ary trees.\label{as:similar}
		\item $G$ and $H$ are not homomorphism indistinguishable over the class of all trees.\label{as:nottoosimilar}
	\end{enumerate}
\end{lemma}
\begin{proof}
	Construct symmetric matrices $M$ and $L$ with non-negative integral entries from $a$ and $b$, respectively, invoking \cref{prop:degree}, and subsequently convert these into the adjacency matrices of simple graphs $G$ and $H$, respectively, via \cref{lem:multigraph}. In this process, the integers $\lambda$, $\lambda'$, and $N$ are chosen to be respectively equal for $G$ and $H$, which is always feasible.
	
	For \cref{as:similar}, let $V \leq \mathbb{Q}^{n}$ denote the space spanned by $\allones, a, \dots, a^{\odot d}$, and analogously $W \leq \mathbb{Q}^n$ the space spanned by $\allones, b, \dots, b^{\odot d}$. Observe that for $0 \leq i, j \leq d$,
	\[
	\left< a^{\odot i}, a^{\odot j} \right>
	= \norm{a}_{i+j}^{i+j}
	= \norm{b}_{i+j}^{i+j}
	= \left< b^{\odot i}, b^{\odot j} \right>
	\]
	Hence, \cref{lemma:gs} guarantees the existence of an orthogonal $U \colon V \to W$ such that $Ua^{\odot i} = b^{\odot i}$ for all $0 \leq i \leq d$.
	Let $v_0, \dots, v_d$ and $w_0, \dots, w_d$ denote the length-$(d+1)$ initial segment of the principal sequence of $a$ and of $b$, respectively, cf.\ \cref{lem:principal}.
	\begin{claim} \label{cl:maps-principal}
		For all $0 \leq i \leq d$, $Uv_i = w_i$.
	\end{claim}
	\begin{claimproof}
		The claim is shown by induction on the definition of the principal sequence in \cref{lem:principal}.
		Let $u_0, \dots, u_d$ and $u'_0, \dots, u'_d$ be for $a$ and $b$ as in \cref{eq:principal-sequence1}.
		Clearly, $Uu_i = u'_i$ for all $0 \leq i \leq d$. 
		Thus, for $i = 0$, $Uv_0 = U(\boldsymbol{1} + a) = \boldsymbol{1} + b = w_0$.
		Furthermore, for $i \geq 0$,
		\begin{align*}
			Uv_{i+1} & = Uu_{i+1} +  \sum_{j = 0}^i \frac{\left< u_{i+1}, v_j \right>}{\left< v_j, v_j \right>} U v_j
			= u'_{i+1} + \sum_{j = 0}^i \frac{\left< Uu_{i+1}, Uv_j \right>}{\left< Uv_j, Uv_j \right>} w_j
			= u'_{i+1} + \sum_{j = 0}^i \frac{\left< u'_{i+1}, w_j \right>}{\left< w_j,w_j \right>} w_j,
		\end{align*}
		by orthogonality of $U$, as desired.
	\end{claimproof}
	Recall from \cref{lem:multigraph}, that $\widehat{z} = z \otimes \allones_N$ denotes the lift of a vector $z \in \mathbb{Q}^n$ to a vector in $\mathbb{Q}^{V(G)} = \mathbb{Q}^{V(H)}$. Write $\widehat{V}$ and $\widehat{W}$ for the lifts of all vectors in $V$ and in $W$, respectively. The map $U$ can be lifted to a map $\widehat{U} \colon \widehat{V} \to \widehat{W}$ via $\widehat{U} \widehat{z} \coloneqq (Uz) \otimes \boldsymbol{1}_N$, $z \in V$. Checking that $\widehat{U}$ is well-defined and unitary amounts to a straightforward calculation.
	
	Towards proving in \cref{cl:uhatsolution} that $\widehat{U}$ is a solution to the system of equations in \cref{cor:dary}, we show the following claim:
	
	\begin{claim}\label{claim:uhatcommuteswitha}
		$\widehat{U}\boldsymbol{A}_G = \boldsymbol{A}_H \widehat{U}$ as maps $\widehat{V} \to \widehat{W}$. 
	\end{claim}
	\begin{claimproof}
		By \cref{lem:multigraph,cl:maps-principal}, for $0 \leq i \leq d$,
		\[
		\widehat{U}\boldsymbol{A}_G \widehat{v_i}
		= \widehat{U} (Mv_i \otimes \allones_N)
		= \mu_i \widehat{U} (v_i \otimes \allones_N)
		= \mu_i(Uv_i) \otimes \allones_N
		= \mu_i w_i \otimes \allones_N
		= L w_i \otimes \allones_N
		= \boldsymbol{A}_H \widehat{U} \widehat{v_i}
		\]
		for $\mu_i \in \{0, 1, \lambda\}$ the $i$-th eigenvalue of $M$ and $L$, cf.\@ \cref{eq:matrixm}.
	\end{claimproof}
	
	\begin{claim} \label{cl:uhatsolution}
		For all $\boldsymbol{B} \in \mathcal{T}^{2d}$, $\widehat{U} \boldsymbol{B}_G = \boldsymbol{B}_H$.
	\end{claim}
	\begin{claimproof}
		By structural induction on $\boldsymbol{B}$. In the base case $\boldsymbol{B} = \boldsymbol{1}$, the claim holds by construction observing that $\boldsymbol{1}_G = \boldsymbol{1}_n \otimes \boldsymbol{1}_N = \widehat{\boldsymbol{1}_n}$.
		For the inductive step, suppose that $\boldsymbol{B} = \circledast^d(\boldsymbol{C}^1, \dots, \boldsymbol{C}^d)$.
		By \cref{claim:daryspan} of \cref{lem:spaces}, write $\boldsymbol{C}^i_G = \alpha_i \allones_G + \beta_i \widehat{a}$ for some $\alpha_i, \beta_i \in \mathbb{Q}$. By induction, $\alpha_i \allones_H + \beta_i \widehat{b} = \widehat{U}\boldsymbol{C}^i_G = \boldsymbol{C}^i_H$.
		The vector $\boldsymbol{C}^1_G \odot \dots \odot \boldsymbol{C}^d_G$ can be written as linear combination of the vectors $\allones_G, \widehat{a}, \dots, \widehat{a}^{\odot d}$ since each individual vector is a linear combination of $\allones_G$ and $\widehat{a}$.
		Hence, since $U a^{\odot i} = b^{\odot i}$ for $0 \leq i \leq d$, $\widehat{U}(\boldsymbol{C}^1_G \odot \dots \odot \boldsymbol{C}^d_G) = \boldsymbol{C}^1_H \odot \dots \odot \boldsymbol{C}^d_H$. By \cref{claim:uhatcommuteswitha}, $\widehat{U}$ preserves $\circledast^d$ and thus $\widehat{U}\boldsymbol{B}_G = \boldsymbol{B}_H$.
	\end{claimproof}
	
	By \cref{cor:dary}, $G$ and $H$ are homomorphism indistinguishable over the class of $d$-ary trees.
	It remains to verify \cref{as:nottoosimilar}. Let $\ell$ be the least positive integer such that $\norm{a}_\ell \neq \norm{b}_\ell$.
	
	\begin{claim} \label{cl:same-coeff}
		There exist coefficients $\alpha, \beta \in \mathbb{Q}$ such that  $M \boldsymbol{1}_n = \alpha \boldsymbol{1}_n + \beta a$ and $L \boldsymbol{1}_n = \alpha \boldsymbol{1}_n + \beta b$.
	\end{claim}
	\begin{claimproof}
		For legibility, we drop the index~$n$.
		By definition in \cref{prop:degree}, 
		\begin{align*}
			M\boldsymbol{1} & = \lambda \left< v_0 , \boldsymbol{1} \right> v_0 + \left< v_1, \boldsymbol{1} \right> v_1 \\
			&= \lambda \left< v_0 , \boldsymbol{1} \right> (\boldsymbol{1} + a) + \left< v_1, \boldsymbol{1} \right> \left(\boldsymbol{1} - a - \frac{\left< u_1, v_0 \right> }{\left< v_0, v_0 \right>} (\boldsymbol{1} - a) \right).
		\end{align*}
		The same equality holds when $M$ is replaced by $L$, $v_0$ by $w_0$, $v_1$ by $w_1$, and $a$ by $b$. 
		By \cref{cl:maps-principal}, $Uv_0 = w_0$ and $Uv_1 = w_1$.
		By construction, $U \boldsymbol{1} = U a^{\odot 0} = b^{\odot 0} = \boldsymbol{1}$.
		Since $U$ is orthogonal, the inner-products between $\boldsymbol{1}$, $v_0$, and $v_1$ are the same as the respective inner-products between $\boldsymbol{1}$, $w_0$, and $w_1$.
		This implies the claim.
	\end{claimproof}
	Let $S$ denote the star with $\ell$ leaves underlying the $1$-labelled graph $(\boldsymbol{A}\boldsymbol{1})^{\odot \ell}$. In contrary to the $1$-labelled star considered in \cref{lem:spaces}, this graph has its label at the central vertex. By \cref{cl:same-coeff}, as in \cref{eq:homstar},
	\begin{align*}
		\hom(S, G)
		= \soe (\boldsymbol{A}_G \boldsymbol{1}_G)^{\odot \ell}
		= N \soe (M\boldsymbol{1}_n)^{\odot \ell}
		= N \sum_{i = 0}^\ell \alpha^{\ell -i} \beta^{i} \soe(a^{\odot i})
		= N \sum_{i = 0}^\ell \alpha^{\ell-i} \beta^{i} \norm{a}_{i}^{i}.
	\end{align*}
	Replacing $a$ by $b$ yields an expression which equals $\hom(S, H)$. Hence,
	\[
	\hom(S, G) - \hom(S, H) = N \beta^\ell \left( \norm{a}_\ell^\ell - \norm{b}_\ell^\ell \right) \neq 0
	\]
	since $\beta \neq 0$ as $\allones_n$ is neither an eigenvector of $M$ nor of $L$, and $\ell$ was chosen to be minimal.
\end{proof}

Given \cref{lem:2graphs}, it remains to construct vectors $a, b \in \NN^n$ such that $\norm{a}_i = \norm{b}_i$ for all $i \leq 2d$ but $\norm{a}_\ell \neq \norm{b}_\ell$ for some $\ell$. 
By Newton's identities, this amounts to constructing vectors $a, b \in \NN^n$ satisfying the former condition such that their multisets of entries are not the same.
The fact that such pairs of vectors exist was established in the following number-theoretic result resolving the Prouhet--Tarry--Escott problem.

\begin{theorem}[{Prouhet--Thue--Morse Sequence~\cite{allouche_ubiquitous_1999}}] \label{thm:ptm}
	Let $d \geq 1$. 
	If the numbers in $\{0, \dots, 2^{d+1}-1\}$ are partitioned  into two sets $S_0 = \{x_1,\dots,x_n\}$ and $S_1 = \{y_1,\dots,y_n\}$ of size $n = 2^{d}$ according to the parity of their binary representations then the following equations are satisfied:
	\begin{align*}
		x_1 + \cdots + x_n &= y_1 + \cdots + y_n, \\ 
		x_1^2 + \cdots + x_n^2 &= y_1^2 + \cdots + y_n^2, \\
		\cdots & \quad \cdots \\   
		x_1^d + \cdots + x_n^d &= y_1^d + \cdots + y_n^d.
	\end{align*}
\end{theorem}

\begin{proof}[Proof of \cref{thm:degreeGH}]
	Given $d \geq 1$, form sets $S_0 = \{x_1, \dots, x_n\}$ and $S_1 = \{y_1, \dots, y_n\}$ as in \cref{thm:ptm}, each of size $n = 2^{2d} > d+1$. Let $a$ and $b$ be vectors of length $n$ formed by ordering the elements of $S_0$ and $S_1$, respectively, in an arbitrary fashion.
	
	Then, by construction, $\norm{a}_i^i = \sum_{j=1}^n x_j^i = \sum_{j=1}^n y_j^i = \norm{b}_i^i$ for all $1 \leq i \leq 2d$ but $\norm{a}_\ell \neq \norm{b}_\ell$ for some $\ell$ by Newton's identities since $S_0 \neq S_1$. Consequently, \cref{lem:2graphs} yields the two desired graphs.
\end{proof}

\section{Conclusion}

Using a the correspondence between (bi)labelled graphs and homomorphism tensors, we reprove known results in a unified way and derive new characterisations of homomorphism indistinguishability over bounded degree trees, graphs of bounded treedepth, graphs of bounded cyclewidth, and graphs of bounded pathwidth. The latter answers an open question from \cite{dell_lovasz_2018}.
The algebraic core of our results are two new variants of a classical theorem by Specht and Wiegmann.

Homomorphism indistinguishabilities over various graph classes can be viewed as similarity measures for graphs, and our new results as well as many previous results show that these are natural and robust. 
Yet homomorphism indistinguishability only yields equivalence relations, or families of equivalence relations, and not a ``quantitative'' distance measure. For many applications of graph similarity, such quantitative measures are needed. Interestingly, we can derive distance measure both from homomorphism indistinguishability and from the equational characterisations we study here. For a class $\mathcal F$ of graphs, we can consider the \emph{homomorphism embedding} that maps graphs $G$ to the vector in $\mathbb R^{\mathcal F}$ whose entries are the numbers  $\hom(F,G)$ for graphs $F\in\mathcal F$. Then a norm on the space $\mathbb R^{\mathcal F}$ induces a graph (pseudo)metric. Such metrics give a generic family of graph kernels (see \cite{Grohe-x2vec}). On the equational side, a notion like fractional isomorphism induces a (pseudo)metric on graphs where the distance between graphs $G$ and $H$ is $\min_X\left\|X \boldsymbol{A}_G-\boldsymbol{A}_HX\right\|$, where $X$ ranges over all doubly stochastic matrices. It is a very interesting question whether the correspondence between the equivalence relations for homomorphism indistinguishability and feasibility of the systems of equations can be extended to the associated metrics. In the special case of isomorphism and homomorphism indistinguishability over all graphs, the theory of graph limits provides some answers~\cite{lovasz_large_2012}. This has recently been extended to fractional isomorphism and homomorphism indistinguishability over trees~\cite{Boker21}.

\section*{Acknowledgements}

We thank Andrei Bulatov for many fruitful discussions on the foundations of theory developed here and its relation to the algebraic theory of valued constraint satisfaction problems.  Moreover, we thank Jan Böker for discussions about linear systems of equations and bilabelled graphs. 
Finally, we thank the anonymous reviewers for suggestions for improvement.

\providecommand{\bysame}{\leavevmode\hbox to3em{\hrulefill}\thinspace}
\providecommand{\MR}{\relax\ifhmode\unskip\space\fi MR }
\providecommand{\MRhref}[2]{\href{http://www.ams.org/mathscinet-getitem?mr=#1}{#2}
}
\providecommand{\href}[2]{#2}

\bibliographystyle{amsplain}

\begin{aicauthors}
	\begin{authorinfo}[magr]
		Martin Grohe \orcidlink{0000-0002-0292-9142}\\
		RWTH Aachen University\\
		Ahornstra\ss{}e 55\\
		52074 Aachen\\
		Germany \\
		grohe\imageat informatik\imagedot rwth-aachen\imagedot de
	\end{authorinfo}
	\begin{authorinfo}[gara]
		Gaurav Rattan \orcidlink{0000-0002-5095-860X}\\
		University of Twente\\
		Zilverling 4029\\
		Hallenweg 19\\
		Enschede\\
		7522 NH\\
		The Netherlands\\
		g\imagedot rattan\imageat utwente\imagedot com
	\end{authorinfo}
	\begin{authorinfo}[tise]
		Tim Seppelt \orcidlink{0000-0002-6447-0568}\\
		IT-Universitetet i K\o{}benhavn\\
		Rued Langgaards Vej 7\\
		2300 K\o{}benhavn S\\
		Denmark\\
		tise\imageat itu\imagedot dk
	\end{authorinfo}
\end{aicauthors}

\end{document}